\documentclass{amsart}

\newtheorem{theorem}{Theorem}[section]
\newtheorem{lemma}[theorem]{Lemma}

\theoremstyle{definition}

\newtheorem{assumption}[theorem]{Assumption}

\theoremstyle{remark}
\newtheorem{remark}[theorem]{Remark}

\numberwithin{equation}{section}

\usepackage[utf8]{inputenc} 

\usepackage{amsmath}
\usepackage{amstext}
\usepackage{amsfonts}
\usepackage{amsthm}
\usepackage{comment}
\usepackage{amssymb}
\usepackage{enumitem}
\usepackage[a4paper, total={6in, 8in}]{geometry}

\usepackage{graphicx}
\usepackage{dsfont}
\usepackage{nicefrac}

\usepackage{hyperref}

\usepackage[nocompress]{cite}
\hypersetup{colorlinks}
\usepackage{glossaries}

\renewcommand{\(}{\left(}
\renewcommand{\)}{\right)}
\newcommand{\lno}{\left\|}
\newcommand{\rno}{\right\|}

\newcommand{\Ltwon}{\rno_{H}}
\newcommand{\Loptwon}{\rno_{\mathcal{L}(H)}}
\newcommand{\Lophtwon}{ P_h\rno_{\mathcal{L}(H)}}
\newcommand{\Lston}{\rno_{L^p(\Omega;H)}}

\renewcommand{\P}{\mathbb{P}}
\newcommand{\D}{\mathcal{D}}
\newcommand{\Ltwonex}{\rno_{L^2(\D)}}

\newcommand{\N}{\mathbb{N}}
\newcommand{\R}{\mathbb{R}}
\newcommand{\E}{\mathbb{E}}

\newcommand{\F}{\mathcal{F}}
\newcommand{\coord}{(t,\mathbf{x})}
\newcommand{\half}{\frac{1}{2}}

\newcommand{\diff}[1]{\,\mathrm{d}#1}
\newcommand{\e}{\mathrm{e}}
\newcommand{\orderPoly}{1}
\newcommand{\verttt}{{\vert\kern-0.25ex\vert\kern-0.25ex\vert}}
\newcommand{\vertttb}{{\big\vert\kern-0.25ex\big\vert\kern-0.25ex\big\vert}}
\newcommand{\verttttB}{{\Big\vert\kern-0.25ex\Big\vert\kern-0.25ex\Big\vert}}
\newcommand{\constChi}{C_{\chi}}

\DeclareMathOperator{\dom}{dom}

\newcommand{\inner}[3][]{\left( #2 , #3 \right)_{#1}}

\renewcommand{\L}{\mathcal{L}}

\begin{document}
	
	\title[A full space time splitting framework for semi-linear SPDEs]{Error bounds for full space-time splitting discretizations of semi-linear SPDEs -- with a focus on dG domain decompositions}

	\author{Monika Eisenmann}
	\author{Eskil Hansen}
	\author{Marvin Jans}
	\email{monika.eisenmann@math.lth.se, eskil.hansen@math.lth.se, marvin.jans@math.lth.se}
	\thanks{The first and the third author were supported in part by the Swedish Research Council under the grant 2023-03930, eSSENCE: The e-Science Collaboration and the Crafoord foundation and the second author by the Swedish Research Council under the grant 2023–04862.
	}
	
	
	\address{Centre for Mathematical Sciences\\
		Lund University\\
		P.O.\ Box 118\\
		221 00 Lund, Sweden}
	
	
	\date{\today}
	
	


\begin{abstract}
	We consider a fully discretized numerical scheme for parabolic stochastic partial differential equations with multiplicative noise. Our abstract framework can be applied to formulate a non-iterative domain decomposition approach. Such methods can help to parallelize the code and utilize distributed hardware. The domain decomposition is integrated via the Douglas--Rachford splitting scheme, in which each split operator acts on a single part of the domain. For an efficient space discretization of the underlying equation, we chose the discontinuous Galerkin method as it suits the parallelization strategy well. For this fully discretized scheme, we provide a strong space-time convergence result under easily verifiable stability assumptions. We conclude the manuscript with numerical experiments validating our theoretical findings.
\end{abstract}

\maketitle

\section{Introduction} \label{sec:intro}
\allowdisplaybreaks
%
%
Consider a class of stochastic partial differential equations (SPDEs) with multiplicative noise, which take the form:
\begin{equation} \label{eq:SPDE}
	\begin{cases}
		\diff{X(t)} =\left[-AX(t)+f(t,X(t))\right]\diff{t}+B(t,X(t)) \diff{W(t)}, \quad t \in (0,t_f],\\
		X(0) = X_0.
	\end{cases}
\end{equation}
Here, $X(t)$ evolves in a real Hilbert space $H$ up to a finite time $t_f$. In this class of equations, $A$ is a linear, typically unbounded operator on the Hilbert space $H$, while the drift term $f(t,X(t))$ and the diffusion term $B(t,X(t))$ are possibly nonlinear but are assumed to be Lipschitz continuous with respect to $X(t)$.
%
%
Examples of semi-linear SPDEs include phase-field models, the Nagumo equation, and fluid flow problems (see \cite{lord_powell_shardlow_2014} for comparison). In these models, the noise can represent small-scale structures arising from thermal fluctuations, which are absent in deterministic models, and can account for the variability in wave speed or stochastic forcing. 

Given the  relevance of such stochastic models, studying efficient numerical approximation methods for SPDEs is crucial. Although these types of equations have gained increased attention in recent years, their numerical approximation still lags behind that of their deterministic counterparts. Our goal is to address this gap. Since stochastic equations typically exhibit lower regularity due to the presence of noise, we do not focus on high-order methods, as they are unlikely to provide significant benefits in such low-regularity settings. Instead, we seek to establish a theoretical foundation for an efficient numerical method that can be parallelized.

Domain decomposition methods are a common choice when designing parallel numerical methods for deterministic partial differential equations. For a general introduction, we refer the reader to \cite{Mathew.2008, DoleanJolivetNataf.2015,ToselliWidlund.2005, QuarteroniValli.1999}. These techniques partition the spatial domain into several subdomains, which enables parallel computations provided appropriate communication between the subdomains. There are various strategies for this communication. The standard approach involves iterative methods, in which the problems in the subdomains are solved sequentially, allowing information exchange between them. However, while such methods can yield accurate solutions, the iterative procedure incurs additional computational costs. To avoid the overhead of iterative schemes, we intend to incorporate the decomposition directly into the time integration using an operator splitting method. Examples of earlier deterministic studies of such methods can be found in \cite{EisenmannHansen.2018, EisenmannHansen.2022, ES.2024, HansenHenningsson.2017, ArrarasEtAl.2017, HansenHenningsson.2016, Mathew.1998, Vabishchevich.2008}. 

In the context of stochastic equations, less research has been conducted on domain decomposition methods. A few works have explored this direction within the realm of random differential equations, examples are \cite{ChenEtAl.2024, MuZhang.2019, BuckwarEtAl.2024, Prohl.2012, Ji.2024}. Our setting is related to \cite{Prohl.2012,BuckwarEtAl.2024} who consider the stochastic heat equation and a temporal semi-discretization, respectively. In contrast, we strive to derive a new abstract framework that requires fewer spatial regularity assumptions and that is applicable to the full space-time error analysis of non-iterative domain decomposition methods. The starting point for this framework will be elements from the backward Euler analysis for SPDEs in~\cite{lord_powell_shardlow_2014} and the deterministic space-time error analyses in \cite{HansenHenningsson.2016,HochbruckKoehler.2022}. 

For the time discretization, we will employ an implicit method with respect to the unbounded operator $A$ for stability reasons. As a result, solving implicit equations at each step of the time discretization is necessary. For $N \in \N$, a step size $\tau = t_f/N$, we define the discretization $X^{n}_{\tau}$ of the exact solution $X$ at the time point $n \tau$. More precisely, using the semi-implicit Euler method as a starting point with the initial condition $X^0_{\tau} = X_0$, the discretization takes on the following form
\begin{equation*}
	(I+\tau A) X^n_{\tau} 
	= X^{n-1}_{\tau} + \tau  f(t_{n-1},X^{n-1}_{\tau}) + B(t_{n-1},X^{n-1}_{\tau}) (W(t_n)-W(t_{n-1}))
\end{equation*}
for $n \in \{1,\dots,N\}$. To expedite the solution of these implicit equations, we propose the usage of a variant of the Douglas--Rachford splitting method \cite{DouglasRachford.1956}. Here, the operator $A$ is decomposed into two operators $A = A_1 + A_2$, where $A_\ell$ acts on one of the spatial subdomains, and the backward Euler step $(I+\tau A)^{-1}$ is approximated as
\begin{equation*}
	(I+\tau A)^{-1}
	= (I+\tau (A_1 + A_2))^{-1}
	\approx (I+\tau A_{2})^{-1} (I+\tau A_{1})^{-1} (I+\tau^2A_{1}A_{2} ),
\end{equation*}
The Douglas--Rachford splitting then reads as 
\begin{equation*}
	(I+\tau A_{1}) (I+\tau A_{2}) X^n_{\tau} 
	= (I+\tau^2A_{1}A_{2} ) X^{n-1}_{\tau} 
	+ \tau  f(t_{n-1},X^{n-1}_{\tau}) + B(t_{n-1},X^{n-1}_{\tau}) (W(t_n)-W(t_{n-1})),
\end{equation*}
where $n \in \{1,\dots,N\}$. Note that an additional modification of the first time step will be employed in the analysis, but this detail is omitted here for the sake of simplicity. The reason for choosing the Douglas--Rachford splitting scheme, instead of some other first-order scheme, is that it has an advantageous error structure, as described in \cite{HansenHenningsson.2017}. 

Next, we introduce a spatial discretization. Let $V_h = \text{span} \{\varphi_{1}, \dots, \varphi_{M(h)}  \}$ be a finite dimensional subspace of $H$ and consider the fully discretized solution $X^n_{h, \tau} = \sum_{i=1}^{M(h)} \alpha_{i}^n \varphi_{i}$ given by the algebraic system
\begin{align*}
	\mathbf{M} \mathbb{\alpha}^n
	&= (I+\tau A_{h,2})^{-1} (I+\tau A_{h,1})^{-1} (I+\tau^2A_{h, 1}A_{h, 2} ) \mathbf{M} \mathbb{\alpha}^{n-1}\\
	&\quad + (I+\tau A_{h,2})^{-1} (I+\tau A_{h,1})^{-1} \big(\tau f_h(t_{n-1},X^{n-1}_{h, \tau}) + B_h(t_{n-1},X^{n-1}_{h,\tau}) (W(t_n)-W(t_{n-1}))\big).
\end{align*}
Here, the mass matrix is given by $(\mathbf{M})_{i,j}  = \inner[H]{\varphi_{i}}{\varphi_{j}}$ and $A_{h,\ell}$ is the discrete counterpart of $A_\ell$. 

An issue for parallelization arises when the mass matrix $\mathbf{M}$ does not have a block structure. This block structure is even missing for the standard finite element method, as $\mathbf{M}$ then becomes tridiagonal in the one dimensional case. This lack of locality hinders efficient parallelization. While mass lumping can be a useful tool to diagonalize the mass matrix, this typically requires high regularity assumptions on the solution that we cannot always expect in our setting. Therefore, we opt for a spatial discretization that yields a more suitable mass matrix. Examples include discontinuous Galerkin (dG) and spectral Galerkin methods. In this study, we will concentrate on a symmetric interior penalty method as presented in \cite[Chapter 4.2]{DiPietro2012}. Other studies relating to dG methods applied to parabolic SPDEs include~\cite{Li2021,Pazner2019,Yang2023}.

%
%
In conclusion, the goals of this work are:
\begin{itemize}
	\item Derive an abstract framework for an operator splitting method that enables a full space-time error analysis for stochastic equations.
	\item Design a parallelizable method based on dG that can efficiently utilize distributed hardware.
	\item Prove state-of-the-art error bounds for the parallelizable method under easily verifiable assumptions in the new abstract framework.
\end{itemize}

%
%
The paper is organized as follows. In Section~\ref{sec:Prob_description}, we provide all the necessary assumptions on the data of \eqref{eq:SPDE}, state the solution concept, and an existence and regularity result for such a solution. With this in mind, we explain all the details for the full discretization in Section~\ref{sec:discretization}. This includes a general spatial discretization framework and the temporal discretization with a variation of the Douglas--Rachford splitting scheme. For this general framework, we then provide explicit error bounds in Section~\ref{sec:Conv_analysis}. For the error bounds, we state some needed auxiliary results, which we then combine to prove our main result in Theorems~\ref{thm:result} and \ref{thm:result_selfadjoint}. Our main application of the theoretical framework is then explained in detail in Section~\ref{sec:dG}. We verify that dG fits in our general spatial discretization framework and state a splitting of the operator needed for the Douglas--Rachford splitting, which is based on a domain decomposition approach. This setting is then further considered in our numerical example in Section~\ref{sec:Num_ex} where we confirm our findings through tests. Finally, needed auxiliary results are summarized in Appendices~\ref{appendix:basic_results}--\ref{appendix:higher_conv}.

\section{Problem description} \label{sec:Prob_description}

In the following section, we introduce the necessary notation and assumptions needed for Equation~\eqref{eq:SPDE}. We abbreviate $\R^+ := (0,\infty)$, $\R_0^+ := [0,\infty)$ and $\R_0^- := (- \infty,0]$. Further, we assume that $t_f \in \R^+$ is a given finite end time and $(H,\inner[H]{\cdot}{\cdot}, \|\cdot\|_H)$ is a real Hilbert space. The underlying filtered probability space for the stochastic equation is denoted by $(\Omega, \F, \{\F_t\}_{t\in [0,t_f]}, \P)$, which satisfies the usual conditions.
In the following, the constant $C \in \R^+$ is generic and can change from line to line, but it is always independent of the temporal and spatial discretization parameters, $\tau$ and $h$, respectively. 
The possibly unbounded operator $A$ in the parabolic Equation~\eqref{eq:SPDE} fulfills the following assumption.

\begin{assumption}\label{ass:A}
	Let the linear operators $A \colon \dom(A) \subset H \to H$ and $A_{\ell} \colon \dom(A_{\ell}) \subset H \to H$, $\ell \in \{1,2\}$, be given such that $A = A_1 + A_2$ on $\dom(A_1)\cap \dom(A_2) \subseteq \dom(A)$. Moreover, the operators fulfill the following criteria.
	\begin{enumerate}[label={(\alph*)}, ref={\ref{ass:A}~(\alph*)}]
		\item \label{ass:A_sect} The operator $A$ is a densely defined, positive operator on $H$. Furthermore, the operator $A$ is sectorial. That is there exists $\varphi \in (0,\frac{\pi}{2})$ such that the sector $S_{\varphi}=\{\lambda\in\mathbb{C} :\varphi < |\arg(\lambda)|\leq \pi \}$ and zero lie in the resolvent set $\rho(A)$. More precisely, for all $\lambda \in S_{\varphi}$, it follows that
		\begin{equation*}
			\|(A - \lambda I)^{-1}\|_{\L(H_{\mathbb{C}})} \leq \frac{C}{|\lambda|},
		\end{equation*}
		where $H_{\mathbb{C}}$ is the complexification $H + i H$ of the real space $H$.
		\item \label{ass:op_exch} For $\ell \in \{1,2\}$, the operator $A_{\ell}A^{-1} \colon H \to H$ is a well-defined, bounded operator.
	\end{enumerate}
\end{assumption}
Note that in this paper, we follow the same convention as in \cite[Definition~1.75]{lord_powell_shardlow_2014} for the definitions of (non-)positive and (non-)negative operators. With the previous assumption in mind, we observe that since $A$ is sectorial, there exists $\lambda_0 \in \R_0^-$ such that the range of $A - \lambda_0 I$ is $H$. Together with the fact that $-A$ is a positive (therefore also dissipative) operator, the range condition implies that the semigroup denoted by $\e^{-tA}$, $t \in \R_0^+$, is a semigroup of contractions on $H$, compare \cite[Chapter~1, Theorem~4.3]{Pazy1983}. Since $A$ is sectorial, the semigroup $\e^{-tA}$ is also analytic, compare \cite[Chapter~2, Theorem 5.2]{Pazy1983}. Assuming that $-A$ generates an analytic semigroup and $0 \in \rho(A)$, we obtain some useful bounds. For $\zeta \in \R_{0}^+$, it follows that
\begin{equation}\label{eq:ana_semi1}
	\| A^{\zeta} \e^{-tA} \|_{\L(H)} \leq C t^{-\zeta} \quad \text{for all } t \in \R^+,
\end{equation}
and for $\zeta\in [0,1]$ it holds that 
\begin{equation}\label{eq:ana_semi2}
	\| A^{-\zeta}(I-\e^{-tA}) \|_{\L(H)} \leq Ct^{\zeta} \quad \text{for all } t \in \R^+.
\end{equation}
For the proof of these results, we refer to \cite[Chapter~2,~Theorem~6.13]{Pazy1983}.
A definition for the fractional operators $A^{\zeta}$, $\zeta \in \R^+$, used in the bounds can be found in \cite[Chapter~2.6]{Pazy1983}.
In the following assumptions, we choose parameters $\theta$ with corresponding indices that are connected to the regularity of the coefficients. These parameters are fixed throughout the paper. Note that we allow for different parameters in the assumptions to be able to be precise about the needed regularity. If we set every $\theta_{\ast}$ to be $\frac{1}{2}-$, in Theorem~\ref{thm:result} we obtain that the error of the method is essentially $\tau^{\frac{1}{2}-} + h^{2-}$. Under additional assumptions on the split operators, we can see in Theorem~\ref{thm:result_selfadjoint} that some of these parameters can be set to $0$.
Whenever we use $\zeta \in \R_0^+$, this should be interpreted as a more general variable parameter that changes depending on the context.

\begin{assumption}\label{ass:X0}
	For a fixed $p \in [2,\infty)$ and $\theta_{X_0} \in[0,1)$, the initial condition $X_0 \colon \Omega \to H$ is a $\mathcal{F}_0$-measurable random variable such that $\| A^{\theta_{X_0}} X_0 \|_{L^p(\Omega;H)} \leq C$.
\end{assumption}

Next, we state the exact assumptions needed for the nonlinear perturbation $f$ of Equation~\eqref{eq:SPDE}.

\begin{assumption}\label{ass:f}
	Let $f\colon \R \times H\rightarrow H$ and $\theta_{f} \in[0,\half)$ be given. The following conditions are fulfilled.
	\begin{enumerate}[label={(\alph*)}, ref={\ref{ass:f}~(\alph*)}]
		\item \label{ass:f_reg} For all $v\in H$ and $w\in \dom(A^{\theta_{f}})$, it follows that
		\begin{equation*}
			\| f(t,v) \|_H \leq C\big(1+ \| v \|_H \big)
			\quad \text{and} \quad
			\| A^{\theta_{f}}f(t,w) \|_H \leq C\big(1+ \| A^{\theta_{f}}w \|_H \big).
		\end{equation*}
		\item \label{ass:f_hold} For all $v,w\in H$ and $s,t\in[0,t_f]$, it follows that
		\begin{equation*}
			\lno f(s,v)-f(t,w)\Ltwon\leq C \big(|s-t|^{\frac{1}{2}}+\lno v-w\Ltwon\big).
		\end{equation*}		
	\end{enumerate}
\end{assumption}

The next step is to define the setting for the stochastic part of Equation~\eqref{eq:SPDE}. First, we begin with the  Wiener noise $W$.

\begin{assumption}\label{ass:W}
	Let the operator $Q$ be a non-negative definite symmetric operator  of trace class on a Hilbert space $(U, \inner[U]{\cdot}{\cdot}, \|\cdot\|_U)$. Moreover, let $\{W(t)\}_{t\in [0,t_f]}$ be a given $Q$-Wiener process that is adapted  to the filtration $\{\F_t\}_{t\in [0,t_f]}$. 
\end{assumption}

For two Hilbert spaces $H_1$ and $H_2$, we denote the Hilbert space of Hilbert--Schmidt operators by $\mathcal{HS}(H_1,H_2)$. This space is equipped with the norm $\|E\|_{\mathcal{HS}(H_1,H_2)} = \sqrt{ \text{tr} (E^*E)}$, where $\text{tr}$ is the trace operator.
In order to state the assumptions on the noise term, we first recall the definition of the Cameron--Martin space, compare \cite[Section~10.3]{lord_powell_shardlow_2014}. For $U_0=Q^{\half}U$, the Cameron--Martin space $L^0_2$ is given by the set of linear operators $E \colon U_0 \rightarrow H$ that fulfill 
\begin{equation*}
	\| E \|_{L^0_2} := \| EQ^{\half} \|_{\mathcal{H}\mathcal{S}(U,H)} = \| E \|_{\mathcal{H}\mathcal{S}(U_0,H)} <\infty.
\end{equation*}

\begin{assumption}\label{ass:B}
	Let $B\colon [0,t_f] \times H\rightarrow L^0_2$ and $\theta_{B} \in[0,\half)$ be given. The following conditions are fulfilled.
	\begin{enumerate}[label={(\alph*)}, ref={\ref{ass:B}~(\alph*)}]
		\item \label{ass:B_reg} For all $v\in H$ and $w\in \dom(A^{\theta_{B}})$, it follows that
		\begin{equation}\label{eq:HS_B}
			\| B(t,v) \|_{L^0_2} \leq C\big(1+ \| v \|_H \big), 
			\quad 
			\| A^{\theta_{B}}B(t,w) \|_{L^0_2} \leq C\big(1+ \| A^{\theta_{B}}w \|_H \big)
		\end{equation}
		and for all $u\in \dom(A^{\theta_{X_0}})$
		\begin{equation}\label{eq:boundedOp_B}
			\| B(t,u) \|_{\L(U,H)} \leq C\big(1+ \| A^{\theta_{X_0}} u \|_H \big),
		\end{equation}
		where $\theta_{X_0}$ is chosen as in Assumption~\ref{ass:X0}.
		\item \label{ass:B_hold} For all $v,w\in H$ and $s,t\in[0,t_f]$, it follows that
		\begin{equation*}
			\| B(s,v)-B(t,w) \|_{L^0_2}\leq C \big(|s-t|^{\frac{1}{2}}+\lno v-w\Ltwon \big).
		\end{equation*}
	\end{enumerate}	
\end{assumption}

In this paper, we work with the mild solution of Equation~\eqref{eq:SPDE}. Under the assumptions stated above, let us recall the solution concept. A predictable $H$-valued process $\{X(t)\}_{t\in[0,t_f]}$ is called a mild solution if 
\begin{equation*}
	\P\Big(\int_0^{t_f}\lno X(t)\Ltwon^2\diff{t}<+\infty \Big) = 1
\end{equation*}
and if for all $t\in[0,t_f]$, it holds $\P$-a.s.~that
\begin{equation}\label{eq:SPDE_int}
	X(t)
	= \e^{-tA}X_0 + \int_0^t \e^{-(t-s)A}f(s,X(s))\diff{s} + \int_0^t \e^{-(t-s)A} B(s,X(s))\diff{W(s)}.
\end{equation}
Moreover, we recall that for $\Phi \in L^2(0,t_f; L^p(\Omega;L_2^0))$ the following Burkholder--Davis--Gundy inequality 
\begin{equation}\label{eq:burk}
	\Big(\E\Big[ \sup_{t\in (0,t_f)} \Big\| \int_0^{t} \Phi(s) \diff{W(s)} \Big\|_H^p \Big]\Big)^{\frac{1}{p}}
	\leq C \Big( 
	\int_0^{t_f} \| \Phi(s)\|_{L^p(\Omega;L_2^0) }^{2} \diff{s} \Big)^{\frac{1}{2}}
\end{equation}
holds. For a proof, we refer to \cite[Theorem 6.1.2]{Liu2015}.
With our assumptions, we have the following existence and regularity result.

\begin{theorem} \label{thm:ex_reg_exact_sol}
	Let Assumptions~\ref{ass:A}--\ref{ass:B} be fulfilled. Then there exists a unique mild solution to Equation~\eqref{eq:SPDE} up to modifications. Assuming that $\theta_{X_0} \in[\theta_{B},\theta_{B} + \frac{1}{2})$, the solution fulfills the following two regularity bounds
	\begin{equation} \label{eq:reg_space}
		\sup_{t\in (0,t_f) } \| A^{\theta_{X_0}} X(t) \|_{L^p(\Omega;H)}
		\leq C
	\end{equation}
	and
	\begin{equation} \label{eq:reg_hoelder}
		\sup_{s,t \in (0,t_f), s \neq t} \frac{ \| X(t)-X(s) \|_{L^p(\Omega;H)} }{|t-s|^{\min(\theta_{X_0}, \frac{1}{2})}} 
		\leq C.
	\end{equation}
\end{theorem}

\begin{proof}
	For a proof, we refer to \cite[Theorem~1]{JENTZEN2012114} when choosing $\alpha = \theta_{B}$ and $r = 0$. Note that this proof only covers coefficients $f$ and $B$ that are independent of $t$. However, adding a time dependence as in Assumptions~\ref{ass:f} and \ref{ass:B} does not further complicate the proof. Compare also \cite[Section~2.5--2.6]{Kruse.2014}, where the coefficients are time dependent, but the regularity assumptions are slightly different.
\end{proof}

\section{Discretization} \label{sec:discretization}

In the coming section, we provide a full discretization scheme for our underlying stochastic Equation~\eqref{eq:SPDE}. We begin by stating the setting for the spatial discretization in Section~\ref{subsec:space_discretization}. With this setting in mind, we can then add a time-stepping method to obtain a fully discretized method in Section~\ref{subsec:time_discretization}. The final scheme is stated in \eqref{eq:scheme}.

\subsection{Spatial discretization} \label{subsec:space_discretization}

For the space discretization, we choose a set $\{V_h\}_{h\in I}$, $I\subset \R^+$, of finite-dimensional subspaces of $H$. Recall that the constant $C$ is always assumed to be independent of the spatial discretization parameter $h$. 
Since $V_h$ is a subset of the Hilbert space $H$, we can use both the inner product $(\cdot,\cdot)_H$ and its corresponding norm $\|\cdot\|_H$ on the space. 

\begin{assumption}\label{ass:proj_err} 
	The bounded projection operator $P_h\colon H\rightarrow V_h$ fulfills $\| (I-P_h)v \|_H \leq Ch^2 \| Av \|_H$ for all $v\in \dom(A)$.
\end{assumption}

\begin{remark}\label{remark:proj_err} 
	Due to Assumption~\ref{ass:proj_err}, we obtain
	\begin{equation*}
		\| (I-P_h) A^{-1} \|_{\L(H)} \leq C h^2
		\quad \text{and} \quad
		\| I-P_h \|_{\L(H)} \leq C.
	\end{equation*}
	Together with the interpolation result from Lemma~\ref{lem:interA}, for $\zeta \in (0,1)$, we then find that
	\begin{equation*}
		\| (I-P_h) A^{-\zeta} \|_{\L(H)} 
		\leq C h^{2\zeta}
		\quad \text{or} \quad 
		\| (I-P_h) v \|_{H} \leq C h^{2\zeta} \| A^{\zeta} v\|_H
	\end{equation*}
	for all $v \in \dom(A^{\zeta})$.
\end{remark}

On $V_h$, we state approximating operators to the continuous operators introduced in the previous section. We begin with an approximation of the (unbounded) operator $A$ and its split operators $A_1$ and $A_2$. We need the discretization of $A$ to be positive and bounded with constants independent of $h$ w.r.t.~a suitable norm in $V_h$. To state the needed assumptions for this, we introduce a second norm $\| \cdot \|_{V_h}$ on $V_h$, which is induced by an inner product $\inner[V_h]{\cdot}{\cdot}$ and fulfills
\begin{equation} \label{eq:normV_h}
	\|v_h\|_H\leq C \|v_h\|_{V_h} \quad \text{for all } v_h\in V_h.
\end{equation}
The idea is that this additional norm is a discrete counterpart to the typical norm on the variational space $V:= \dom(A^{\frac{1}{2}})$. 

\begin{assumption}\label{ass:Ah}
	Let Assumptions~\ref{ass:A} and  \ref{ass:proj_err} be fulfilled and let the norm $\|\cdot \|_{V_h}$ be given as in \eqref{eq:normV_h}. Further, let	the operators $A_h, A_{h,1}, A_{h,2} \in \L(V_h)$ fulfill the following conditions.
	\begin{enumerate}[label={(\alph*)}]
		\item \label{ass:Ah_coercive}  The operator $A_h$ is strongly positive, i.e.
		\begin{equation*}
			(A_hv_h,v_h)_H\geq C\|v_h\|^2_{V_h} \quad \text{for all } v_h \in V_h.
		\end{equation*}
		\item \label{ass:Ah_bound} The operator $A_h$ fulfills the following boundedness condition
		\begin{equation*}
			\big| (A_hv_h,w_h)_H \big| \leq C\|v_h\|_{V_h}\|w_h\|_{V_h} \quad \text{for all } v_h,w_h\in V_h.
		\end{equation*}
		\item \label{ass:Ah_sum} The operators fulfill $A_h = A_{h,1} + A_{h,2}$ in $V_h$.
		\item \label{ass:Ah_pos} The operators $A_{h,1}$ and $A_{h,2}$ are non-negative operators in $V_h$ w.r.t.~$(\cdot,\cdot)_H$.
		\item \label{ass:Ah2} For $\ell \in \{1,2\}$, it holds that $\| A_{h,\ell} P_h \|_{\L(H)} \leq Ch^{-2}$.
		\item \label{ass:dissl_exch} For $\ell \in \{1,2\}$, it holds that
		\begin{equation*}
			\| (P_hA_{\ell}-A_{h,\ell}P_h) v \|_H \leq C \| Av \|_H \quad \text{ for all } v\in \dom(A).
		\end{equation*}
		\item \label{ass:dissa_exch} The operator $A_h$ fulfills that $\| A^{-1}-A_{h}^{-1}P_h \|_{\L(H)} \leq Ch^2$.
	\end{enumerate}	
\end{assumption}

In the following, we need the two bounds \eqref{eq:ana_semi1} and \eqref{eq:ana_semi2} also for the discretized operator $A_h$. First, we show that the operator $A_h$ is sectorial. On top of that, we need to verify that the constant $C$ in the inequalities \eqref{eq:ana_semi1} and \eqref{eq:ana_semi2} does not depend on $h$. To verify these properties, the spaces $X$ and $Z$ from \cite[Section~1.7.1]{Yagi2010} are chosen to be $(Z, \inner[Z]{\cdot}{\cdot}, \|\cdot\|_Z) = (V_h, \inner[V_h]{\cdot}{\cdot}, \|\cdot\|_{V_h})$ and $(X, \inner[X]{\cdot}{\cdot}, \|\cdot\|_X) = (V_h, \inner[H]{\cdot}{\cdot}, \|\cdot\|_H)$. This choice of space fits into the assumptions of \cite[Section~1.7.1]{Yagi2010}. Thus, with \cite[Theorem~2.1]{Yagi2010} and Assumptions~\ref{ass:Ah}~\ref{ass:Ah_coercive}--\ref{ass:Ah_bound} we obtain that $A_h$ is sectorial with an angle $\varphi_h \in (\frac{\pi}{2},\pi)$. It then follows that $A_h$ is a densely defined, positive, sectorial operator on $V_h$. Without loss of generality, we can always pick $\varphi_h$ and $\varphi$ from Assumption~\ref{ass:A_sect} such that both $A$ and $A_h$ are sectorial with the same angle that we refer to as $\varphi$ in the following. The constant to bound $\|(A_h P_h - \lambda I)^{-1} \|_{\L(H)}$ is then independent of $h$ for $\lambda \in S_{\varphi}$. 
With this in mind, we can now state the bounds \eqref{eq:ana_semi1} and \eqref{eq:ana_semi2} for $A_h$, where $C$ is independent of $h$, compare \cite[Chapter~2, Theorem 6.13]{Pazy1983} or \cite[Section~2.7.7]{Yagi2010} for a proof. More precisely, we have that for $\zeta\in \R_{0}^+$
\begin{equation}\label{eq:ana_semi_h1}
	\| A_h^{\zeta} \e^{-t A_h} P_h\|_{\L(H)} \leq C t^{-\zeta} \quad \text{for all } t \in \R^+,
\end{equation}
and for $\zeta\in[0,1]$
\begin{equation}\label{eq:ana_semi_h2}
	\| A_h^{-\zeta} (I - \e^{-t A_h}) P_h \|_{\L(H)} \leq Ct^{\zeta} \quad \text{for all } t \in \R^+.
\end{equation}
For the discretizations of $f$ and $B$, we add suitable projections to finite-dimensional subspaces. This can be done as follows for the function $f$
\begin{equation} \label{eq:def_fh}
	f_h \colon [0,t_f] \times H \to V_h \colon  (t,v)\mapsto P_hf(t,v).
\end{equation}
Due to the projection properties, this function can easily be compared to the original function. For the discretized version of $B$, we need some additional assumptions that we summarize in the following.

\begin{assumption}\label{ass:Proj_error_U}
	Let Assumption~\ref{ass:W} be fulfilled, let $\theta_{U} \in[0,\half)$ given and
	let $\{e_k\}_{k\in \N}$ be the eigenfunctions with their corresponding eigenvalues $\{q_k\}_{k\in \N}$ of $Q$ that build an orthonormal basis of $U$. For $r \in \R_0^+$ and $\varepsilon \in \R^+$, the eigenvalues fulfill $q_k = \mathcal{O}(k^{-(2r+1+\varepsilon)})$. Further, let $U_h$ be a finite-dimensional subspace of $U$ given by $U_h = \text{span} \{e_1, \dots, e_{N_U}\}$ such that $(N_U+1)^{-r} \leq C h^{2\theta_{U}+1}$. The orthogonal projection on $U_h$ is denoted by 
	\begin{equation}
		P_U \colon U \to U_h.
	\end{equation}
\end{assumption}

With this in mind, we are now prepared to state the discretized version of the noise operator $B$
\begin{equation}\label{eq:def_Bh}
	B_h \colon [0,t_f] \times H \to \mathcal{HS}(U_0,V_h)\colon   (t,v) \mapsto  P_h B(t,v) P_U.
\end{equation}

\begin{lemma}\label{lem:fh_Bh}
	Let Assumptions~\ref{ass:A}, \ref{ass:f}, \ref{ass:B}, \ref{ass:proj_err}, and \ref{ass:Proj_error_U} be fulfilled and let $f_h$ and $B_h$ be given as in \eqref{eq:def_fh} and \eqref{eq:def_Bh}, respectively. For all $u \in \dom(A^{\theta_{X_0}})$, $v,w \in H$ and $s, t\in[0,t_f]$, it follows that
	\begin{enumerate}[label={(\roman*)}, ref={(\roman*)}]
		\item \label{lem:fh_hold} 
		$\| f_h(s,v)-f_h(t,w) \|_H \leq C \big( |s-t|^{\frac{1}{2}} + \|v-w\|_H \big)$;
		\item \label{lem:Bh_regP} 
		$\| P_hB(t,u)-B_h(t,u) \|_{L^0_2}\leq Ch^{2\theta_{U}+1} \big(1+\| A^{\theta_{X_0}} u\|_H\big)$;
		\item \label{lem:Bh_hold} 
		$\| B_h(s,v)-B_h(t,w) \|_{L^0_2} \leq C \big( |s-t|^{\frac{1}{2}} + \|v-w\|_H \big)$.
	\end{enumerate}
\end{lemma}

\begin{proof}
	First, we observe that \ref{lem:fh_hold} follows by an application of Assumption~\ref{ass:f_hold}
	\begin{equation*}
		\| f_h(s,v)-f_h(t,w) \|_H
		= \big\| P_h \big(f(s,v)-f(t,w)\big) \big\|_H
		\leq C\big(|s-t|^{\frac{1}{2}} + \| v-w\|_H\big).
	\end{equation*}
	For the proof of \ref{lem:Bh_regP}, we apply Assumption~\ref{ass:B_reg} and obtain
	\begin{align*}
		\| P_hB(t,u)-B_h(t,u) \|_{L^0_2}
		&= \| P_h B(t,u)( I - P_U) \|_{L^0_2}
		\leq \| B(t,u) \|_{\L(U,H)} \| (I - P_U) Q^{\frac{1}{2}}\|_{\mathcal{HS}(U,U)} \\
		&\leq C(1 + \|A^{\theta_{X_0}} u\|_H) \| (I - P_U) Q^{\frac{1}{2}} \|_{\mathcal{HS}(U,U)}.
	\end{align*}
	We then obtain the claimed result by combining the previous bound with 
	\begin{align*}
		\| (I - P_U) Q^{\frac{1}{2}} \|_{\mathcal{HS}(U,U)}
		&= \Big(\sum_{k = N_U +1}^{\infty} q_k\Big)^{\frac{1}{2}}
		\leq C \Big(\sum_{k = N_U +1}^{\infty} k^{-(2r +1 + \varepsilon)}\Big)^{\frac{1}{2}} \\
		&\leq (N_U+1)^{-r} \Big(\sum_{k = N_U +1}^{\infty} k^{-(1 + \varepsilon)} \Big)^{\frac{1}{2}} 
		\leq Ch^{2\theta_{U}+1},
	\end{align*}
	where we applied Assumption~\ref{ass:Proj_error_U}.
	The last remaining step is to prove \ref{lem:Bh_hold}. This can be done using Assumption~\ref{ass:B_hold} and the fact that the norm of a projection operator is less than one, which yields
	\newpage
	\begin{align*}
		\| B_h(s,v)-B_h(t,w) \|_{L^0_2} 
		&= \| P_h B(s,v) P_U - P_h B(t,w)P_U \|_{L^0_2}\\
		&\leq \| B(s,v) - B(t,w) \|_{L^0_2}
		\leq C \big( |s-t|^{\frac{1}{2}} + \| v - w \|_{H} \big).
	\end{align*}
\end{proof}

\subsection{Temporal discretization}\label{subsec:time_discretization}

The temporal discretization method that we add to the spatial discretization from the previous subsection is a Douglas--Rachford splitting method. In this method, we decompose the operator $A_h$ into two parts $A_{h,1}$ and $A_{h,2}$ with $A_h = A_{h,1} + A_{h,2}$. Here, we solve two sub-steps containing $A_{h,1}$ and $A_{h,2}$ instead of one step only containing $A_{h}$. 

First, let us fix the notation used on the temporal discretization in the paper. We choose an equidistant time grid with $t_n = n \tau$ and a step size  $\tau=\frac{t_f}{N}$ for $N\in\N$ and $n \in \{0, \dots,N\}$. We now state $X^n_{h,\tau}$ to approximate the exact solution $X(t_n)$ at a grid point $t_n$, $n \in \{1,\dots,N\}$,
\begin{equation}
	\label{eq:scheme}
	\begin{cases}
		X^0_{h,\tau} = P_hX_0,\\
		X^1_{h,\tau} = (I+\tau A_{h,2})^{-1} (I+\tau A_{h,1})^{-1} \big( X^0_{h,\tau} + \tau f_h(t_0,X^0_{h,\tau}) + B_h(t_0,X^0_{h,\tau}) W(t_1) \big),\\
		X^n_{h,\tau} = S_{h,\tau} X^{n-1}_{h,\tau} + (I+\tau A_{h,2})^{-1}(I+\tau A_{h,1})^{-1} \big( \tau  f_h(t_{n-1},X^{n-1}_{h,\tau})\\ 
		\hspace{2cm}+ B_h(t_{n-1},X^{n-1}_{h,\tau}) (W(t_n)-W(t_{n-1})) \big)
	\end{cases}
\end{equation}
where $n \in \{2,\dots,N\}$ and
\begin{equation}\label{eq:def_S}
	S_{h,\tau}=(I+\tau A_{h,2})^{-1}(I+\tau A_{h,1})^{-1} (I+\tau^2A_{h,1}A_{h,2} ).
\end{equation}
The inverse operators $(I+\tau A_{h,1})^{-1}$ and $(I+\tau A_{h,2})^{-1}$ are indeed well-defined. Due to the non-negativity of $A_{h,1}$ and $A_{h,2}$ and the fact that they act on finite-dimensional spaces, it follows that $-A_{h,1}$ and $-A_{h,2}$ generate semigroups of contraction, and therefore the inverses are well-defined. Moreover, both $\|(I+\tau A_{h,1})^{-1} P_h\|_{\L(H)}$ and $\|(I+\tau A_{h,2})^{-1} P_h\|_{\L(H)}$ are bounded by $1$, compare \cite[Chapter~1, Theorem~3.1]{Pazy1983}. 

Note that we approximate $\e^{-\tau A_h}$ by $(I+\tau A_{h,2})^{-1}(I+\tau A_{h,1})^{-1}$ in the first temporal step and by $S_{h,\tau}$ in all coming steps. The reason for this is that $S_{h,\tau}$ in the first step leads to a CFL-type condition in our analysis. Being able to consider the operator $(I+\tau A_{h,1})^{-1} (I+\tau^2A_{h,1}A_{h,2} ) (I+\tau A_{h,2})^{-1}$ instead, simplifies the analysis.
Thus, our approximation requires less regularity in the initial value compared to $S_{h,\tau}^{n}$ and offers a simple error recursion.

We can argue inductively that $X_{h, \tau}^n$ is $\F_{t_n}$-measurable for every $n \in \{1,\dots,N\}$. First, we observe that the initial value is $\F_{t_0} = \F_0$-measurable and an element of $L^p(\Omega;H)$ (Assumption~\ref{ass:X0}). Assuming that $X_{h, \tau}^{n-1}$ is $\F_{t_{n-1}}$-measurable, it follows that $X_{h, \tau}^n$ is $\F_{t_n}$-measurable as a composition of an $\F_{t_{n-1}}$-measurable function and the $\F_{t_{n}}$-measurable increment $W(t_n)-W(t_{n-1})$. Moreover, $X_{h, \tau}^n$ is an element of $L^p(\Omega;H)$ as the operators $S_{h,\tau}^n (I+\tau A_{h,2})^{-1}P_h $, $(I+\tau A_{h,2})^{-1}P_h$ and $(I+\tau A_{h,1})^{-1}P_h$ are bounded on $H$; compare Lemma~\ref{lem:Sne} below.

\begin{remark}
	When implementing the scheme, our method can be rewritten for more efficiency. Using the transformation $X^n_{h,\tau}=(I+\tau A_{h,2})^{-1} Y_{h,\tau}^n$, the method \eqref{eq:scheme} is equivalent to 
	\begin{equation*}
		\begin{cases}
			Y^1_{h,\tau} = (I+\tau A_{h,1})^{-1}X^0_{h,\tau}+(I+\tau A_{h,1})^{-1} \big( \tau f_h(0,X^0_{h,\tau}) + B_h(0,X^0_{h,\tau}) W(t_1) \big),\\
			Y^n_{h,\tau}
			= (I+\tau A_{h,1})^{-1} \big( (2(I+\tau A_{h,2})^{-1}-I )Y^{n-1}_{h,\tau} + \tau f_h(t_{n-1},X^{n-1}_{h,\tau})\\
			\qquad \qquad+ B_h(t_{n-1},X^{n-1}_{h,\tau}) (W(t_{n}) - W(t_{n-1}) ) \big)  +(I-(I+\tau A_{h,2})^{-1}) Y^{n-1}_{h,\tau}, 
			\quad n\in \{2,\dots,N\}.
		\end{cases}
	\end{equation*}
	When comparing the Douglas--Rachford splitting to the Lie splitting, the additional quadratic term $I + \tau^2 A_{h,1} A_{h,2}$ appears in the former, compare \eqref{eq:def_S}. 
	The advantage of this reformulation is that this quadratic term does not have to be evaluated, avoiding an extra matrix multiplication. This means that we do not need to evaluate more matrix operations compared to the simpler Lie splitting method. 
	In the semi-discrete setting, the quadratic term is problematic from a regularity point of view. Additionally, in a full discretization the product $A_{h,1} A_{h,2}$ is of the order $h^{-4}$. In the alternative formulation, such a term does not appear.
	
	To obtain the $X^n_{h,\tau}$-terms, we only need to calculate $(I+\tau A_{h,2})^{-1}Y_{h,\tau}^n$. Note since we compute $(I+\tau A_{h,2})^{-1} Y_{h,\tau}^n$ within every step, we obtain every $X^n_{h,\tau}$. So no additional computations are needed to get the arguments of $B_h$ and $f_h$ or to save $\{X^n_{h,\tau}\}_{n=1}^N$.
\end{remark}

\section{Convergence analysis}\label{sec:Conv_analysis}

We can now turn to the main analytical results of this paper. In the coming section, we prove explicit bounds for the error of our numerical scheme \eqref{eq:scheme} in Theorem~\ref{thm:result} and \ref{thm:result_selfadjoint}. To provide this result, we begin with some auxiliary statements collected in the coming lemmas. In Section~\ref{subsec:aux_conv_analysis}, we state some useful results from terms that appear in the error bound. In Section~\ref{subsec:errorTerms}, we begin by looking more closely at three different error parts: the error stemming from the initial condition, the error from the drift term, and the error obtained by the diffusion term. These three error parts can then be combined into the main result at the end of the subsection.

\subsection{Basic estimates}\label{subsec:aux_conv_analysis}

We begin to provide two basic results about certain operator products to stay bounded. This helps to shift around operators and prove the desired bounds.

\begin{lemma}\label{lem:Ah}
	Let Assumptions~\ref{ass:A}, \ref{ass:proj_err}, and \ref{ass:Ah} be fulfilled. Then it follows that 
	\begin{enumerate}[label={(\roman*)}, ref={(\roman*)}]
		\item \label{lem:switch}
		$\| A_hP_hA^{-1} \|_{\L(H)} \leq C$;
		\item \label{lem:disc2cont}
		$\| A_{h,\ell}A_h^{-1}P_h \|_{\L(H)}  \leq C$, $\ell \in \{1,2\}$.
	\end{enumerate}
\end{lemma}

\begin{proof}
	To prove \ref{lem:switch}, we insert Assumptions~\ref{ass:Ah}~\ref{ass:Ah2} and~\ref{ass:dissa_exch} and obtain
	\begin{equation*}
		\| A_hP_hA^{-1} \|_{\L(H)} 
		\leq \lno A_hP_h(A^{-1}-A_h^{-1}P_h)\Loptwon+\lno P_h\Loptwon
		\leq Ch^{-2}h^2+1.
	\end{equation*}
	For \ref{lem:disc2cont}, we use Assumptions~\ref{ass:Ah}~\ref{ass:Ah2}--\ref{ass:dissa_exch} and~\ref{ass:op_exch} to find
	\begin{align*}
		&\| A_{h,\ell}A_h^{-1}P_h \|_{\L(H)} \\
		&\leq  \lno A_{h,\ell}\(A_h^{-1}P_h-P_hA^{-1}\)\Loptwon+ \| A_{h,\ell}P_hA^{-1}\|_{\L(H)} \\
		&\leq \| A_{h,\ell}P_h \|_{\L(H)}  \lno A_h^{-1}P_h-A^{-1}\Loptwon+\lno (A_{h,\ell}P_h-P_hA_{\ell})A^{-1}\Loptwon + \| P_hA_{\ell}A^{-1} \|_{\L(H)} \\
		&\leq   Ch^{-2}h^2 +C + \| A_{\ell}A^{-1} \|_{\L(H)} \leq C.
	\end{align*}
\end{proof}

The following bound shows that the norm of $S_{h,\tau}^n (I+\tau A_{h,2})^{-1}$ is bounded by one for every $n \in \N$. Note that we include $(I+\tau A_{h,2})^{-1}$ to handle the quadratic term.

\begin{lemma}\label{lem:Sne}
	Let Assumptions~\ref{ass:A}, \ref{ass:proj_err}, and \ref{ass:Ah} be fulfilled. Further, let $S_{h,\tau}$ be given as in \eqref{eq:def_S}. Then for every $n \in \N$, it follows that
	\begin{equation*}
		\| S_{h,\tau}^n (I+\tau A_{h,2})^{-1} P_h \|_{\L(H)} \leq 1.
	\end{equation*}
\end{lemma}

\begin{proof}
	The proof is inspired by some arguments from \cite[Lemma~3.1]{HansenHenningsson.2016}. First, we rewrite $S_{h,\tau}^n (I+\tau A_{h,2})^{-1}$ as follows
	\newpage
	\begin{align}
		S_{h,\tau}^n (I+\tau A_{h,2})^{-1}
		\nonumber
		&= \big((I+\tau A_{h,2})^{-1}(I+\tau A_{h,1})^{-1} (I+\tau^2A_{h,1}A_{h,2}) \big)^n (I+\tau A_{h,2})^{-1}\\
		\nonumber
		&= (I+\tau A_{h,2})^{-1} \big((I+\tau A_{h,1})^{-1} (I+\tau^2A_{h,1}A_{h,2}) (I+\tau A_{h,2})^{-1}\big)^n \\
		\nonumber
		&= (I+\tau A_{h,2})^{-1} \Big( \half(I+\tau A_{h,1})^{-1}\big( (I-\tau A_{h,1})(I-\tau A_{h,2})\\
		\nonumber
		&\hspace{3cm}+(I+\tau A_{h,1})(I+\tau A_{h,2}) \big) (I+\tau A_{h,2})^{-1} \Big)^n\\
		\label{eq:proofSbound}
		&= (I+\tau A_{h,2})^{-1}\Big(\frac{1}{2}(I+\tau A_{h,1})^{-1}(I-\tau A_{h,1})(I-\tau A_{h,2})(I+\tau A_{h,2})^{-1}+\frac{1}{2}I\Big)^n.
	\end{align}
	In the next step, we show that the operators $(I+\tau A_{h,1})^{-1}(I-\tau A_{h,1})P_h $ and $(I-\tau A_{h,2})(I+\tau A_{h,2})^{-1}P_h$ are non-expansive. This follows from the fact that for every $v_h \in V_h$, we obtain
	\begin{align*}
		&\| (I+\tau A_{h,1})^{-1}(I-\tau A_{h,1})v_h \|_{H}^2\\
		&= \| (I+\tau A_{h,1})^{-1}v_h \|_{H}^2
		-2\tau  \((I+\tau A_{h,1})^{-1}v_h , A_{h,1}(I+\tau A_{h,1})^{-1}v_h\)_{H}
		+\| \tau(I+\tau A_{h,1})^{-1} A_{h,1}v_h \|_{H}^2\\
		&\leq \| (I+\tau A_{h,1})^{-1}v_h \|_{H}^2
		+ 2\tau\( (I+\tau A_{h,1})^{-1}v_h, A_{h,1}(I+\tau A_{h,1})^{-1}v_h\)_{H}
		+ \| \tau(I+\tau A_{h,1})^{-1} A_{h,1}v_h \|_{H}^2\\
		&=	\| (I+\tau A_{h,1})^{-1}(I+\tau A_{h,1})v_h \|_{H}^2 
		= \| v_h \|_{H}^2,
	\end{align*}
	where we used Assumption~\ref{ass:Ah}~\ref{ass:Ah_pos}. Analogously, the same follows for $(I-\tau A_{h,2})(I+\tau A_{h,2})^{-1}P_h$.
	Using the bound
	\begin{align*}
		&\Big\| \frac{1}{2}(I+\tau A_{h,1})^{-1}(I-\tau A_{h,1})(I-\tau A_{h,2})(I+\tau A_{h,2})^{-1} P_h +\frac{1}{2}I \Big\|_{\L(H)}\\
		&\leq \frac{1}{2} \big\| (I+\tau A_{h,1})^{-1}(I-\tau A_{h,1})(I-\tau A_{h,2})(I+\tau A_{h,2})^{-1} P_h \big\|_{\L(H)} 
		+ \frac{1}{2}
		\leq 1
	\end{align*}
	in \eqref{eq:proofSbound}, it follows that the $\L(H)$-norm of $S_{h,\tau}^n (I+\tau A_{h,2})^{-1}P_h$ is bounded by one.
\end{proof}

The following two lemmas give results about the difference between the semigroup $\e^{-\tau A_h}$ generated by the $A_h$ and either $(I+\tau A_{h,2})^{-1}(I+\tau A_{h,1})^{-1}$ or $S_{h,\tau}$.

\begin{lemma}\label{lem:est_eab}
	Let Assumptions~\ref{ass:A}, \ref{ass:proj_err}, and \ref{ass:Ah} be fulfilled. Then for all values $\zeta\in[0,1]$, it follows that
	\begin{equation*}
		\big\| (I+\tau A_{h,2})\( \e^{-\tau A_h}-(I+\tau A_{h,2})^{-1}(I+\tau A_{h,1})^{-1}\)P_hA^{-\zeta} \big\|_{\L(H)}
		\leq C \tau^{\zeta}.
	\end{equation*}
\end{lemma}

\begin{proof}
	We prove this lemma using an interpolation result. For this, we begin to prove the bound for $\zeta = 0$, then $\zeta = 1$, and provide the bound for all the in-between values using Lemma~\ref{lem:interA}.
	For the case $\zeta=0$, we apply \eqref{eq:ana_semi_h1} and Lemma~\ref{lem:Ah}~\ref{lem:disc2cont} 	to find
	\begin{align*}
		\|\mathcal{B} \|_{\L(H)}
		&:=\big\| (I+\tau A_{h,2}) \big( \e^{-\tau A_h}-(I+\tau A_{h,2})^{-1}(I+\tau A_{h,1})^{-1} \big) P_h \big\|_{\L(H)}\\
		&\leq \| (I+\tau A_{h,2}) \e^{-\tau A_h} P_h \|_{\L(H)} + \| (I+\tau A_{h,1})^{-1} P_h \|_{\L(H)}\\
		&\leq \| \e^{-\tau A_h}P_h \|_{\L(H)} + \| \tau A_{h,2}A_{h}^{-1}A_h \e^{-\tau A_h} P_h\|_{\L(H)} + 1\\
		&\leq 2+C \| A_{h,2}A_{h}^{-1}P_h \|_{\L(H)} \| \tau A_h \e^{-\tau A_h} P_h\|_{\L(H)}
		\leq C.
	\end{align*}
	This concludes the case for $\zeta=0$. Now we look at the case $\zeta=1$ and use the identity $(I+\tau A_{h,1})^{-1}=I-\tau (I+\tau A_{h,1})^{-1} A_{h,1}$ in $V_h$ as well as Lemma~\ref{lem:Ah}~\ref{lem:switch}, \eqref{eq:ana_semi_h2}, and Lemma~\ref{lem:Ah}~\ref{lem:disc2cont} to find that
	\begin{align*}
		&\| \mathcal{B} A^{-1} \|_{\L(H)} \\
		&\leq \big\| \big( (I+\tau A_{h,2})\e^{-\tau A_h} - (I+\tau A_{h,1})^{-1} \big) A_h^{-1} P_h \big \|_{\L(H)}
		\| A_hP_hA^{-1} \|_{\L(H)}\\
		&= \big\| \big( (I+\tau A_{h,2})\e^{-\tau A_h} - (I-\tau(I+\tau A_{h,1})^{-1}A_{h,1}) \big )A_h^{-1} P_h \big\|_{\L(H)}  \| A_hP_hA^{-1} \|_{\L(H)}\\
		&\leq C \big( \| A_h^{-1} (\e^{-\tau A_h}-I ) P_h \|_{\L(H)}
		+ \tau \| A_{h,2}A_h^{-1}\e^{-\tau A_h} P_h \|_{\L(H)} 
		+ \tau \|(I+\tau A_{h,1})^{-1}A_{h,1}A_{h}^{-1} P_h \|_{\L(H)} \big)\\
		&\leq C \big( \tau+ \tau \| A_{h,2}A_h^{-1} P_h \|_{\L(H)} 
		+ \tau \| A_{h,1}A_{h}^{-1} P_h \|_{\L(H)} \big) 
		\leq C\tau.
	\end{align*}
	This concludes the case for $\zeta=1$. 
	The result for $\zeta\in(0,1)$ follows from Lemma~\ref{lem:interA}, which implies that $\lno\mathcal{B} A^{-\zeta}\Loptwon \leq C \tau^{\zeta}$ holds. This completes the proof.
\end{proof}

\begin{lemma}\label{lem:est_eS}
	Let Assumptions~\ref{ass:A}, \ref{ass:proj_err}, and \ref{ass:Ah} be fulfilled. Further, let $S_{h,\tau}$ be given as in \eqref{eq:def_S}. For all $\zeta\in[0,1]$ and $s \in \R^+$, it follows that
	\begin{equation*}
		\lno(I+\tau A_{h,2})\(\e^{-\tau A_h}-S_{h,\tau}\)\e^{-sA_h}P_hA^{-\zeta}\Loptwon
		\leq C\frac{\tau^{1+\zeta}}{s}.
	\end{equation*}
\end{lemma}

\begin{proof}
	Again, we prove this lemma by using an interpolation result from Lemma~\ref{lem:interA}.
	Inserting the definition of $S_{h,\tau}$ from \eqref{eq:def_S}, the identity $I=(I+\tau A_{h,1})^{-1}+\tau (I+\tau A_{h,1})^{-1}A_{h,1}$, and $A_{h,1}+A_{h,2}=A_{h}$, we can rewrite the left-hand side from the claimed result for $\zeta = 0$ as
	\begin{align*}
		\mathcal{B} &:= (I+\tau A_{h,2}) (\e^{-\tau A_h}-S_{h,\tau} ) \e^{-sA_h}P_h\\
		&= \big( ((I+\tau A_{h,1})^{-1}+\tau(I+\tau A_{h,1})^{-1} A_{h,1}) (I+\tau A_{h,2})\e^{-\tau A_h}\\
		&\qquad-(I+\tau A_{h,1})^{-1} (I+\tau^2 A_{h,1} A_{h,2}) \big) \e^{-sA_h}P_h\\
		&= \big( ( (I+\tau A_{h,1})^{-1} + \tau(I+\tau A_{h,1})^{-1} A_{h,1}) \e^{-\tau A_h}\\
		&\qquad+ ( \tau (I+\tau A_{h,1})^{-1} A_{h,2} + \tau^2(I+\tau A_{h,1})^{-1} A_{h,1} A_{h,2}) \e^{-\tau A_h}\\
		&\qquad- (I+\tau A_{h,1})^{-1} - \tau^2 (I+\tau A_{h,1})^{-1} A_{h,1} A_{h,2} \big) \e^{-sA_h}P_h\\
		&= \big( ((I+\tau A_{h,1})^{-1}+\tau^2(I+\tau A_{h,1})^{-1} A_{h,1}A_{h,2}) (\e^{-\tau A_h}-I)\\
		&\qquad + \tau (I+\tau A_{h,1})^{-1}A_h \e^{-\tau A_h} \big) A^{-1}_hA_h\e^{-sA_h}P_h.
	\end{align*}
	For the case $\zeta=0$, we use \eqref{eq:ana_semi_h1}, the fact that the operators $(I+\tau A_{h,1})^{-1}P_h$, $\e^{-\tau A_h} P_h$, and $\tau(I+\tau A_{h,1})^{-1} A_{h,1} P_h = P_h - (I+\tau A_{h,1})^{-1}P_h $ are bounded operators, \eqref{eq:ana_semi_h2}, Lemma~\ref{lem:Ah}~\ref{lem:disc2cont}, and find
	\begin{align*}
		\lno \mathcal{B}  \Loptwon
		&\leq \big\| \big( (I+\tau A_{h,1})^{-1}+\tau^2(I+\tau A_{h,1})^{-1} A_{h,1}A_{h,2} \big) A_{h}^{-1} ( \e^{-\tau A_h}-I)P_h \\ 
		&\qquad\qquad + \tau (I+\tau A_{h,1})^{-1} \e^{-\tau A_h} P_h \big\|_{\L(H)}
		\lno A_h \e^{-sA_h}P_h \Loptwon\\
		&\leq\frac{C}{s}\Big( \lno(I+\tau A_{h,1})^{-1}A_{h}^{-1} (\e^{-\tau A_h}-I) P_h \Loptwon \\
		&\qquad +\lno \tau^2(I+\tau A_{h,1})^{-1} A_{h,1} A_{h,2}A_{h}^{-1} (\e^{-\tau A_h}-I) P_h \Loptwon\\
		&\qquad +\lno\tau (I+\tau A_{h,1})^{-1} \e^{-\tau A_h} P_h \Loptwon \Big)\\
		&\leq\frac{C}{s} \Big( \lno A_{h}^{-1} (\e^{-\tau A_h}-I ) P_h \Loptwon \\
		&\qquad + \tau \lno \tau(I+\tau A_{h,1})^{-1} A_{h,1}\Lophtwon \lno A_{h,2}A_{h}^{-1} P_h \Loptwon
		\lno(\e^{-\tau A_h}-I) P_h \Loptwon +\tau \Big)\\
		&\leq C \frac{\tau}{s}.
	\end{align*}
	This concludes the first part of the proof for $\zeta=0$. 
	For the second case $\zeta=1$, we use \eqref{eq:ana_semi_h1}, Lemma~\ref{lem:Ah}~\ref{lem:switch}, and obtain
	\newpage
	\begin{align*}
		\lno\mathcal{B} A^{-1}\Loptwon
		& \leq \big\| \big( (I+\tau A_{h,1})^{-1}+\tau^2(I+\tau A_{h,1})^{-1} A_{h,1}A_{h,2} \big) (\e^{-\tau A_h}-I) A^{-2}_h P_h\\
		&\qquad + \tau (I+\tau A_{h,1})^{-1}A_h \e^{-\tau A_h} A^{-2}_h P_h \big\|_{\L(H)} 
		\| A_h\e^{-sA_h} P_h \|_{\L(H)}
		\| A_hP_hA^{-1} \|_{\L(H)} \\
		& \leq \frac{C}{s} \Big( \big\| (I+\tau A_{h,1})^{-1} (\e^{-\tau A_h}+\tau A_h \e^{-\tau A_h}-I )A_h^{-2} P_h  \big\|_{\L(H)}\\
		&\qquad + \tau^2 \big\| (I+\tau A_{h,1})^{-1} A_{h,1}A_{h,2}A_h^{-2} ( \e^{-\tau A_h}-I ) P_h \big\|_{\L(H)} \Big) \\
		& \leq \frac{C}{s} \Big( \big\| (I+\tau A_{h,1})^{-1} (\e^{-\tau A_h}+\tau A_h \e^{-\tau A_h}-I)A_h^{-2} P_h  \big\|_{\L(H)}\\
		&\qquad + \tau \| \tau(I+\tau A_{h,1})^{-1} A_{h,1} P_h \|_{\L(H)} 
		\| A_{h,2}A_h^{-1} P_h \|_{\L(H)} \lno A_h^{-1} (\e^{-\tau A_h}-I) \Lophtwon \Big) \\
		&\leq \frac{C}{s} \tau^2, 
	\end{align*}
	where we use  the facts that
	\begin{equation*}
		\lno (\e^{-\tau A_h}+\tau A_h\e^{-\tau A_h}-I )A_h^{-2}\Lophtwon
		= \Big \| \int_0^{\tau} r A_h^2\e^{-rA_h}\diff{r} A_h^{-2} P_h \Big\|_{\L(H)}
		\leq C\tau^2,
	\end{equation*}
	the operator $ \tau(I+\tau A_{h,1})^{-1} A_{h,1}P_h = P_h - (I+\tau A_{h,1})^{-1}P_h$ is bounded, Lemma~\ref{lem:Ah}~\ref{lem:disc2cont}, and \eqref{eq:ana_semi_h2} in the last step.
	Now we have proved the claimed bound for both $\zeta = 0$ and $\zeta = 1$. The last step is to deduce the result for $\zeta\in(0,1)$. We can apply Lemma~\ref{lem:interA} and obtain
	\begin{equation*}
		\lno \mathcal{B} A^{-\zeta} \Loptwon\leq C \Big(\frac{\tau^2}{s}\Big)^{\zeta}
		\Big(\frac{\tau}{s} \Big)^{1-\zeta} = C \frac{\tau^{1+\zeta} }{s},
	\end{equation*}
	which proves the claim of the lemma.
\end{proof}

The previous lemmas can now be combined to a bound that quantifies the difference between the exact flow given through the semigroup $\e^{-t_n A}$ and its approximation $S^{n-1}_{h,\tau}(I+\tau A_{h,2})^{-1}(I+\tau A_{h,1})^{-1}P_h$. We will state two different versions of this bound. These differ between a minimal set of assumptions in Lemma~\ref{lem:Split_error}, and additional assumptions on the setting in Lemma~\ref{lem:Split_error} that provide the same final error as in \cite[Theorem~10.34]{lord_powell_shardlow_2014} for our scheme.

\begin{lemma}\label{lem:Split_error}
	Let Assumptions~\ref{ass:A}, \ref{ass:proj_err}, and \ref{ass:Ah} be fulfilled. Further, let $S_{h,\tau}$ be given as in \eqref{eq:def_S}. For every $n \in \{1, \dots, N\}$ and $\theta, \zeta \in [0,1]$, it follows that 
	\begin{equation*}
		\big\| \big( \e^{-t_nA} - S^{n-1}_{h,\tau}(I+\tau A_{h,2})^{-1}(I+\tau A_{h,1})^{-1}P_h \big) A^{-\theta} \big\|_{\mathcal{L}(H)}
		\leq C \big((1+\ln(n)) \tau^{\theta} + t_n^{- \zeta(1-\theta)} h^{2\theta + 2\zeta(1-\theta)}\big).
	\end{equation*}
\end{lemma}

\begin{proof}
	We begin to split the term that we want to bound into two parts
	\begin{align*}
		&\big\| \big( \e^{-t_nA} - S^{n-1}_{h,\tau}(I+\tau A_{h,2})^{-1}(I+\tau A_{h,1})^{-1}P_h \big)A^{-\theta} \big\|_{\mathcal{L}(H)}\\
		&\leq \big\| \big( \e^{-t_nA} - \e^{-t_nA_h} P_h \big) A^{-\theta} \big\|_{\mathcal{L}(H)}
		+ \big\| \big( \e^{-t_nA_h} - S^{n-1}_{h,\tau}(I+\tau A_{h,2})^{-1}(I+\tau A_{h,1})^{-1}\big) P_h A^{-\theta} \big\|_{\mathcal{L}(H)}\\
		&=:  \Gamma_1   +  \Gamma_2  .
	\end{align*}
	In the following, we abbreviate $\mathcal{B} = \e^{-t_n A}- \e^{-t_n A_h} P_h$. Then using that the semigroups $\e^{-t_nA}$ and $\e^{-t_nA_h}$ are bounded operators as well as Lemmas~\ref{app2} and \ref{app3}, it follows that
	\begin{equation}\label{eq:proof_errS1}
		\| \mathcal{B} \|_{\L(H)} \leq C,\quad 
		\| \mathcal{B}\|_{\L(H)} \leq C\frac{h^2}{t_n} \quad\text{and}\quad
		\| \mathcal{B}A^{-1}\|_{\L(H)} \leq Ch^2.
	\end{equation}
	Furthermore, we have
	\begin{equation}\label{eq:proof_errS2}
		\| \mathcal{B}\|_{\L(H)}
		=\| \mathcal{B}\|_{\L(H)}^{\zeta}\| \mathcal{B}\|_{\L(H)}^{1-\zeta}
		\leq \Bigl(C\frac{h^2}{t_n} \Bigr)^\zeta C^{1-\zeta}
		=C\frac{h^{2\zeta}}{t_n^\zeta}.
	\end{equation}
	For all given $\theta\in [0,1]$, the value $\zeta(1-\theta)$ also lies in $[0,1]$ for every choice of $\zeta \in [0,1]$.
	Hence, after applying Lemma~\ref{lem:interA} and inserting the third bound of \eqref{eq:proof_errS1} and \eqref{eq:proof_errS2}, it follows that
	\begin{equation*}
		\Gamma_1
		= \| \mathcal{B}A^{-\theta}\|_{\L(H)} 
		\leq C\| \mathcal{B}A^{-1}\|_{\L(H)}^{\theta} \| \mathcal{B}\|_{\L(H)}^{1-\theta}
		\leq \bigl(Ch^2\bigr)^\theta 
		\Bigl(C\frac{h^{2\zeta}}{t_n^\zeta}\Bigr)^{1-\theta}
		= C t_n^{- \zeta(1-\theta)} h^{2\theta + 2\zeta(1-\theta)}.
	\end{equation*}
	It remains to bound $ \Gamma_2$. We decompose $\Gamma_2$ using a telescopic sum structure where we can bound the single summands with Lemmas~\ref{lem:Sne},~\ref{lem:est_eS},~\ref{lem:est_eab} and then obtain
	\begin{align*}
		\Gamma_2  
		&=\big\| \big( \e^{-t_nA_h}-S^{n-1}_{h,\tau}(I+\tau A_{h,2})^{-1}(I+\tau A_{h,1})^{-1} \big)P_hA^{-\theta} \big\|_{\L(H)}\\
		&\leq \sum_{k=1}^{n-1} \big\| \big( S^{n-k-1}_{h,\tau}(I+\tau A_{h,2})^{-1} (I+\tau A_{h,2} ) 
		(\e^{-\tau A_h}-S_{h,\tau}) \e^{-t_k A_h}\big) P_hA^{-\theta} \big\|_{\L(H)}\\
		&\quad + \big\| S^{n-1}_{h,\tau}(I+\tau A_{h,2})^{-1} (I+\tau A_{h,2}) 
		\big(\e^{-\tau A_h}-(I+\tau A_{h,2})^{-1}(I+\tau A_{h,1})^{-1}  \big) P_hA^{-\theta} \big\|_{\L(H)}\\
		&\leq \sum_{k=1}^{n-1} \big\| \big( (I+\tau A_{h,2} ) 
		(\e^{-\tau A_h}-S_{h,\tau}) \e^{-t_k A_h}\big) P_hA^{-\theta} \big\|_{\L(H)}\\
		&\quad + \big\| (I+\tau A_{h,2}) 
		\big(\e^{-\tau A_h}-(I+\tau A_{h,2})^{-1}(I+\tau A_{h,1})^{-1}  \big) P_hA^{-\theta} \big\|_{\L(H)}\\
		&\leq C\tau^{\theta} \Big( 1+\sum_{k=1}^{n-1}\frac{\tau}{t_k} \Big)
		\leq C (1+\ln(n)) \tau^{\theta} .
	\end{align*}
	Here we used Lemma~\ref{lem:riemann} in the last step. This concludes the proof.
\end{proof}

\begin{lemma}\label{lem:Split_error_B}
	Let Assumptions~\ref{ass:A}, \ref{ass:proj_err}, and \ref{ass:Ah} be fulfilled, additionally let $A_{h,1}$ and $A_{h,2}$ be self-adjoint and commute. Further, let $S_{h,\tau}$ be given as in \eqref{eq:def_S}. For every $n \in \{1, \dots, N\}$ and $\theta, \zeta \in [0,1]$, it follows that 
	\begin{align*}
		&\big\| \big( \e^{-t_nA} - S^{n-1}_{h,\tau}(I+\tau A_{h,2})^{-1}(I+\tau A_{h,1})^{-1}P_h \big) A^{-\theta} \big\|_{\mathcal{L}(H)}\\
		&\leq C t_n^{- \zeta(1-\theta)} \big( (1+\ln(n))^{\theta} \tau^{\theta + \zeta(1-\theta)} + h^{2\theta + 2\zeta(1-\theta)}\big).
	\end{align*}
\end{lemma}

\begin{proof}
	In the following, we use the same abbreviation as in the proof of Lemma~\ref{lem:Split_error}. Analogously, we can bound the first error term and find that $\Gamma_1 \leq C t_n^{- \zeta(1-\theta)} h^{2\theta + 2\zeta(1-\theta)}$. The difference to the proof of the previous lemma is how we handle $\Gamma_{2}$. In the following, we abbreviate $\tilde{\mathcal{B}} = \e^{-t_n A_h} P_h - S^{n-1}_{h,\tau}(I+\tau A_{h,2})^{-1}(I+\tau A_{h,1})^{-1}P_h$. Then using that the  semigroup $\e^{-t_nA_h}$ is a bounded operator as well as Lemmas~\ref{lem:Sne},~\ref{lem:high_conv}, and the estimate of $\Gamma_2$ in the proof of Lemma~\ref{lem:Split_error} for $\theta = 1$, it follows that
	\begin{equation}\label{eq:proof_errS3}
		\| \tilde{\mathcal{B}} \|_{\L(H)} \leq C,\quad 
		\| \tilde{\mathcal{B}} \|_{\L(H)} \leq C\frac{\tau}{t_n} \quad\text{and}\quad
		\| \tilde{\mathcal{B}} A^{-1}\|_{\L(H)} \leq C \tau(1+\ln(n)).
	\end{equation}
	Additionally, we can combine the first two bounds from \eqref{eq:proof_errS3} and find for all $\zeta \in [0,1]$
	\begin{equation}\label{eq:proof_errS4}
		\| \tilde{\mathcal{B}}\|_{\L(H)}
		=\| \tilde{\mathcal{B}}\|_{\L(H)}^{\zeta}\| \tilde{\mathcal{B}}\|_{\L(H)}^{1-\zeta}
		\leq \Bigl(C\frac{\tau}{t_n} \Bigr)^\zeta C^{1-\zeta}
		=C\frac{\tau^{\zeta}}{t_n^\zeta}.
	\end{equation}
	An analogous argument combining \eqref{eq:proof_errS3} and \eqref{eq:proof_errS4} with the help of Lemma~\ref{lem:interA} as for $\Gamma_1$ in the proof of Lemma~\ref{lem:Split_error} shows that
	\begin{equation*}
		\Gamma_2
		= \| \tilde{\mathcal{B}}A^{-\theta}\|_{\L(H)} 
		\leq \bigl(C \tau (1+\ln(n)) \bigr)^\theta 
		\Bigl(C \frac{\tau^{\zeta}}{t_n^\zeta}\Bigr)^{1-\theta}
		= C (1+\ln(n))^{\theta} t_n^{- \zeta(1-\theta)} \tau^{\theta + \zeta(1-\theta)}.
	\end{equation*}
\end{proof}

\subsection{Convergence results} \label{subsec:errorTerms}

With the auxiliary results from the previous subsection in mind, we can now begin to bound the error of the numerical method \eqref{eq:scheme}. We begin by considering the difference between the exact solution in integral form \eqref{eq:SPDE_int} at a grid point $t_n$ and the numerical approximation \eqref{eq:scheme}. To analyze the error in the following subsection, we split the difference into three parts, as in \cite[Theorem 10.34]{lord_powell_shardlow_2014} for the semi-implicit Euler method,
\begin{align*}
	&X(t_n)-X_{h,\tau}^n\\
	&=\big(\e^{-t_nA}-S^{n-1}_{h,\tau}(I+\tau A_{h,2})^{-1}(I+\tau A_{h,1})^{-1}P_h \big)X_0\\
	&\quad+\sum_{k=0}^{n-1} \Big[ \int_{t_k}^{t_{k+1}} \e^{-(t_{n}-s)A}f(s,X(s))\diff{s}- \tau S^{n-k-1}_{h,\tau}(I+\tau A_{h,2})^{-1}(I+\tau A_{h,1})^{-1} f_h(t_k,X^k_{h,\tau}) \Big]\\\
	&\quad+\sum_{k=0}^{n-1} \int_{t_k}^{t_{k+1}} \big[ \e^{-(t_{n}-s)A}B(s,X(s))-S^{n-k-1}_{h,\tau}(I+\tau A_{h,2})^{-1}(I+\tau A_{h,1})^{-1} B_h(t_k,X^{k}_{h,\tau}) \big] \diff{W(s)}\\
	&= \Gamma_{X_0} + \Gamma_f + \Gamma_B .
\end{align*}
We will now consider these three terms in more detail in the following lemmas. The error $\Gamma_{X_0}$ can be estimated as follows.

\begin{lemma}\label{lem:conv_in} 
	Let Assumptions~\ref{ass:A}, \ref{ass:X0}, \ref{ass:proj_err}, and \ref{ass:Ah} be fulfilled. For every $n \in \{1, \dots, N\}$ and $\zeta \in [0,1]$, it follows that 
	\begin{equation*}
		\| \Gamma_{X_0} \|_{L^p(\Omega;H)} 
		\leq C \big((1+\ln(n)) \tau^{\theta_{X_0}} + t_n^{- \zeta(1-\theta_{X_0})} h^{2\theta_{X_0} + 2\zeta(1-\theta_{X_0})} \big).
	\end{equation*}
\end{lemma}

\begin{proof}
	This follows directly from Lemma~\ref{lem:Split_error} and Assumption~\ref{ass:X0}.
\end{proof}

In the next step, we consider the error $ \Gamma_f $ that arises from the drift term.

\begin{lemma}\label{lem:conv_f}
	Let Assumptions~\ref{ass:A}--\ref{ass:B}, \ref{ass:proj_err}, and \ref{ass:Ah} be fulfilled for $\theta_{f} \in (0, \theta_{X_0}) \cap (0, \frac{1}{2})$. For every $n \in \{1, \dots, N\}$, it follows that 
	\begin{equation*}
		\| \Gamma_f \|_{L^p(\Omega;H)} 
		\leq C \big(\tau^{\min(\theta_{X_0},\frac{1}{2})} + (1+\ln(n)) \tau^{\theta_f} + h^{2}\big) + C \tau\sum_{k=0}^{n-1} \| X(t_k)-X^k_{h,\tau} \|_{L^p(\Omega;H)}.
	\end{equation*}
\end{lemma}

\begin{proof}
	To bound the error term $ \Gamma_f $, we decompose it into four separate terms 
	\begin{align*}
		\Gamma_f
		&=\sum_{k=0}^{n-1} \int_{t_k}^{t_{k+1}} \big(\e^{-(t_n-s)A}-\e^{-(t_n-t_{k})A} \big)f(s,X(s))\diff{s} \\
		&\quad+\sum_{k=0}^{n-1} \e^{-(t_n-t_{k})A}\int_{t_k}^{t_{k+1}} \big( f(s,X(s))- f(t_k,X(t_k)) \big)\diff{s} \\
		&\quad+\tau\sum_{k=0}^{n-1} \big(\e^{-(t_n-t_{k})A}-S^{n-k-1}_{h,\tau}(I+\tau A_{h,2})^{-1}(I+\tau A_{h,1})^{-1}P_h\big) f(t_k,X(t_k)) \\
		&\quad+\tau\sum_{k=0}^{n-1} S^{n-k-1}_{h,\tau}(I+\tau A_{h,2})^{-1}(I+\tau A_{h,1})^{-1} \big(P_h f(t_k,X(t_k))-f_h(t_k,X^k_{h,\tau})\big) \\
		&:= \Gamma_{f,1}  + \Gamma_{f,2}  + \Gamma_{f,3}  + \Gamma_{f,4}  .
	\end{align*}
	For $ \Gamma_{f,1}  $, we apply the semigroup bounds \eqref{eq:ana_semi1}, \eqref{eq:ana_semi2}, use Assumption~\ref{ass:f_reg} and \eqref{eq:reg_space} with $\zeta =0$ from Theorem~\ref{thm:ex_reg_exact_sol} to find
	\begin{align*}
		\| \Gamma_{f,1}  \|_{L^p(\Omega;H)}
		&\leq \sum_{k=0}^{n-1} \Big\| \int_{t_k}^{t_{k+1}} A\e^{-(t_n-s)A} A^{-1} \big(I-\e^{-(s-t_{k}) A}\big) f(s,X(s))\diff{s} \Big\|_{L^p(\Omega;H)}\\
		&\leq C\tau \sup_{s\in[t_{n-1}, t_n]}\|f(s,X(s))\|_{L^p(\Omega;H)}+ C\sum_{k=0}^{n-2}\int_{t_k}^{t_{k+1}} \frac{|s-t_{k}|}{|t_n-s|} \lno f(s,X(s))\Lston \diff{s} \\
		&\leq C\tau \Big(1+\sup_{s\in[0,t_f]}\lno X(s)\Lston\Big)+C\sum_{k=0}^{n-2}\int_{t_k}^{t_{k+1}}\frac{\tau}{|t_n-s|} \big( 1+\lno X(s)\Lston \big) \diff{s}\\
		&\leq C(1+\ln(n)) \tau \Big(1+\sup_{s\in[0,t_f]}\lno X(s)\Lston\Big)
		\leq C (1+\ln(n)) \tau.
	\end{align*}
	The second term $ \Gamma_{f,2}  $ can be bound by applying the semigroup $\e^{-(t_n-t_{k})A}$ is a bounded operator, Assumption~\ref{ass:f_hold} and \eqref{eq:reg_hoelder} from Theorem~\ref{thm:ex_reg_exact_sol}
	\begin{align*}
		\| \Gamma_{f,2}  \|_{L^p(\Omega;H)}
		&\leq \sum_{k=0}^{n-1} \Big\| \e^{-(t_n-t_{k})A} \int_{t_k}^{t_{k+1}} \big( f(s,X(s))- f(t_k,X(t_k)) \big) \diff{s} \Big\|_{L^p(\Omega;H)} \\
		&\leq \sum_{k=0}^{n-1} \int_{t_k}^{t_{k+1}} \lno f(s,X(s))- f(t_k,X(t_k))\Lston \diff{s}\\
		&\leq C\sum_{k=0}^{n-1} \int_{t_k}^{t_{k+1}} |s-t_k|^{\min(\theta_{X_0},\frac{1}{2})} \Big( 1+\frac{ \|X(s)- X(t_k)\|_{L^p(\Omega;H)}}{|s-t_k|^{\min(\theta_{X_0},\frac{1}{2})}} \Big) \diff{s}\\
		&\leq C\tau^{\min(\theta_{X_0},\frac{1}{2})} \Big( 1+\sup_{s,t\in[0,t_f],s\neq t} \frac{ \|X(s)- X(t)\|_{L^p(\Omega;H)}}{|s-t|^{\min(\theta_{X_0},\frac{1}{2})}} \Big)
		\leq C\tau^{\min(\theta_{X_0},\frac{1}{2})}.
	\end{align*}
	Next we apply Lemma~\ref{lem:Split_error} with $\zeta = 1$ in combination with Assumption~\ref{ass:f_reg} and Lemma~\ref{lem:riemann}, which then leads to
	\begin{align*}
		\| \Gamma_{f,3}  \|_{L^p(\Omega;H)}
		&\leq \tau \sum_{k=0}^{n-1} \big\| \big( \e^{-(t_n-t_{k})A} - S^{n-k-1}_{h,\tau} (I+\tau A_{h,2})^{-1}(I+\tau A_{h,1})^{-1} P_h \big) f(t_k,X(t_k)) \big\|_{L^p(\Omega;H)}  \\
		&\leq C \tau \sum_{k=0}^{n-1} \Big(\tau^{\theta_f} (1+\ln(n-k)) + \frac{h^{2}} {(t_n-t_k)^{1 - \theta_f}} \Big)  \| A^{\theta_f}f(t_k,X(t_k)) \|_{L^p(\Omega;H)} \\
		&\leq C\big((1+\ln(n)) \tau^{\theta_f} + h^{2} \big)  \Big( 1+\sup_{s\in[0,t_f]} \| A^{\theta_f}X(s)\|_{L^p(\Omega;H)} \Big)\\
		&\leq C\big((1+\ln(n)) \tau^{\theta_f} + h^{2} \big) ,
	\end{align*}
	where we used \eqref{eq:reg_space} and $\theta_{f} \leq \theta_{X_0}$ in the last step.
	The last error term $ \Gamma_{f,4}  $ can be bounded with Lemma~\ref{lem:Sne}, the facts that $\|(I+\tau A_{h,1})^{-1}P_h\|_{\L(H)} \leq 1$ and $P_h f (t_k,X(t_k)) = f_h(t_k,X(t_k))$ together with Lemma~\ref{lem:fh_Bh}~\ref{lem:fh_hold}. Then, it follows that 
	\begin{align*}
		\| \Gamma_{f,4}  \|_{L^p(\Omega;H)}
		&\leq \tau\sum_{k=0}^{n-1} \big\| S^{n-k-1}_{h,\tau}(I+\tau A_{h,2})^{-1} (I+\tau A_{h,1})^{-1} \big(P_hf(t_k,X(t_k))-f_h(t_k,X^k_{h,\tau}) \big) \big\|_{L^p(\Omega;H)}\\
		&\leq C \tau \sum_{k=0}^{n-1} \| f_h(t_k,X(t_k))-f_h(t_k,X^k_{h,\tau}) \|_{L^p(\Omega;H)}
		\leq C \tau\sum_{k=0}^{n-1} \| X(t_k)-X^k_{h,\tau} \|_{L^p(\Omega;H)}.
	\end{align*}
	Combining the bounds for $ \Gamma_{f,1}  $, $ \Gamma_{f,2} $, $\Gamma_{f,3}$, and $ \Gamma_{f,4} $, we obtain the claimed result.
\end{proof}

For the error part $ \Gamma_B $ arising from the stochastic perturbation, we bound the square of the $L^p(\Omega;H)$-norm for notational convenience. 

\begin{lemma}\label{lem:conv_W}
	Let Assumptions~\ref{ass:A}--\ref{ass:B}, \ref{ass:proj_err}, \ref{ass:Ah}, and \ref{ass:Proj_error_U} be fulfilled for $\theta_{X_0} \in [\theta_{B}, \theta_{B} + \frac{1}{2})$. For every $n \in \{1, \dots, N\}$, it follows that 
	\begin{equation*}
		\| \Gamma_B \|_{L^p(\Omega;H)}^2 
		\leq C ( 1+\ln(n))^{2} \big( \tau^{2\min(\theta_{X_0},\theta_{B}) } + h^{2(2\min(\theta_{B}, \theta_{U})+1)} \big) 
		+ C \tau \sum_{k=0}^{n-1} \| X(t_k) - X^k_{h,\tau} \|_{L^p(\Omega;H)}^2.
	\end{equation*}
\end{lemma}

\begin{proof}
	To estimate the $ \Gamma_B $-error, we begin to apply the Burkholder--Davis--Gundy inequality \eqref{eq:burk} and decompose the error into four parts as follows
	\begin{align*}
		&\| \Gamma_B \|_{L^p(\Omega;H)}^2\\
		&\leq C \sum_{k=0}^{n-1} \int_{t_k}^{t_{k+1}} \big\| \e^{-(t_n-s)A} B(s,X(s))- S^{n-k-1}_{h,\tau}(I+\tau A_{h,2})^{-1}(I+\tau A_{h,1})^{-1} B_h(t_k,X^{k}_{h,\tau}) \big\|_{L^p(\Omega;L_2^0)}^2 \diff{s} \\
		&\leq C \Big( \sum_{k=0}^{n-1} \int_{t_k}^{t_{k+1}} \big\| \big( \e^{-(t_n-s)A} - \e^{-(t_n-t_k)A} \big)B(s,X(s)) \big\|_{L^p(\Omega;L_2^0)}^2 \diff{s}\\
		&\quad+ \sum_{k=0}^{n-1} \int_{t_k}^{t_{k+1}} \big\| \e^{-(t_n-t_{k})A}\big(B(s,X(s))-B(t_k,X(t_k))\big) \big\|_{L^p(\Omega;L_2^0)}^2 \diff{s}\\
		&\quad+ \sum_{k=0}^{n-1} \int_{t_k}^{t_{k+1}} \big\| \big(\e^{-(t_n-t_{k})A}-S^{n-k-1}_{h,\tau}(I+\tau A_{h,2})^{-1}(I+\tau A_{h,1})^{-1} P_h\big)B(t_k,X(t_k)) \big\|_{L^p(\Omega;L_2^0)}^2 \diff{s}\\
		&\quad+ \sum_{k=0}^{n-1} \int_{t_k}^{t_{k+1}} \big\| S^{n-k-1}_{h,\tau}(I+\tau A_{h,2})^{-1}(I+\tau A_{h,1})^{-1} \big(P_h B(t_k,X(t_k))-B_h(t_k,X^k_{h,\tau})\big) \big\|_{L^p(\Omega;L_2^0)}^2 \diff{s} \Big)\\
		&:=C \big( \Gamma_{B,1} + \Gamma_{B,2} + \Gamma_{B,3} + \Gamma_{B,4} \big).
	\end{align*}
	We begin with $ \Gamma_{B,1}  $. This can be bound by applying the semigroup bounds \eqref{eq:ana_semi1}, \eqref{eq:ana_semi2}, Assumption~\ref{ass:B_reg}, and \eqref{eq:reg_space} with $\zeta =0$ from Theorem~\ref{thm:ex_reg_exact_sol}. More precisely, we obtain
	\begin{align*}				
		\Gamma_{B,1}  
		&= \sum_{k=0}^{n-1} \int_{t_k}^{t_{k+1}} \big\| A^{\frac{1}{2}} \e^{-(t_n-s)A} A^{- \frac{1}{2}} (I-\e^{-(s-t_k)A}) B(s,X(s)) \big\|_{L^p(\Omega; L^0_2)}^{2} \diff{s} \\ 
		&\leq \tau \sup_{s \in [t_{n-1}, t_n]} \big\| B(s,X(s)) \big\|_{L^p(\Omega; L^0_2)}^{2}
		+ C \sum_{k=0}^{n-2} \int_{t_k}^{t_{k+1}} \frac{|s-t_k|}{|t_n-s|} \big( 1+ \| X(s)\|_{L^p(\Omega;H) }^2 \big)\diff{s} \\ 
		&\leq \Big(1+ C\tau \sum_{k=0}^{n-2} \int_{t_k}^{t_{k+1}} |t_n-s|^{-1} \diff{s} \Big)
		\Big(1+\sup_{s\in[0,t_f]} \| X(s) \|_{L^p(\Omega;H) }^2 \Big) 
		\leq C (1 + \ln(n)) \tau.
	\end{align*}
	To bound the second term $ \Gamma_{B,2}  $, we use Assumption~\ref{ass:B_hold} and \eqref{eq:reg_hoelder} from Theorem~\ref{thm:ex_reg_exact_sol}, to find
	\begin{align*}
		\Gamma_{B,2}  
		&= \sum_{k=0}^{n-1} \int_{t_k}^{t_{k+1}} \big\| \e^{-(t_n-t_{k})A} \big(B(s,X(s))-B(t_k,X(t_k))\big) \big\|_{L^p(\Omega;L_2^0)}^2 \diff{s} \\
		&\leq C \sum_{k=0}^{n-1} \int_{t_k}^{t_{k+1}} 
		\big\| B(s,X(s))-B(t_k,X(t_k)) \big\|_{L^p(\Omega;L_2^0)}^2
		\diff{s}\\
		&\leq C \sum_{k=0}^{n-1} \int_{t_k}^{t_{k+1}} |s - t_k|^{2\min(\theta_{X_0}, \frac{1}{2}) } \Big\| 1+\frac{\|X(s)-X(t_k)\|_{H}}{|s-t_k|^{\min(\theta_{X_0}, \frac{1}{2}) }}  \Big\|_{L^p(\Omega;\R)}^2
		\diff{s} 
		\leq C \tau^{2\min(\theta_{X_0}, \frac{1}{2}) }.
	\end{align*}
	When bounding the third term $ \Gamma_{B,3} $, we can use Lemma~\ref{lem:Split_error} with $\zeta \in [0,1]$ such that $\zeta(1-\theta_{B}) = \frac{1}{2}$, combine this with Assumption~\ref{ass:B_reg} and apply \eqref{eq:reg_space} from Theorem~\ref{thm:ex_reg_exact_sol} and Lemma~\ref{lem:riemann}. This then leads to
	\begin{align*}
		\Gamma_{B,3}  
		&= \tau \sum_{k=0}^{n-1} \big\| \big(\e^{-(t_n-t_{k})A}-S^{n-k-1}_{h,\tau}(I+\tau A_{h,2})^{-1}(I+\tau A_{h,1})^{-1} P_h\big)B(t_k,X(t_k)) \big\|_{L^p(\Omega;L_2^0)}^2\\
		&\leq C \tau \sum_{k=0}^{n-1} \Big( \tau^{\theta_{B}} ( 1+\ln(n-k)) + \frac{h^{2\theta_{B}+1}}{(t_n-t_k)^{\frac{1}{2}}} \Big)^2 \big\| A^{\theta_{B}} B(t_k,X(t_k)) \big\|_{L^p(\Omega;L_2^0)}^2\\
		&\leq C \tau \sum_{k=0}^{n-1} \Big(\tau^{2\theta_{B}} ( 1+\ln(n-k))^2 + \frac{h^{2(2\theta_{B}+1)}}{t_n-t_k} \Big) \big( 1 + \| A^{\theta_{B}} X(t_k)\|_{L^p(\Omega;H)}\big)^2\\
		&\leq C (1+\ln(n))^2 \big( \tau^{2\theta_{B}} + h^{2(2\theta_{B}+1)} \big).
	\end{align*}	
	The remaining term $\Gamma_{B,4}$, can be bound by an application of Lemma~\ref{lem:Sne} and the fact that the term $\| (I+\tau A_{h,1})^{-1} P_h \|_{\L(H)}$ is bounded. Using Lemmas~\ref{lem:fh_Bh}~\ref{lem:Bh_regP}, \ref{lem:Bh_hold}, and \eqref{eq:reg_space} from Theorem~\ref{thm:ex_reg_exact_sol}, we then obtain 
	\begin{align*}
		\Gamma_{B,4}  
		&= \tau \sum_{k=0}^{n-1} \big\| S^{n-k-1}_{h,\tau}(I+\tau A_{h,2})^{-1}(I+\tau A_{h,1})^{-1} \big(P_h B(t_k,X(t_k))-B_h(t_k,X^k_{h,\tau})\big) \big\|_{L^p(\Omega;L_2^0)}^2 \\
		&\leq C \tau \sum_{k=0}^{n-1} \Big(\| P_hB(t_k,X(t_k))-B_h(t_k,X(t_k)) \|_{L^p(\Omega;L_2^0)}^2\\
		&\qquad \qquad \qquad+ \| B_h(t_k,X(t_k))-B_h(t_k,X^k_{h,\tau}) \|_{L^p(\Omega;L_2^0)}^2\Big)\\
		&\leq C \tau h^{2(2\theta_{U}+1)} \sum_{k=1}^{n-1} \big(1+\|A^{\theta_{X_0}} X(t_k)\|_{L^p(\Omega;H)}^2\big)
		+ C \tau\sum_{k=0}^{n-1} \| X(t_k)-X^k_{h,\tau} \|_{L^p(\Omega;H)}^2\\
		&\leq C h^{2(2\theta_{U}+1)} + C \tau\sum_{k=0}^{n-1} \| X(t_k)-X^k_{h,\tau} \|_{L^p(\Omega;H)}^2.
	\end{align*}
	Combining the bounds for $\Gamma_{B,1}$, $\Gamma_{B,2}$, $\Gamma_{B,3}$, and $\Gamma_{B,4}$, we obtain the claimed result.
\end{proof}

\begin{theorem}\label{thm:result}
	Let Assumptions~\ref{ass:A}--\ref{ass:B}, \ref{ass:proj_err}, \ref{ass:Ah}, and \ref{ass:Proj_error_U} be fulfilled for $\theta_{X_0} \in [\theta_{B}, \theta_{B} + \frac{1}{2})$ and $\theta_{f} \in (0, \min(\theta_{X_0}, \frac{1}{2}))$. For every $n \in \{1, \dots, N\}$ and $\zeta\in[0,1]$, it follows that 
	\begin{align*}
		\| X(t_n)-X^n_{h,\tau} \|_{L^p(\Omega;H)}
		&\leq C (1+\ln(n)) \big( \tau^{\min(\theta_{X_0}, \theta_f, \theta_B)} + h^{2\min(\theta_{B}, \theta_{U})+1} \big) \\
		&\quad + C \big( t_n^{- \zeta(1-\theta_{X_0})} h^{2\theta_{X_0} + 2\zeta(1-\theta_{X_0}) } \big).
	\end{align*}
\end{theorem}
\begin{remark}
	For the final time $t_N = t_f$, i.e. $n = N$, we can choose $\zeta \in [0,1]$ such that $\zeta(1-\theta_{X_0})=\frac{1}{2}$ and obtain the error bound
	\begin{equation*}
		\| X(t_N)-X^N_{h,\tau} \|_{L^p(\Omega;H)}
		\leq C (1+\ln(N)) \big( \tau^{\min(\theta_{X_0}, \theta_f, \theta_B)} + h^{2\min(\theta_{B}, \theta_{U})+1} \big) 
		+ C h^{2\min(\theta_{X_0}, \frac{1}{2}) + 1}.
	\end{equation*}
	We additionally included $Ct_n^{-\zeta(1-\theta)} h^{2\theta+2\zeta(1-\theta)}$ in the error bound to include the case of early time steps where $t_n^{-\zeta(1-\theta)}$ cannot be bounded independently of $\tau$. In this case, the parabolic smoothing has not decreased the error coming from the initial value. To avoid the pole containing $t_n$, we can choose $\zeta = 0$ and find
	\begin{equation*}
		\| X(t_n)-X^n_{h,\tau} \|_{L^p(\Omega;H)}
		\leq C (1+\ln(n)) \big( \tau^{\min(\theta_{X_0}, \theta_f, \theta_B)} + h^{2\min(\theta_{B}, \theta_{U})+1} \big) 
		+ C h^{2 \theta_{X_0} }
	\end{equation*}
	or we keep the pole in the form of
	\begin{equation*}
		\| X(t_n)-X^n_{h,\tau} \|_{L^p(\Omega;H)}
		\leq C (1+\ln(n)) \big( \tau^{\min(\theta_{X_0}, \theta_f, \theta_B)} + h^{2\min(\theta_{B}, \theta_{U})+1} \big) 
		+ C t_n^{-\frac{1}{2}} h^{2\theta_{X_0} + 1 }
	\end{equation*}
	to keep the optimal spatial convergence rate.
\end{remark}

\begin{proof}[Proof of Theorem~\ref{thm:result}]
	For $\zeta \in [0,1]$, we combine Lemmas~\ref{lem:conv_in}, \ref{lem:conv_f}, and~\ref{lem:conv_W}. This then leads to
	\begin{align*}
		&\| X(t_n)-X^n_{h,\tau} \|_{L^p(\Omega;H)}^2\\
		&\leq C \| \Gamma_{X_0} \|_{L^p(\Omega;H)}^2 
		+ C \| \Gamma_f \|_{L^p(\Omega;H)}^2 
		+ C \| \Gamma_B \|_{L^p(\Omega;H)}^2 \\
		&\leq C (1+\ln(n))^{2} \big( \tau^{2\min(\theta_{X_0}, \theta_f, \theta_B)} + h^{2(2\min(\theta_{B}, \theta_{U})+1)} \big) 
		+ C t_n^{- 2\zeta(1-\theta_{X_0})} h^{2 (2\theta_{X_0} + 2\zeta(1-\theta_{X_0})) }\\
		&\quad 
		+ h^{4}
		+ C \tau \sum_{k=0}^{n-1} \| X(t_k) - X^k_{h,\tau} \|_{L^p(\Omega;H)}^2.
	\end{align*}
	Using the discrete Gr\"onwall's inequality (Lemma~\ref{gron}) and taking the square root, we obtain the claimed result
	\begin{align*}
		\| X(t_n)-X^n_{h,\tau} \|_{L^p(\Omega;H)}
		&\leq C (1+\ln(n)) \big( \tau^{\min(\theta_{X_0}, \theta_f, \theta_B)} + h^{2\min(\theta_{B}, \theta_{U})+1} \big) \\
		&\quad + C  t_n^{- \zeta(1-\theta_{X_0})} h^{2\theta_{X_0} + 2\zeta(1-\theta_{X_0})} .
	\end{align*}
\end{proof}

The previous theorem does not yet provide an error bound that matches that of the Euler–-Maruyama method (without splitting). Under additional assumptions, however, this gap can be closed.
For splitting methods, it is well known that the splitting error vanishes when the involved operators commute, see for example \cite[Section~IV.1.1]{HundsdorferVerwer.2003}. This commutativity is exploited in Lemma~\ref{lem:high_conv} to derive an improved bound for the splitting error $\e^{-t_n  A_h} P_h - S_{h,\tau}^{n-1}(I+\tau A_{h,2})^{-1}(I+\tau A_{h,1})^{-1} P_h$.
In addition, we assume that the operators are self-adjoint. This assumption mainly serves to simplify the proof. It ensures that the appearing operators are sectorial and allows for a straightforward application of the functional calculus theorem. If these arguments can be established by other means, it may be possible to relax this assumption.

\begin{theorem}\label{thm:result_selfadjoint}
	Let Assumptions~\ref{ass:A}--\ref{ass:B}, \ref{ass:proj_err}, \ref{ass:Ah}, and \ref{ass:Proj_error_U} be fulfilled and let $A_{h,1}$ and $A_{h,2}$ additionally be self-adjoint and commute with each other. For every $n \in \{1, \dots, N\}$, it follows that 	
	\begin{align*}
		\| X(t_n)-X^n_{h,\tau} \|_{L^p(\Omega;H)}
		&\leq C t_n^{- \zeta(1-\theta_{X_0})} \big( (1+\ln(n))^{\theta_{X_0}} \tau^{\theta_{X_0} + \zeta(1-\theta_{X_0})} + h^{2\theta_{X_0} + 2\zeta(1-\theta_{X_0})}\big)\\
		&\quad + C (1 + \ln(n))^{\frac{1}{2}+\theta_B} \big( \tau^{\theta_B + \frac{1}{2}} + h^{2\theta_B + 1}\big)
		+ C \big(\tau^{\min(\theta_{X_0}, \frac{1}{2}) } + h^{2\theta_{U}+1}\big),
	\end{align*}
	where $\zeta \in [0,1]$.
\end{theorem}

\begin{proof}
	The proof essentially follows analogously to the proof of Theorem~\ref{thm:result}. The difference is that we exchange Lemma~\ref{lem:Split_error} by Lemma~\ref{lem:Split_error_B} in Lemmas~\ref{lem:conv_in}--\ref{lem:conv_W}. More precisely, the result of Lemma~\ref{lem:conv_in} becomes
	\begin{equation*}
		\| \Gamma_{X_0} \|_{L^p(\Omega;H)} 
		\leq C t_n^{- \zeta(1-\theta_{X_0})} \big( (1+\ln(n))^{\theta_{X_0}} \tau^{\theta_{X_0} + \zeta(1-\theta_{X_0})} + h^{2\theta_{X_0} + 2\zeta(1-\theta_{X_0})}\big).
	\end{equation*}
	To adjust Lemma~\ref{lem:conv_f}, we apply Lemma~\ref{lem:Split_error_B} with $\theta = 0$ and $\zeta=1$ instead of Lemma~\ref{lem:Split_error} for $\Gamma_{f,3}$. An alternative bound to Lemma~\ref{lem:conv_f} then is
	\begin{equation*}
		\| \Gamma_f \|_{L^p(\Omega;H)} 
		\leq C \big(\tau^{\min(\theta_{X_0},\frac{1}{2})} + (1+\ln(n)) (\tau + h^{2})\big) + C \tau\sum_{k=0}^{n-1} \| X(t_k)-X^k_{h,\tau} \|_{L^p(\Omega;H)}.
	\end{equation*}
	Additionally, we bound $ \Gamma_{B,3} $ in Lemma~\ref{lem:conv_W} using Lemma~\ref{lem:Split_error_B} with $\theta = \theta_B$, $\zeta = (2(1-\theta_B))^{-1}$ instead of Lemma~\ref{lem:Split_error}. This leads to an updated bound of Lemma~\ref{lem:conv_W} of the form
	\begin{align*}
		\| \Gamma_B \|_{L^p(\Omega;H)}^2 
		&\leq 
		C (1+\ln(n))^{1+2\theta_B} \big( \tau^{2(\theta_B + \frac{1}{2}) } + h^{2(2\theta_B + 1)}\big)
		+ C \big(\tau^{2\min(\theta_{X_0}, \frac{1}{2}) } + h^{2(2\theta_{U}+1)}\big)\\
		&\quad + C \tau\sum_{k=0}^{n-1} \| X(t_k)-X^k_{h,\tau} \|_{L^p(\Omega;H)}^2.
	\end{align*}
	The last step to prove the desired bound is to combine the three error terms, apply Gr\"onwall's inequality and taking the square root.
\end{proof}

\section{Example: A fully discretized domain decomposition scheme}\label{sec:dG}

In this section, we exemplify the theoretical results. For the abstract Equation~\eqref{eq:SPDE}, we state a more concrete SPDE in Section~\ref{subsec:cont_Prob} and verify that the equation fits in the framework from Section~\ref{sec:Prob_description}. Following the problem description, we state the space discretization in Section~\ref{subsec:dG_Notation} followed by the full discretization in Section~\ref{subse:full_disc}. In both sections, we verify that the assumptions stated in Sections~\ref{subsec:space_discretization} and \ref{subsec:time_discretization} are fulfilled.

\subsection{Continuous problem}
\label{subsec:cont_Prob}
We consider the semi-linear stochastic diffusion equation
\begin{equation}\label{eq:model}
	\begin{cases}
		\diff{X}\coord = \big[\nabla\cdot (K(\mathbf{x}) \nabla X\coord )
		+f(t,X\coord) \big] \diff{t} + B(t,X\coord) \diff{W}(t,\mathbf{x}), \hspace{-4pt}& \coord\in (0,t_f]\times \D;\\
		X\coord=0, & \coord \in [0,t_f] \times \partial\D;\\
		X(0,\mathbf{x})=X_0(\mathbf{x}),& \mathbf{x}\in \D,
	\end{cases}
\end{equation}
where $\D \subset \R^d$, $d \in \N$, is an open, convex polygon, and $t_f \in \R^+$.
The linear operator $A$ is given through
\begin{equation}\label{eq:def_A}
	A v(\mathbf{x}) = - \nabla\cdot\left(K(\mathbf{x}) \nabla v(\mathbf{x}) \right), \quad \mathbf{x} \in \D,
\end{equation}
where the matrix-valued function $K$ is Lipschitz continuous w.r.t. $\mathbf{x}$, fulfills
\begin{equation}\label{eq:def_K_sym}
	K \in \R^{d \times d} \quad \text{ is symmetric } 
\end{equation}
and the eigenvalues of $K(\mathbf{x})$ lie in the interval
\begin{equation}\label{eq:def_K_eigenvalues}
	[K_0,K_1] \quad \text{for } K_0,K_1 \in \R^+ \text{ for almost every } \mathbf{x} \in \D.
\end{equation}
We interpret $A$ as an unbounded operator in $L^2(\D)$, i.e.~$A \colon \dom(A)\subset L^2(\D)\rightarrow L^2(\D)$, where the domain of $A$ is given by $\dom(A) = \{ v \in H^1_0(\D) : Av \in L^2(\D)\} = H^2(\D)\cap H^1_0(\D)$. Note that $\dom(A)$ is a dense subset of $L^2(\D)$ and that $ \| v \|_{H^2(\D)} \leq C \| Av \|_{L^2(\D)}$ for all $v \in \dom(A)$, compare \cite[Theorem~9.24]{Hackbusch2017} in combination with \cite[Theorem~1.4.3.]{Grisvard.2011}. 
For more details on Sobolev spaces and their norms, we refer the reader to \cite[Chapter~6.2]{Hackbusch2017}.

For our method, we want to split $A$ into two parts $A_1$ and $A_2$. As a decomposition of the operator $A$, we choose a domain decomposition. To specify this, we choose two overlapping subdomains $\{\D_{\ell}\}_{{\ell}=1}^{2}$ of $\D$ with a Lipschitz boundary. The union of the two subdomains is $\D$ again. On these subdomains, we define the non-negative weight functions $\{\chi_{\ell}\}_{{\ell}=1}^{2}\subset W^{1,\infty}(\D)$. The support of $\chi_{\ell}$ is $\D_{\ell}$ and the two weight functions form a partition of unity on $\D$ and fulfill $\| \chi_{\ell}\|_{L^{\infty}(\D)} \leq 1$ for $\ell \in \{1,2\}$. Furthermore, we assume that the weight functions are piecewise linear. We define 
\begin{equation}\label{eq:Asplit}
	A_{\ell} v(\mathbf{x}) = - \nabla\cdot\left( \chi_{\ell} (\mathbf{x}) K(\mathbf{x})\nabla v(\mathbf{x}) \right), \quad \ell \in \{1,2\}, \quad \mathbf{x} \in \D,
\end{equation}
for $v\in \dom(A_\ell) =\{u\in H^1_0(\D):A_{\ell}u\in L^2(\D)\}$. Note that $\dom(A) \subseteq \dom(A_{\ell})$ is fulfilled. As the set $\dom(A_{\ell})$ does not appear in the analysis of the fully discretized method, we do not introduce this in much detail here but refer the reader to \cite[Equation~(2.8) and Lemma~2.4]{HansenHenningsson.2017} for more details.

The operators $A$, $A_1$, and $A_2$ fit into the setting of Assumption~\ref{ass:A}, where $H = L^2(\D)$. First, we note that since $\chi_1 + \chi_2 = 1$ on $\D$, it follows that $A_1 + A_2 = A$. 
Moreover, we can show that Assumption~\ref{ass:A_sect} is fulfilled. Since $K$ is positive definite, the operator $A$ is positive, and $-A$ generates an analytical contraction semigroup, compare \cite[Chapter~7, Theorem 2.7]{Pazy1983}. Furthermore, in our example, the operator $A$ is additionally self-adjoint. Note that only the constant of Assumption~\ref{ass:A_sect} is dependent on $\| \nabla \chi_{\ell} \|_{L^{\infty}(\D)^d}$, which can be linked to the overlap size of the domain decomposition.

\begin{lemma}\label{lem:AellAinv}
	Let $A$ and $A_{\ell}$, $\ell \in \{1,2\}$, be given as in \eqref{eq:def_A} and \eqref{eq:Asplit}, respectively. Then it follows that
	\begin{equation*}
		\| A_\ell A^{-1} \|_{\mathcal{L}(L^2(\mathcal{D}))}
		\leq 1+C \| \nabla \chi_{\ell} \|_{L^{\infty}(\mathcal{D})^d },
	\end{equation*}
	where $C$ does not depend on $\chi_{\ell}$. This verifies that Assumption~\ref{ass:op_exch} is fulfilled. 
\end{lemma}

\begin{proof}
	For the proof, we look at the operator norm of $A_{\ell}$ and $A^{-1}$ separately. More precisely, we can bound the operator norm of the product by 
	\begin{equation*}
		\| A_\ell A^{-1} \|_{\mathcal{L}(L^2(\mathcal{D}))}
		\leq\| A_\ell\|_{\mathcal{L}(\dom(A);L^2(\mathcal{D}))} \| A^{-1}\|_{\mathcal{L}\left(L^2(\mathcal{D});\dom(A)\right)},
	\end{equation*}
	where we equip $\dom(A)$ with the graph norm $\|\cdot\|_{L^2(\mathcal{D})}+\| A\cdot\|_{L^2(\mathcal{D})}$.
	First, we note that the operator $A^{-1} \colon L^2(\mathcal{D}) \to \dom(A)$ is bounded. It remains to show that $A_\ell \colon \dom(A) \to L^2(\mathcal{D})$ is bounded. For all $v\in \dom(A)$, it follows that 
	\begin{align*}
		\frac{\| A_{\ell}v \|_{L^2(\mathcal{D})}}{\|v\|_{L^2(\mathcal{D})}+\| Av\|_{L^2(\mathcal{D})}}
		&\leq \frac{\| \nabla \cdot (\chi_{\ell} K \nabla v )\|_{L^2(\mathcal{D})}}{\|v\|_{L^2(\mathcal{D})}+\| Av\|_{L^2(\mathcal{D})}}\leq \frac{\| \chi_{\ell}Av\|_{L^2(\mathcal{D})}+\| \nabla \chi_{\ell} \cdot K\nabla v\|_{L^2(\mathcal{D})}}{\|v\|_{L^2(\mathcal{D})}+\| Av\|_{L^2(\mathcal{D})}}\\
		&\leq \frac{\| Av\|_{L^2(\mathcal{D})} + \| \nabla\chi_{\ell}\|_{L^{\infty}(\D)^d}\| K\nabla v\|_{L^2(\mathcal{D})^d}}{\|v\|_{L^2(\mathcal{D})}+\| Av\|_{L^2(\mathcal{D})}}\\ 
		&\leq 1+\frac{K_1\| \nabla\chi_{\ell}\|_{L^{\infty}(\D)^d} \| \nabla v\|_{L^2(\mathcal{D})^d}}{\|v\|_{L^2(\mathcal{D})}+\| Av\|_{L^2(\mathcal{D})}}
		\leq 1+C\| \nabla \chi_{\ell}\|_{L^{\infty}(\mathcal{D})^d},
	\end{align*}
	where we used in the last step that $\dom(A)$ is continuously embedded in $H^1_0(\D)$ and $\| \nabla \cdot \|_{L^2(\D)^d}$ is an equivalent norm in $H^1_0(\D)$.
\end{proof}

Moreover, for $\theta_f, \theta_{B} \in [0,\frac{1}{2})$, $\theta_{X_0} \in [\max(\theta_{f}, \theta_{B}), \theta_{B} + \frac{1}{2})$, and $U=L^2(\D)$, let $X_0$, $f$, $W$, and $B$ satisfy Assumptions~\ref{ass:X0}--\ref{ass:B}. 

\subsection{Discontinuous Galerkin spatial discretization}\label{subsec:dG_Notation}

When discretizing the SPDE \eqref{eq:model} in space, we choose the discontinuous Galerkin method, more specifically, the symmetric interior penalty  Galerkin (SIPG) method. In the following, we give a short introduction and explanation of the notation used. For further information, we refer the reader to \cite{DiPietro2012, Riviere}. 

We discretize the spatial domain $\D$ with the mesh $\mathcal{T}_h$, which contains the elements $T \in \mathcal{T}_h$. The elements $T$ are affine mappings of a reference element $\hat{T}$, which is a convex polyhedron in $\R^d$ with $n_0$ faces.
In the following, we denote the diameter of an element $T$ by $h_T$. The value $h$ indicates the maximum of the diameters $h_T$ of elements in the mesh $\mathcal{T}_h$. We assume that the mesh family $\{\mathcal{T}_{h_i}\}_{h\in I}$ with $I\subset \R^+$, fulfills the conditions of \cite[Lemma~1.62]{DiPietro2012}, and there is a constant $C \in \R^+$, independent of $h$, such that $h_T\geq Ch$ for all $T\in\mathcal{T}_h$. Moreover, we assume that $\chi_{\ell}\vert_T$, $\ell \in \{1,2\}$, is linear on every $T\in\mathcal{T}_h$ for every $h$. 
The collection of the edges $e$ of elements $T$ in $\mathcal{T}_h$ is denoted by $F_h$. As one edge can belong to two elements of $\mathcal{T}_h$ at the same time, those two corresponding elements are denoted by $T_e^+$ and $T^-_e$, where the choice of notation of the two elements is arbitrary but after that the choice is consistent. If $e \subset \partial \D$, there is only one element $T_e$ associated with this particular edge. In the following, we denote $h_e = \min(h_{T_e^+}, h_{T_e^-})$ if $e$ does not lie on the boundary $\partial \D$ and $h_e = h_{T}$ for $e \subset \partial \D$.
The discretization space is defined as
\begin{equation*}
	V_h=\{v_h\in L^2(\D):v_h|_{T}\in \mathbb{P}^{\orderPoly}(T) \text{ for all } T\in\mathcal{T}_h \}, 
\end{equation*}
where $\mathbb{P}^\orderPoly(T)$ denotes the polynomials of at most order $\orderPoly$ on the element $T$, and $v_h\vert_{T}$ is the restriction of $v_h$ to an element $T$. 
An example for a basis for the space $V_h$ can be put together by defining polynomial bases on each element $T \in \mathcal{T}_h$ such that every function has its support only in one element $T$. When considering two basis functions $\varphi_{i}$ and $\varphi_{j}$ of $V_h$, the inner product $\inner[H]{\varphi_{i}}{\varphi_{j}}$ can only be nonzero if the functions' support lies on the same element $T$. Thus, after choosing a suitable order of the basis elements, the mass matrix $(\mathbf{M})_{i,j}  = \inner[H]{\varphi_{i}}{\varphi_{j}}$ obtains a block structure. The blocks represent the elements $T \in \mathcal{T}_h$.

Moreover, for $\nu \in \{0,1,2\}$, we introduce the following broken space $H^{\nu}(\mathcal{T}_h )$ on the partition $\mathcal{T}_h$, compare with \cite[Section 2.3]{Riviere},
\begin{equation*}
	H^{\nu}(\mathcal{T}_h ) 
	= \{v\in L^2(\D) : v\vert_{T}\in H^{\nu}(T) \text{ for all } T\in\mathcal{T}_h\}
	\text{ with } 
	\verttt v \verttt_{H^{\nu}(\mathcal{T}_h )}
	=\Big(\sum_{T\in\mathcal{T}_h} \|v\vert_{T}\|_{H^{\nu}(T)}^2 \Big)^{\frac{1}{2}}.
\end{equation*}
Note that $H^0(\mathcal{T}_h )$ is equal to the Lebesgue space $L^2(\D)$ and that for $v \in H^{\nu}(\D)$, we find that $\verttt v \verttt_{H^{\nu}(\mathcal{T}_h )} = \| v \|_{H^{\nu}(\D )}$. Additionally, for $v\in\dom(A)$ no jumps appear on the edges $e \in F_h$, compare \cite[Lemma~4.3]{DiPietro2012}.
On an element $T$, we can then make the following statement for the $L^2(\D)$-projection operator $P_h$ that maps $L^2(\D)$ on $V_h$.

\begin{lemma}\label{lem:Projerr}
	For all $\nu\in\{0,1,2\}$, $m \in \{0,\dots,\nu\}$, $T\in\mathcal{T}_h$, and $v\in H^{\nu}(T),$ it follows that
	\begin{equation*}
		\|\left(I-P_h\right)v\|_{H^m(T)}\leq C h_T^{\nu-m} \|v\|_{H^{\nu}(T)},
	\end{equation*}
	where $C$ is independent of $h$ and $T$.
\end{lemma}

The proof of this lemma follows from \cite[Lemma~1.58]{DiPietro2012}. 
With the help of this lemma, the projection error of $P_h$ can be bounded as stated in Assumption~\ref{ass:proj_err}, i.e.,~we observe that
\begin{align*}
	\| (I-P_h) v \|_{L^2(\D)}
	&= \Big(\sum_{T \in \mathcal{T}_h} \| (I-P_h) v \|_{L^2(T)}^2\Big)^{\frac{1}{2}}
	\leq C \Big(\sum_{T \in \mathcal{T}_h} h_T^4 \| v \|_{H^2(T)}^2\Big)^{\frac{1}{2}}\\
	&\leq C h^2 \| v \|_{H^2(\D)}
	\leq C h^2 \| A v \|_{L^2(\D)}
\end{align*}
is fulfilled for every $v \in H^2(\D) \cap H_0^1(\D) = \dom(A)$.
To show that $V_h$ fits into Assumption~\ref{ass:Ah}, we first need to state some additional definitions.
Functions from the discretization space $V_h$ can have jumps on the edges $e \in F_h$. Therefore, we define the jump function and average function on an interior edge $e$ as
\begin{equation}\label{eq:jump_average}
	[v]\vert_{e}=v\vert_{T_e^+}-v\vert_{T_e^-} 
	\quad \text{and} \quad 
	\{v\}\vert_{e}=\half\big(v_{T_e^+}+v_{T_e^-}\big),
\end{equation}
where $T_e^+$ and $T_e^-$ are the two elements that share the edge $e$. 
On an edge on the boundary $\partial \D$, the functions are defined as
\begin{equation}\label{eq:jump_average_boun}
	\{v\}\vert_e=v\vert_{T_e}=[v]\vert_e.
\end{equation}
Furthermore, we also use the following semi-norm related to jumps on the faces. For $v \in L^2(\D)$, this norm is denoted by
\begin{equation*}
	|v|_{J_h}=\Big(\sum_{e\in F_h} \frac{1}{h_e} \|[v]\|^2_{L^2(e)}\Big)^{\half}. 
\end{equation*}
To show that Assumption~\ref{ass:Ah} is fulfilled, we need to introduce a suitable norm as mentioned in \eqref{eq:normV_h}. We choose the norm that is induced by the inner product
\begin{equation*}
	\inner[V_h]{v_h}{w_h}
	= \sum_{T\in\mathcal{T}_h} \inner[H^{1}(T)]{v_h\vert_{T}}{w_h\vert_{T}} + \sum_{e\in F_h} \frac{1}{h_e} \inner[L^2(e)]{[v_h]}{[w_h]} \quad \text{for all }v_h, w_h \in V_h
\end{equation*}
and then given by
\begin{equation}\label{eq:defNormVh}
	\lno\cdot\rno_{V_h}= \big(\verttt\cdot \verttt_{H^{1}(\mathcal{T}_h )}^2+|\cdot|_{J_h}^2\big)^{\frac{1}{2}}.
\end{equation}
For $\sigma \in \R_0^+$ and $K$ as stated in \eqref{eq:def_K_sym}--\eqref{eq:def_K_eigenvalues}, we can now define the discretized operator $A_h\colon V_h \to V_h$ via its corresponding bilinear form $a_h$ through $(A_{h} v_h,w_h) = a_{h}(v_h,w_h)$ for all $v_h, w_h \in V_h$, where 
\begin{equation} \label{eq:A_dG_formulation}
	\begin{split}
		a_{h}(v_h,w_h)
		&:= \sum_{T\in\mathcal{T}_h} \left( K\nabla v_h,\nabla w_h\right)_{L^2(T)^d}
		- \sum_{e\in F_h}\int_e \{K\nabla v_h\cdot \mathbf{n}_e\}[w_h] \diff{\xi}\\
		&\quad - \sum_{e\in F_h}\int_e \{K\nabla w_h\cdot \mathbf{n}_e\}[v_h] \diff{\xi}
		+ \sum_{e\in F_h}\int_e \frac{\sigma}{h}[v_h][w_h] \diff{\xi}
	\end{split}
\end{equation}
for all $v_h,w_h\in V_h$. Here, $\mathbf{n}_e$ is the outer pointing normal derivative such that $\mathbf{n}_e := \mathbf{n}_{T_e^+} =  - \mathbf{n}_{T_e^-}$ for an inner edge $e$ and $\mathbf{n}_e := \mathbf{n}_{T_e}$ for an edge $e \subset \partial T$.
In \eqref{eq:A_dG_formulation}, we see four relevant terms in the dG formulation: the weak formulation of $A$, the consistency term, the symmetry term, and the penalty term. For a detailed derivation, we refer the reader to \cite[Chapter~4.2.1.1--4.2.1.3]{DiPietro2012} and \cite[Equation~(4.20)]{DiPietro2012} for \eqref{eq:normV_h}. Note also that the discrete operator weakly enforces the Dirichlet boundary condition. The parameter $\sigma$ is chosen such that $A_h$ is strongly positive.

\begin{lemma} \label{lem:Ah_pos_bounded} 
	Let $A_{h} \colon V_h \to V_h$ be given as in \eqref{eq:A_dG_formulation}. Further, let $n_0$ be the maximal number of vertices of an element $T \in \mathcal{T}_h$.
	For $\sigma > \constChi n_0$, where $\constChi$ is from Lemma~\ref{lem:disctra}, it follows that $A_{h}$ is strongly positive and bounded, where the norm $\|\cdot\|_{V_h}$ is given in \eqref{eq:defNormVh}. This ensures that Assumptions~\ref{ass:Ah}~\ref{ass:Ah_coercive}--\ref{ass:Ah_bound} are fulfilled.
\end{lemma}

This statement follows from the lemmas from \cite[Lemmas~4.12 and 4.16]{DiPietro2012} which we combine with \cite[Lemma~4.20]{DiPietro2012}.
The discretizations $f_h \colon [0,t_f] \times L^2(\D) \to V_h$ and $B_h \colon [0,t_f] \times L^2(\D) \to L_2^0$ of $f$ and $B$ are given as in \eqref{eq:def_fh} and \eqref{eq:def_Bh}, respectively, where $P_U$ is chosen such that Assumption~\ref{ass:Proj_error_U} is satisfied.

\subsection{Full discretization} \label{subse:full_disc}

The operator $A_h$ stated in \eqref{eq:A_dG_formulation} is now split into two separate parts. As a decomposition of the operator $A$, we choose a domain decomposition method. We then define $A_{h,{\ell}} \colon V_h \to V_h$ that fulfills $(A_{h,{\ell}}v_h,w_h) = a_{h,{\ell}}(v_h,w_h)$ for all $v_h, w_h \in V_h$, where 
\begin{equation}\label{eq:Ahsplit}
	\begin{split}
		a_{h,{\ell}}(v_h,w_h)
		&:=\sum_{T\in\mathcal{T}_h}\left(\chi_{\ell}K\nabla v_h,\nabla w_h\right)_{L^2(T)^d}
		-\sum_{e\in F_h}\int_e \{\chi_{\ell}K\nabla v_h\cdot \mathbf{n}_e\}[w_h] \diff{\xi}\\
		&\quad -\sum_{e\in F_h}\int_e \{\chi_{\ell}K\nabla w_h\cdot \mathbf{n}_e\}[v_h] \diff{\xi}
		+\sum_{e\in F_h}\int_e \chi_{\ell}\frac{\sigma}{h_e}[v_h][w_h] \diff{\xi}
	\end{split}
\end{equation}
for $\ell \in \{1,2\}$. Note that the flux is zero on the interior boundary of $\mathcal{D}_{\ell}$, the boundary part of $\mathcal{D}_{\ell}$ that does not overlap with the boundary of $\mathcal{D}$.
In the following, we go through Assumption~\ref{ass:Ah}~\ref{ass:Ah_sum}--\ref{ass:dissa_exch} and prove that $A_{h}$, $A_{h,{1}}$, and $A_{h,{2}}$ fit into the setting.
The sum property $A_{h} = A_{h,{1}} + A_{h,{2}}$ from Assumption~\ref{ass:Ah}~\ref{ass:Ah_sum} is a direct consequence form the sum property of $\chi_1$ and $\chi_2$. This can easily be observed after inserting $1 = \chi_1 + \chi_2$ into \eqref{eq:Ahsplit}.
The next step is to prove the non-negativity of a split operator $A_{h,\ell}$, $\ell \in \{1,2\}$, i.e.~Assumption~\ref{ass:Ah}~\ref{ass:Ah_pos}. Before we turn to the proof of this result, in the same fashion as \cite[Section~2.7.1]{Riviere} , we first provide an auxiliary result that is of use in the proof for non-negativity. For this, we keep track of the constants because they give restrictions on the parameter $\sigma$.

\begin{lemma}\label{lem:chidisc}
	For every $T\in\mathcal{T}_h$, every $e \in F_h$ such that $e \subset \partial T$, the outward pointing normal derivative $\mathbf{n}_e = \mathbf{n}_{\partial T}|_e$ and $v_h\in V_h$, it follows that
	\begin{equation*}
		\big\| \chi_{\ell}^{\half}\nabla v_h\vert_{T} \big\|_{L^2(e)^d}^2
		\leq \constChi h_T^{-1} \big\| \chi_{\ell}^{\frac{1}{2}}\nabla v_h \big\|_{L^2(T)^d}^2, \quad \ell \in \{1,2\},
	\end{equation*}
	where the constant $\constChi$ is stated in Lemma~\ref{lem:disctra}.
\end{lemma}

\begin{proof}
	The main idea of this proof is to apply Lemma~\ref{lem:disctra}. We cannot directly apply this lemma, though, as a function $\chi_{\ell}^{\half}\nabla v_h\vert_T$ is not necessarily a polynomial. Note that in the $L^2(e)$-norm, a square appears which cancels out the square root. Thus, we rewrite the $L^2$-norm as an $L^1$-norm. This then enables us to apply Lemma~\ref{lem:disctra}, and we obtain
	\begin{align*}
		\big\| \chi_{\ell}^{\half}\nabla v_h\vert_T \big\|^2_{L^2(e)^d}
		&=\lno \chi_{\ell}\nabla v_h\vert_T: \nabla v_h\vert_T\rno_{L^1(e)^d}\\
		&\leq \constChi h_T^{-1}\lno \chi_{\ell}\nabla v_h\vert_T: \nabla v_h\vert_T\rno_{L^1(T)^d}
		=\constChi h_T^{-1} \big\| \chi_{\ell}^{\half}\nabla v_h\vert_T \big\|_{L^2(T)^d}^2,
	\end{align*}
	where the colon $:$ denotes the element-wise multiplication between two vectors.
\end{proof}
\begin{lemma}\label{lem:dissip}
	Let $A_{h,\ell} \colon V_h \to V_h$ be given as in \eqref{eq:Ahsplit}. Further, let $K$ be as stated in \eqref{eq:def_K_sym}--\eqref{eq:def_K_eigenvalues} and $n_0$ the maximal number of vertices of an element the elements $T \in \mathcal{T}_h$.
	For $\sigma \geq \constChi  K_1^2 K_0^{-1} n_0$, where $\constChi$ is from Lemma~\ref{lem:disctra}, and $\ell \in \{1,2\}$, it follows that $A_{h,\ell}$ is non-negative on $V_h$ with respect to the $L^2(\D)$-norm, i.e., for all $v_h \in V_h$ it holds that $\left(A_{h,\ell}v_h,v_h\right)_{L^2(\D)}\geq 0$. This ensures that Assumption~\ref{ass:Ah}~\ref{ass:Ah_pos} is fulfilled.
\end{lemma}

\begin{proof}
	For all $v_h\in V_h$ it follows that
	\begin{align*}
		&\left(A_{h,\ell}v_h,v_h\right)_{L^2(\D)}\\
		&=\sum_{T\in\mathcal{T}_h} \big(\chi_{\ell}^{\half} K^{\half} \nabla v_h , \chi_{\ell}^{\half} K^{\half} \nabla v_h \big)_{L^2(T)^d}
		-2\sum_{e\in F_h}\int_e\{\chi_{\ell}^{\half} K\nabla v_h\cdot \mathbf{n}_e\}[\chi_{\ell}^{\half}v_h]\diff{\xi}\\
		&\quad+\sum_{e\in F_h}\int_e\frac{\sigma}{h_e}[\chi_{\ell}^{\half}v_h][\chi_{\ell}^{\half}v_h]\diff{\xi}\\
		&\geq \|\chi_{\ell}^{\half} K^{\half}\nabla v_h \|_{L^2(\D)^d}^2
		-2\sum_{e\in F_h} h_e^{\half} \big\| \{\chi_{\ell}^{\half} K\nabla v_h\cdot \mathbf{n}_e\} \big\|_{L^2(e)}h_e^{-\half} \big\| [\chi_{\ell}^{\half}v_h] \big\|_{L^2(e)}
		+\sigma \big| \chi_{\ell}^{\half}v_h \big|_{J_h}^2\\
		&=:  \Gamma_1 - \Gamma_2 + \Gamma_3 .
	\end{align*}
	In the remainder of the proof, we want to show that $ \Gamma_1 - \Gamma_2 + \Gamma_3  \geq 0$. We begin to apply Cauchy--Schwartz inequality for sums and obtain
	\begin{align*}
		\Gamma_2 
		&\leq 2 \Big(\sum_{e\in F_h}h_e \big\| \{\chi_{\ell}^{\half} K\nabla v_h\cdot \mathbf{n}_e\} \big\|_{L^2(e)}^2 \Big)^{\half}
		\Big(\sum_{e\in F_h}h_e^{-1} \big\| [\chi_{\ell}^{\half}v_h] \big\|_{L^2(e)}^2 \Big)^{\half}\\
		&= 2 \Big(\sum_{e\in F_h}h_e \big\| \{\chi_{\ell}^{\half} K \nabla v_h\cdot \mathbf{n}_e\} \big\|_{L^2(e)}^2 \Big)^{\half}|\chi_{\ell}^{\half}v_h|_{J_h}.
	\end{align*}
	If $e$ does not lie on $\partial\D$, we insert the definition of the average from  \eqref{eq:jump_average}. For such an $e$ we can apply Lemma~\ref{lem:chidisc} and obtain
	\begin{align*}
		h_e^{\frac{1}{2}} \big\| \{\chi_{\ell}^{\half}  K \nabla v_h\cdot \mathbf{n}_e\} \big\|_{L^2(e)}
		&= \frac{h_e^{\frac{1}{2}}}{2} \big\| \big(\chi_{\ell}^{\half}  K \nabla v_h\vert_{T^+_e}\cdot \mathbf{n}_e + \chi_{\ell}^{\half} K \nabla v_h\vert_{T_e^-}\cdot \mathbf{n}_e \big) \big\|_{L^2(e)}\\
		&\leq \frac{h_e^{\frac{1}{2}} K_1 }{2} \big\| \chi_{\ell}^{\half}  \nabla v_h\vert_{T^+_e} \big\|_{L^2(e)^d} 
		+ \frac{h_e^{\frac{1}{2}} K_1}{2} \big\| \chi_{\ell}^{\half}  \nabla v_h\vert_{T_e^-} \big\|_{L^2(e)^d}\\
		&\leq \frac{\constChi^{\frac{1}{2}} K_1}{2} \Big( \Big(\frac{h_e}{h_{T^+_e}}\Big)^{\frac{1}{2}} \big\| \chi_{\ell}^{\half}  \nabla v_h \big\|_{L^2(T^+_e)^d}
		+ \Big(\frac{h_e}{h_{T^-_e}}\Big)^{\frac{1}{2}} \big\| \chi_{\ell}^{\half}  \nabla v_h \big\|_{L^2(T^-_e)^d} \Big)\\
		&\leq \frac{\constChi^{\frac{1}{2}} K_1}{2} \Big(\big\| \chi_{\ell}^{\half}  \nabla v_h \big\|_{L^2(T^+_e)^d}
		+ \big\| \chi_{\ell}^{\half}  \nabla v_h \big\|_{L^2(T^-_e)^d} \Big),
	\end{align*}
	where we use $h_e = \min(h_{T_e^+}, h_{T_e^-})$ in the last step.
	In a similar way, for an edge $e \subset \partial \D$, we apply the definition of the average on the boundary \eqref{eq:jump_average_boun} and Lemma~\ref{lem:chidisc} to find
	\begin{equation*}
		h_e^{\frac{1}{2}} \big\| \{\chi_{\ell}^{\half} K \nabla v_h\cdot \mathbf{n}_e\} \big\|_{L^2(e)}
		\leq \constChi^{\frac{1}{2}} K_1 \big\| \chi_{\ell}^{\half}  \nabla v_h \big\|_{L^2(T_e)^d}.
	\end{equation*}
	When we reorder the  sum, we apply the fact that every element $T$ in $\mathcal{T}_h$ is counted $n_0$ times in total after summing over $T_e^+$ and $T_e^-$ for all edges $e \in F_h$. We can therefore bound $ \Gamma_2 $ by
	\begin{align*}
		\Gamma_2
		&\leq 2\constChi^{\frac{1}{2}}K_1  \Big(\sum_{e\in F_h, e \not\subset \partial \D} \Big(\big\| \chi_{\ell}^{\half}  \nabla v_h \big\|_{L^2(T^+_e)^d}^2
		+\big\| \chi_{\ell}^{\half}  \nabla v_h \big\|_{L^2(T^-_e)^d}^2\Big)\\
		&\quad +\sum_{e\in F_h, e \subset \partial \D} \big\| \chi_{\ell}^{\half}  \nabla v_h \big\|_{L^2(T_e)^d}^2
		 \Big)^{\half}  |\chi_{\ell}^{\half}v_h|_{J_h}\\
		&\leq 2 \constChi^{\frac{1}{2}} K_1 K_0^{-\half} n_0^{\frac{1}{2}} \Big(\sum_{T \in \mathcal{T}_h} \big\| \chi_{\ell}^{\half} K^{\frac{1}{2}} \nabla v_h \big\|_{L^2(T)^d}^2 \Big)^{\half}|\chi_{\ell}^{\half}v_h|_{J_h}\\
		&= 2 \constChi^{\frac{1}{2}} K_1 K_0^{-\half} n_0^{\frac{1}{2}} \| \chi_{\ell}^{\half} K^{\frac{1}{2}} \nabla v_h \|_{L^2(\D)^d} |\chi_{\ell}^{\half}v_h|_{J_h}.
	\end{align*}
	Applying Young's inequality for products, we obtain
	\begin{equation*}
		\Gamma_2 \leq  \| \chi_{\ell}^{\half} K^{\frac{1}{2}} \nabla v_h \|_{L^2(\D)^d}^2 + \constChi K_1^2 K_0^{-1} n_0|\chi_{\ell}^{\half}v_h|_{J_h}^2.
	\end{equation*}
	Thus, for $\sigma \geq \constChi K_1^2 K_0^{-1} n_0$, we find
	\begin{align*}
		&\left(A_{h,\ell}v_h,v_h\right)_{L^2(\D)}\\
		&\geq \| \chi_{\ell}^{\half} K^{\half}\nabla v_h \|_{L^2(\D)^d}^2
		- \| \chi_{\ell}^{\half} K^{\frac{1}{2}} \nabla v_h \|_{L^2(\D)^d}^2 
		- \constChi  K_1 K_0^{-\half} n_0|\chi_{\ell}^{\half}v_h|_{J_h}^2
		+\sigma \big| \chi_{\ell}^{\half}v_h \big|_{J_h}^2
		\geq 0,
	\end{align*}
	which finishes the proof of the lemma.
\end{proof}

\begin{remark}
	To our knowledge, no sharp bounds for $\constChi$ are known. An exception is the case $p=2$; compare \cite{Warburton2003}. Therefore, we have no sharp bound for the restriction of $\sigma$.
\end{remark}

\begin{lemma}\label{lem:Ah2}
	The operator $A_{h,\ell}$ defined in \eqref{eq:Ahsplit} fulfills
	\begin{equation*}
		\| A_{h,\ell} P_h \|_{\L(L^2(\D))} \leq Ch^{-2}.
	\end{equation*}
	Thus, Assumption~\ref{ass:Ah}~\ref{ass:Ah2} is fulfilled.
\end{lemma}

\begin{proof}
	Since $A_{h,\ell}$ is self-adjoint and therefore also $A_{h,\ell} P_h$, we can express the operator norm by
	\begin{equation*}
		\| A_{h,\ell} P_h \|_{\L(L^2(\D))} 
		= \sup_{\|v\|_{L^2(\mathcal{D})}\leq 1} | (A_{h,\ell} P_h v, v)_{L^2(\mathcal{D})}|
		= \sup_{v_h \in V_h, \|v_h\|_{L^2(\mathcal{D})}\leq 1} | (A_{h,\ell} v_h, v_h)_{L^2(\mathcal{D})}|,
	\end{equation*}
	compare \cite[Satz~V.5.7]{Werner.2000}. Inserting the definition from \eqref{eq:Ahsplit}, we find
	\begin{align*}
		| (A_{h,\ell}v_h,v_h )_{L^2(\D)}|
		&\leq \Big| \sum_{T\in\mathcal{T}_h} \big(\chi_{\ell} K \nabla v_h, \nabla v_h \big)_{L^2(T)^d} \Big|
		+ 2 \Big| \sum_{e\in F_h} \int_e\{\chi_{\ell} K\nabla v_h\cdot \mathbf{n}_e\}[ v_h]\diff{\xi}\Big|\\
		&\quad + \Big| \sum_{e\in F_h} \int_e \chi_{\ell} \frac{\sigma}{h_e}[v_h][v_h]\diff{\xi}\Big|\\
		&=:  \Gamma_1  +  \Gamma_2  +  \Gamma_3.
	\end{align*}
	We begin to bound $ \Gamma_1 $ using the Cauchy-Schwarz inequality, Lemma~\ref{lem:inv} and $\max_{T \in \mathcal{T}_h} h_T^{-1} \leq C h^{-1}$ to obtain
	\begin{equation*}
		\Gamma_1  = \Big|\sum_{T\in\mathcal{T}_h}\left(\chi_{\ell}K\nabla v_h,\nabla v_h\right)_{L^2(T)} \Big|
		\leq \sum_{T\in\mathcal{T}_h} \|\chi_{\ell} \|_{L^{\infty}(T)} \| K^{\frac{1}{2}} \nabla v_h\|_{L^2(T)^d}^2
		\leq C h^{-2} \| v_h \|_{L^2(\D)}^2.
	\end{equation*}
	We now turn our attention towards $\Gamma_{2}$. From Cauchy-Schwartz inequality, as well as, Lemmas~\ref{lem:disctra} and \ref{lem:inv}, it follows
	\begin{align*}
		\Gamma_2 
		&= 2 \Big| \sum_{e\in F_h}\int_e\{\chi_{\ell}K\nabla v_h\cdot \mathbf{n}_e\}[v_h] \diff{\xi }\Big|\\
		&\leq 2 \Big(\sum_{e\in F_h} \| \{ K \nabla v_h\cdot \mathbf{n}_e\} \|_{L^2(e)}^2\Big)^{\frac{1}{2}} 
		\Big(\sum_{e\in F_h} \| [v_h] \|_{L^2(e)}^2\Big)^{\frac{1}{2}} \\
		&\leq C h^{-1} \Big(\sum_{T\in \mathcal{T}_h} \| \nabla v_h \|_{L^2(T)^d}^2 \Big)^{\frac{1}{2}}
		\Big(\sum_{T\in \mathcal{T}_h} \| v_h \|_{L^2(T)}^2 \Big)^{\frac{1}{2}}
		\leq C h^{-2} \| v_h \|_{L^2(\D)}^2.
	\end{align*}
	For the last remaining summand $ \Gamma_3$, we apply Lemma~\ref{lem:disctra} and find
	\begin{equation*}
		\Gamma_3 
		= \Big| \sum_{e\in F_h} \int_e \chi_{\ell} \frac{\sigma}{h_e}[v_h]^2\diff{\xi}\Big|
		\leq C h^{-1}  \sum_{e\in F_h} \| [v_h] \|_{L^2(e)}^2 \leq C h^{-2}  \sum_{T\in \mathcal{T}_h} \| v_h \|_{L^2(T)}^2 = C h^{-2} \| v_h \|_{L^2(\D)}^2.
	\end{equation*}
	Combining the bounds for $ \Gamma_1 $, $ \Gamma_2$, and $ \Gamma_3 $, we obtain the desired result.
\end{proof}

The result of the following lemma is similar sort of the result for Friedrichs' operators in \cite[Proposition~4.14]{HochbruckKoehler.2020}.

\begin{lemma}\label{lem:swap}
	Let $A_{\ell}$ and $A_{h,\ell}$ be defined as in \eqref{eq:Asplit} and \eqref{eq:Ahsplit}, respectively. For $\ell \in \{1,2\}$, it follows that
	\begin{equation*}
		\big\| \big( A_{h,\ell}P_h-P_hA_{\ell} \big)v \big\|_{L^2(\D)} \leq C \| Av \|_{L^2(\D)} 
		\qquad \text{for all } v\in \dom(A).
	\end{equation*}
	This shows that Assumption~\ref{ass:Ah}~\ref{ass:dissl_exch} is satisfied.
\end{lemma}

\begin{proof}
	Since $\big( A_{h,\ell}P_h - P_hA_{\ell} \big)v \in V_h$ for $v\in \dom(A)$, we can write that
	\begin{equation*}
		\big\| \big( A_{h,\ell}P_h - P_hA_{\ell} \big)v \big\|_{L^2(\D)} 
		= \sup_{w_h\in V_h, \lno w_h\Ltwonex=1} \big| \big( (A_{h,\ell}P_h-A_{\ell})v,w_h \big)_{L^2(\D)}\big|.
	\end{equation*}	
	First, let us look at the non-discretized operator $A_{\ell}$ a bit closer. Using integration by parts, we find
	\begin{align*}
		\left( A_{\ell} v,w_h\right)_{L^2(\D)}
		&= \sum_{T\in\mathcal{T}_h} \left(\chi_{\ell}K\nabla v,\nabla w_h\right)_{L^2(T)}
		-\sum_{e\in F_h}\int_e ( \chi_{\ell}K\nabla v\cdot \mathbf{n}_e )[w_h]\diff{\xi}\\
		&= \sum_{T\in\mathcal{T}_h} \left(\chi_{\ell}K\nabla v,\nabla w_h\right)_{L^2(T)}
		-\sum_{e\in F_h}\int_e\{\chi_{\ell}K\nabla v\cdot \mathbf{n}_e\}[w_h]\diff{\xi}\\
		&\quad -\sum_{e\in F_h}\int_e\{\chi_{\ell}K\nabla w_h\cdot \mathbf{n}_e\}[v]\diff{\xi}+\sum_{e\in F_h}\int_e\chi_{\ell}\frac{\sigma}{h_e}[v][w_h]\diff{\xi},
	\end{align*}
	where in the first step the boundary terms do not disappear because of the possible discontinuities of $w_h$. In the second step, we use the fact that $\chi_{\ell}K\nabla v\cdot \mathbf{n}_e = \{\chi_{\ell}K\nabla v\cdot \mathbf{n}_e\}$ and $\lno[v]\rno_{L^2(e)}=0$ for $v \in \dom(A)$, compare \cite[Lemma~1.23]{DiPietro2012}. With an application of H\"older's inequality we then obtain that
	\begin{align*}
		& \big| \big( (A_{h,\ell}P_h-A_{\ell} ) v,w_h \big)_{L^2(\D)} \big|\\
		&= \Big| \sum_{T\in\mathcal{T}_h} \left(\chi_{\ell}K\nabla (P_h-I)v,\nabla w_h\right)_{L^2(T)}
		-\sum_{e\in F_h}\int_e\{\chi_{\ell}K\nabla (P_h-I)v\cdot \mathbf{n}_e\}[w_h]\diff{\xi}\\
		&\quad -\sum_{e\in F_h}\int_e\{\chi_{\ell}K\nabla w_h\cdot \mathbf{n}_e\}[(P_h-I)v]\diff{\xi}+\sum_{e\in F_h}\int_e\chi_{\ell}\frac{\sigma}{h_e}[(P_h-I)v][w_h]\diff{\xi} \Big| \\
		&\leq K_1\sum_{T\in\mathcal{T}_h}\lno\nabla (P_h-I)v\rno_{L^2(T)^d}\lno\nabla w_h\rno_{L^2(T)^d}\\
		&\quad + \sum_{e\in F_h}\lno\{K \nabla (P_h-I)v\cdot \mathbf{n}_e\}\rno_{L^2(e)}\lno[w_h]\rno_{L^2(e)}\\
		&\quad + \sum_{e\in F_h}\lno\{K \nabla w_h\cdot \mathbf{n}_e\}\rno_{L^2(e)}\lno[(P_h-I)v]\rno_{L^2(e)}\\
		&\quad +\sigma \sum_{e\in F_h} \frac{1}{h_e} \lno[(P_h-I)v]\rno_{L^2(e)}\lno[w_h]\rno_{L^2(e)}
		=: \Gamma_1 + \Gamma_2 + \Gamma_3 + \Gamma_4.
	\end{align*}
	To prove the desired bound, we will consider $\Gamma_{1}, \Gamma_2, \Gamma_3$, and $\Gamma_4,$ separately.
	First, for $\Gamma_1$, we can use Cauchy-Schwarz inequality for sums and then Lemmas~\ref{lem:Projerr} and \ref{lem:inv}
	\begin{align*}
		\Gamma_1 
		&=K_1\sum_{T\in\mathcal{T}_h}\lno\nabla (P_h-I)v\rno_{L^2(T)^d}\lno\nabla w_h\rno_{L^2(T)^d}\\
		&\leq C \Big(\sum_{T\in\mathcal{T}_h} h_T^{-2} \lno (P_h-I)v\rno_{H^1(T)}^2 \Big)^{\frac{1}{2}} 
		\Big(\sum_{T\in\mathcal{T}_h} h_T^2 \lno\nabla w_h\rno_{L^2(T)}^2 \Big)^{\frac{1}{2}} \\
		&\leq C \Big(\sum_{T\in\mathcal{T}_h} \lno v \rno_{H^2(T)}^2 \Big)^{\frac{1}{2}} 
		\Big(\sum_{T\in\mathcal{T}_h} \lno w_h\rno_{L^2(T)}^2 \Big)^{\frac{1}{2}} 
		= C \| v \|_{H^{2}(\D)} \lno w_h\Ltwonex,
	\end{align*}
	where we used that $\| v \|_{H^{2}(\D)} = \verttt v\verttt_{H^{2}(\mathcal{T}_h)}$ for a function $v \in H^2(\D)$.
	For the second term, $ \Gamma_2 $, we apply Lemma~\ref{lem:tra} and Lemma~\ref{lem:disctra}. Note that Lemma~\ref{lem:disctra} is formulated for scalar-valued functions instead of a vector-valued function $\nabla (P_h-I)v$. We can still apply the lemma for the components of the vector-valued function.
	With Lemma~\ref{lem:Projerr} we can then bound the projection error and then obtain in combination with the Cauchy--Schwarz inequality for sums
	\begin{align*}
		\Gamma_2 
		&=\sum_{e\in F_h}\lno\{K \nabla (P_h-I)v\cdot \mathbf{n}_e\}\rno_{L^2(e)}\lno[w_h]\rno_{L^2(e)}\\
		&\leq C \Big(\sum_{e\in F_h} \lno\{K \nabla (P_h-I)v\cdot \mathbf{n}_e\}\rno_{L^2(e)}^2\Big)^{\frac{1}{2}}
		\Big(\sum_{e\in F_h} \lno[w_h]\rno_{L^2(e)}^2 \Big)^{\frac{1}{2}}\\
		&\leq C \Big(\sum_{T\in\mathcal{T}_h} \| (P_h-I)v \|_{H^1(T)}
		\big(h_{T}^{-1} \| (P_h-I)v \|_{H^1(T)} + \| (P_h-I)v \|_{H^2(T)} \big) \Big)^{\frac{1}{2}}\\
		&\qquad \times \Big(\sum_{T\in\mathcal{T}_h} h_T^{-1}\lno w_h \rno_{L^2(T)}^2\Big)^{\frac{1}{2}}\\
		&\leq C \Big(\sum_{T\in\mathcal{T}_h}  h_T\| v \|_{H^2(T)}^2 \Big)^{\frac{1}{2}}
		\Big(\sum_{T\in\mathcal{T}_h} h_T^{-1}\lno w_h \rno_{L^2(T)}^2\Big)^{\frac{1}{2}}
		\leq C \| v \|_{H^2(\D)}\| w_h \|_{L^2(\D)},
	\end{align*}
	since $h_T \leq C h$ holds.
	Similarly, we bound $\Gamma_3$ by applying Lemmas~\ref{lem:disctra},~\ref{lem:inv} and \ref{lem:tra} for the first factor and Lemmas~\ref{lem:disctra} and \ref{lem:Projerr} for the second. An additional application of Cauchy--Schwarz inequality for sums then shows that
	\begin{align*}
		\Gamma_3
		&= \sum_{e\in F_h} \lno\{K \nabla w_h\cdot \mathbf{n}_e\}\rno_{L^2(e)}\lno[(P_h-I)v]\rno_{L^2(e)}\\
		&\leq C \Big(\sum_{e\in F_h} \lno\{K \nabla w_h\cdot \mathbf{n}_e\}\rno_{L^2(e)}^2\Big)^{\frac{1}{2}}
		\Big(\sum_{e\in F_h} \lno[(P_h-I)v]\rno_{L^2(e)}^2 \Big)^{\frac{1}{2}}\\
		&\!\!\leq C\! \Big( \hspace{-3pt}\sum_{T\in\mathcal{T}_h}\!\! h_T^{-1}\!\lno\nabla w_h\rno_{L^2(T)}^2 \hspace{-3pt} \Big)^{\frac{1}{2}}\!
			\Big(\hspace{-3pt}\sum_{T\in\mathcal{T}_h}\! \lno(P_h \!-\! I)v\rno_{L^2(T)}\big(h_T^{-1}\!\lno(P_h \!-\! I)v\rno_{L^2(T)}+\!\lno\nabla(P_h-I)v\rno_{L^2(T)^d}\big) \hspace{-3pt} \Big)^{\frac{1}{2}}\\
		&\!\leq C \Big(\sum_{T\in\mathcal{T}_h} h_T^{-3}\lno w_h \rno_{L^2(T)}^2\Big)^{\frac{1}{2}}
		\Big(\sum_{T\in\mathcal{T}_h} h_T^{3}\lno v\rno_{H^2(T)}^2 \Big)^{\frac{1}{2}}
		\leq C \| w_h \|_{L^2(\D)} \| v \|_{H^2(\D)}.
	\end{align*}
	For the last term $ \Gamma_4 $, we apply Lemmas~\ref{lem:tra}, \ref{lem:disctra} for the two obtained factors and the fact that $h_e = \min(h_{T_e^+}, h_{T_e^-}) \geq C h$ for every edge $e$ of an element $T$. Moreover, we use Lemma~\ref{lem:Projerr} to bound the projection error and Cauchy--Schwarz inequality for sums. More precisely, we obtain
	\begin{align*}
		\Gamma_4 
		&= \sigma \sum_{e\in F_h} h_e^{-1} \lno[(P_h-I)v]\rno_{L^2(e)}\lno[w_h]\rno_{L^2(e)}\\
		&\leq C  \Big(\!\sum_{e\in F_h} h_e^{-2}\lno[(P_h-I)v]\rno_{L^2(e)}^2\Big)^{\frac{1}{2}}
		\Big(\sum_{e\in F_h} \lno[w_h]\rno_{L^2(e)}^2 \Big)^{\frac{1}{2}}\\
		&\leq\!\! C \! \Big(\!\sum_{T\in\mathcal{T}_h} \! h_{T}^{-2}\!\lno (P_h \!-\! I)v\rno_{L^2(T)}\!\big(h_{T}^{-1}\!\lno(P_h \!-\! I)v\rno_{L^2(T)}\!\!+\!\!\lno\nabla(P_h \!-\! I)v\rno_{L^2(T)^d}\big) \!\!\Big)^{\frac{1}{2}}\!\!
			\Big(\!\sum_{T\in\mathcal{T}_h} h_T^{-1}\!\!\lno w_h\rno_{L^2(T)}^2 \!\Big)^{\frac{1}{2}}\\
		&\leq C \Big(\sum_{T\in\mathcal{T}_h} h_{T} \lno v\rno_{H^2(T)}^2\Big)^{\frac{1}{2}}
		\Big(\sum_{T\in\mathcal{T}_h} h_T^{-1}\lno w_h\rno_{L^2(T)}^2 \Big)^{\frac{1}{2}}\leq C \| v \|_{H^{2}(\D)}\lno w_h\Ltwonex.
	\end{align*}
	Combining the bound for $ \Gamma_1 $ to $ \Gamma_4 $, we can now prove the claimed result of the lemma
	\begin{equation*}
		\big\| \big(A_{h,\ell}P_h-P_hA_{\ell} \big)v \big\|_{L^2(\D)}
		=\sup_{w_h\in V_h, \lno w_h\Ltwonex=1} \big( \big(A_{h,\ell}P_h-A_{\ell}\big) v, w_h \big)_{L^2(\D)}
		\leq C \| v \|_{H^{2}(\D)}.
	\end{equation*}
	Since $\| v \|_{H^{2}(\D)}\leq C \| Av \|_{L^2(\D)}$ for $v\in\dom(A)$, we have completed the proof of the required statement.
\end{proof}

The last lemma of this section completes the verification of Assumption~\ref{ass:Ah}. 

\begin{lemma} \label{lem:bound_inv_A}
	Let $A$ and $A_{\ell}$ be defined as in \eqref{eq:def_A} and \eqref{eq:Asplit}, respectively. If $\sigma>C\frac{K_1^2}{K_0}$ where $C$ depends on the mesh, it follows that
	\begin{equation*}
		\| A^{-1} - A_h^{-1}P_h \|_{\L(L^2(\D))}\leq C h^{2},
	\end{equation*}
	i.e.~Assumption~\ref{ass:Ah}~\ref{ass:dissa_exch} is fulfilled.
\end{lemma}

A proof can be performed in a similar fashion as in \cite[Corollary~4.26]{DiPietro2012}. Additionally, we use the fact that $\| A^{-1}v \|_{H^2(\D)} \leq C \| v \|_{L^2(\D)}$, compare \cite[Theorem~9.24]{Hackbusch2017}) for our choice of $K$.

The lemmas of Section~\ref{sec:dG} show that the stochastic evolution equation from \eqref{eq:model} fits into the theoretical framework of Section~\ref{sec:Prob_description} and that the dG framework fulfills the assumptions stated in Section~\ref{sec:discretization}.
We have therefore verified that our main theoretical convergence result from Theorem~\ref{thm:result} is indeed applicable.

\section{Numerical experiments} \label{sec:Num_ex}
In the following section, we will validate our theoretical results through numerical tests. To implement the dG spatial discretization scheme, we used the software module DUNE-FEM \cite{dedner_2020_3706994,dune:Fem,dunereview:21,blatt2025distributedunifiednumericsenvironment}. For the dG approximation, we choose polynomials of at most order one and the parameter $\sigma=3$ in \eqref{eq:A_dG_formulation} and \eqref{eq:Ahsplit}, which is determined empirically.
We looked at two examples to test our method. First, we look at a semi-linear stochastic heat equation with a homogeneous Dirichlet boundary condition in Section~\ref{subsec:ex1}. This setting fits into the framework of Section~\ref{sec:Prob_description} and therefore is used to verify our error bound from Theorem~\ref{thm:result}. To show that the method also performs well in a more general framework, we also test it in a quasi-linear setting. In Section~\ref{subsec:ex2}, we therefore look at the stochastic porous media equation.

\subsection{Semi-linear test example} \label{subsec:ex1}
We look at a stochastic heat equation in a domain $\D=(0,1)^2\subset\R^2$ with a reaction term and multiplicative noise. More precisely, we look at the equation
\begin{equation}\label{eq:exmodel}	
	\begin{cases}
		\diff{X}\coord= \Big(\Delta X\coord +\sqrt{t}\pi^2(1+X\coord)\sin(\pi x)\sin(\pi y)\Big)\diff{t}&\\
		\quad\quad\quad\quad\quad\:\:+10X\coord \diff{W}(t,\mathbf{x}),& \coord\in (0,0.5]\times \D;\\
		X\coord=0, & \coord \in (0,0.5] \times \partial\D;\\
		X(0,\mathbf{x})=\sin(\pi x)\sin(\pi y),& \mathbf{x}\in \D,
	\end{cases}
\end{equation}
where we abbreviate $\mathbf{x}=(x,y)$. In the following, the spaces $H$ and $U$ from the theory in the previous sections are chosen to be $L^2(\D)$.
The $Q$-Wiener is defined by its Karhunen-Lo\`{e}ve expansion
\begin{equation}\label{eq:W_num_exp1}
	W\coord
	=\sum_{k = 1}^{\infty} (g_1(k)^2+g_2(k)^2)^{-2-\frac{\varepsilon}{2}} \sin(g_1(k)\pi x)\sin(g_2(k)\pi y)\beta_{g(k)}(t),
\end{equation}
where
\begin{equation*}
	g \colon \N \to \N^2, \quad k \mapsto (g_1(k), g_2(k))
	= \Big( k  - \frac{1}{2}(h(k) - 1) (h(k) - 2) , \ h(k) - k + \frac{1}{2}(h(k) - 1) (h(k) - 2) \Big)
\end{equation*}
for $h(k) = \lfloor \frac{3}{2} + (\frac{1}{4} + 2 (k-1) )^{\frac{1}{2}} \rfloor$, $\beta_{g(k)}$ are i.i.d.~$\mathcal{F}_t$-adapted Brownian motions and $\varepsilon=2\cdot 10^{-5}$. One can show that $g$ is a bijective mapping between $\N$ and $\N^2$. 

First, we note that the initial value $X_0$ is smooth and bounded in space and deterministic. Thus, Assumption~\ref{ass:X0} is fulfilled for every $\theta_{X_0} \in [0,1)$.
The function $f$ is chosen to be $\mathbf{x} \mapsto f(t,v)(\mathbf{x}) = \pi^2(1+v(\mathbf{x}))\sin(\pi x)\sin(\pi y)$. Since the $\sin$ functions are bounded on $\D$, it follows that $f(t,v)$ lies in $L^2(\D)$ for every $v \in L^2(\D)$. To verify $\| A^{\theta_{f}}f(t,w) \|_H \leq C\big(1+ \| A^{\theta_{f}}w \|_H \big)$ for every $w \in \dom(A^{\theta_f})$ and $\theta_f \in [0,\frac{1}{2}) \setminus \{\frac{1}{4}\}$, we refer the reader to the calculation in \cite[Equations~(20) and (23)]{JENTZEN2012114}. Thus, Assumption~\ref{ass:f_reg} holds.
Since the function $v$ appears in a linear fashion, Assumption~\ref{ass:f_hold} is also fulfilled, which proves that Assumption~\ref{ass:f} holds.

Moreover, we can verify that the diffusion coefficient given by $B(t,v) u  = v \cdot u$ fits into Assumption~\ref{ass:B}. The well-definedness, the bounds in \eqref{eq:HS_B}, and Assumption~\ref{ass:B_hold} are verified in \cite[page~121]{JENTZEN2012114}, where we also use the fact that $B(t,0) = 0$. It only remains to show that \eqref{eq:boundedOp_B} is fulfilled. For $v \in \dom(A^{\theta_{X_0}})$, we obtain that
\begin{equation*}
	\| B(t,v) \|_{\L(U,H)} 
	= \sup_{\|u\|_{L^2(\D)} = 1} \| v \cdot u \|_{L^2(\D)} 
	\leq \| v \|_{L^{\infty}(\D)}
	\leq \| v \|_{H^{2\theta_{X_0} }(\D)}
	\leq C \|A^{\theta_{X_0}} v \|_H,
\end{equation*}
where first use that $H^{2\theta_{X_0}}(\D)$ is continuously embedded into $L^{\infty}(\D)$ for $\theta_{X_0} > \frac{1}{2}$, see \cite[Theorem~4.12]{AdamsFournier.2003}. Additionally, we used that the fractional Sobolev space $H^{2\theta_{X_0} }(\D)$ is the $\theta_{X_0}$-interpolation between $L^2(\D)$ and $H^2(\D)$ (compare \cite[Theorem~12.4.]{LionsMagenes.1972}) while $\dom(A^{\theta_{X_0}})$ is the $\theta_{X_0}$-interpolation between $L^2(\D)$ and $\dom(A) = H^2(\D) \cap H^1_0(\D)$ (see \cite[Theorem~4.36]{Lunardi.2009}). Since we can equip both $H^2(\D)$ and $H^2(\D) \cap H^1_0(\D)$ with the same norm, the norms of $H^{2\theta_{X_0} }(\D)$ and $\dom(A^{\theta_{X_0}})$ are equivalent. Thus, Assumption~\ref{ass:B} is satisfied for all $\theta_{B} \in [0,\half) \setminus \{\frac{1}{4}\}$ and $\theta_{X_0} \in (\frac{1}{2},1)$.

As weight functions, we choose
\begin{equation*}
	\chi_0=\begin{cases}
		1& \text{if }x<\half-\delta;\\
		\frac{-x+\half+\delta}{2\delta}& \text{if }x<\half-\delta;\\
		0& \text{if }x>\half+\delta
	\end{cases}
	\quad 
	\text{and}
	\quad
	\chi_1=\begin{cases}
		0& \text{if }x<\half-\delta;\\
		\frac{x-\half+\delta}{2\delta}& \text{if }x<\half-\delta;\\
		1& \text{if }x>\half+\delta,
	\end{cases}
\end{equation*}
with $\delta=0.1$. We choose $V_{h}$ such that the corresponding $\mathcal{T}_h$ is a Cartesian $M$ by $M$ grid for varying $M\in\N$ and $h = M^{-1}$.

We note that one can show that $C (k+1) \leq (g_1(k)^2+g_2(k)^2) \leq C (k+1)$ for all $k \in \N$. Thus, in terms of Assumption~\ref{ass:Proj_error_U} the eigenvalues lie in $\mathcal{O}(k^{-(2r + 1 + \varepsilon) })$ for $r = \frac{3}{2}$. For $N_U = \lfloor h^{-\frac{4}{3}} \rfloor$, we then observe $(N_U + 1)^{-r} \leq (h^{-\frac{4}{3}})^{-\frac{3}{2}} = h^{2} \leq h^{2 \theta_U+ 1}$ for all $\theta_U \in [0,\half)$.

For this setting, we have two experiments, where we compute the $L^2(\Omega;L^2(\D))$-error at the final time. To approximate the $L^2(\Omega;L^2(\D))$-norm, we use a Monte Carlo simulation with one hundred samples in the form
\begin{equation}\label{Num_Error_approx}
	\|X_{h,\tau}^N(t_f)-X(t_f)\|_{L^2(\Omega; L^2(\D))} 
	\approx \Big(\frac{1}{100}\sum_{j=1}^{100} \|X_{h,\tau}^N(t_f,\omega_j)-X_{h_{\text{ref}},\tau_{\text{ref}}}(t_f,\omega_j) \|_{L^2(\D)}^2\Big)^{\half},
\end{equation}
where $X_{h_{\text{ref}},\tau_{\text{ref}}}$ is a reference solution.
For the first test, we fix the time step to $\tau = 2^{-12}$ and have different space discretization with $h=\frac{1}{5}\cdot 2^{-j}$ and $j=\{1,\dots,4\}$. In this case, our reference solution has been computed with $\tau_{\text{ref}} =2^{-12}$ and $h_{\text{ref}} = \frac{1}{5} \cdot 2^{-6}$.

For the second experiment, we check the convergence in time. Furthermore, we compare it to the method with the Lie splitting method and the semi-implicit Euler method (without splitting). The Lie splitting is obtained by replacing $S_{h,\tau}$ by $(I+\tau A_{h,2})^{-1} (I+\tau A_{h,1})^{-1}$ in scheme~\eqref{eq:scheme} while the semi-implicit Euler method is obtained by choosing $A_{h,1}=A_{h}$ and $A_{h,2}=0$ in scheme~\eqref{eq:scheme}. In this case, we consider a varying $\tau = 2^{-j}$ with $j=\{6,\dots,13\}$ and fix $h=1/150$. The reference solution has been computed with $\tau_{\text{ref}}= 2^{-15}$ and $h_{\text{ref}}=1/150$ using the semi-implicit Euler method. The convergence rates are as expected and no large difference is visible in these plots. The fact that our method performs similarly well as the semi-implicit Euler method is a positive result. This means that the splitting error is relatively small and a code uses our method for parallelization will not make a large additional error. While the error plot does not show a big difference between the Douglas--Racheford and the Lie splitting method, when comparing the solutions at the final time, it is visible that the error distribution of the Douglas--Racheford splitting is more even on the domain and not as concentrated on the overlap as for the Lie splitting.

Note that the decomposed operators $A_{h,1}$ and $A_{h,2}$ do not commute in this example. Thus, the assumptions of Theorem~\ref{thm:result_selfadjoint} are not satisfied. However, Theorem~\ref{thm:result} already implies a rate of convergence of $\frac{1}{2}$ in time and of $2$ in space due to the regularity properties of the data.

\begin{figure}[ht]
	\includegraphics[width=0.45\textwidth]{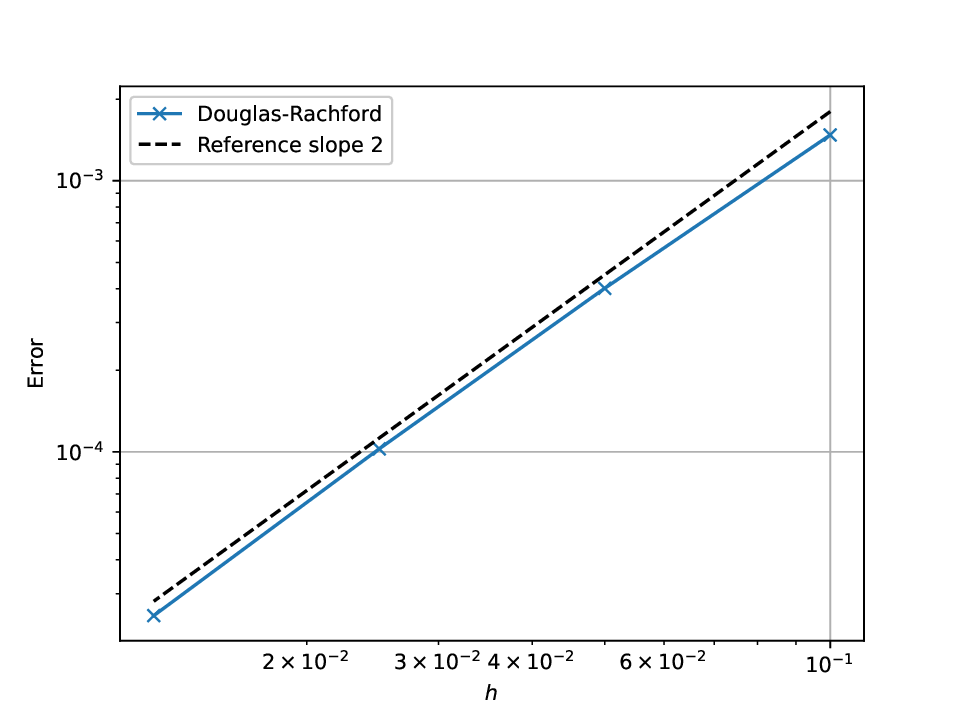}
	\includegraphics[width=0.45\textwidth]{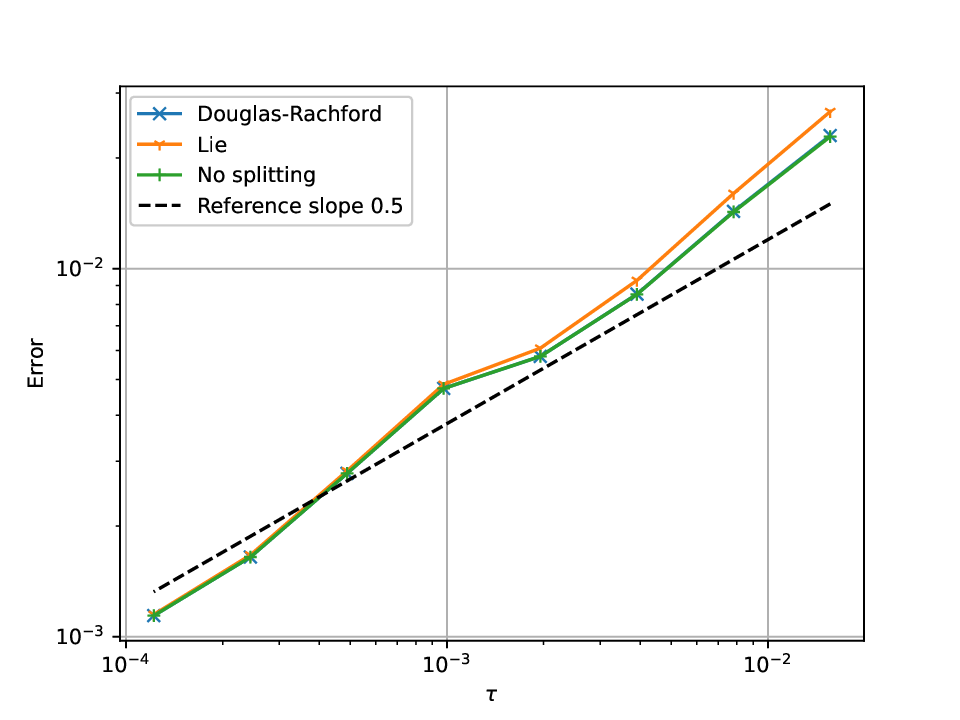}
	\caption{Left: Space convergence plot of the strong error at the final time for Experiment 1. Right: Time convergence plot of the strong error at the final time for Experiment 1.}
\end{figure}

\subsection{Quasi-linear test example} \label{subsec:ex2}
In this experiment, we generalize the problem class to a setting that is not included in the theoretical result. This is to suggest that the method \eqref{eq:scheme} can also be applied to more general cases. We consider the stochastic porous medium equation
\begin{equation}\label{eq:exmodelPME}
	\begin{cases}
		\diff{X}\coord= \Delta X^4\coord\diff{t}+X\coord \diff{W}(t,\mathbf{x}),& \coord\in [0,0.05]\times \D;\\
		X\coord=0, & \coord \in [0,0.05] \times \partial\D;\\
		X(0,\mathbf{x})=  S^{-\frac{1}{5}} \max\Big(0, \frac{1}{10}-\frac{3}{40}\frac{4(x-\half)^2}{S^{\frac{2}{5}}} \Big),&\mathbf{x}\in \D,
	\end{cases}
\end{equation}
where $\D=(0,1)\subset\R$, $\mathbf{x}=(x,y)$, $S=0.02$. The $Q$-Wiener process can be stated as
\begin{equation*}
	W\coord=\sum_{k=1}^{\infty} k^{-\frac{5}{2}-2\varepsilon}\sin(k\pi x)\beta_{k}(t),
\end{equation*}
where $\beta_{k}$ are i.i.d. $\mathcal{F}_t$-adapted Brownian motions and $\varepsilon=10^{-5}$. The initial condition is chosen such that we have the Barenblatt solution in the deterministic case. The short time interval is to ensure that the solution has not reached the boundary yet, at least in the deterministic case.
We choose $V_{h}$ such that the domain is divided into intervals of equal length $h = \frac{1}{M}$, $M \in \N$. For the discretized operator $B_h$, we choose $N_U=M$, which satisfies Assumption~\ref{ass:Proj_error_U} for $\theta_U\in[0,\half)$.

For this experiment, we also look at the spatial and temporal convergence. We drop the symmetry term in the discrete operator of \eqref{eq:A_dG_formulation} due to implementation issues. In this case, it is an incomplete interior penalty Galerkin (IIPG) method, see~\cite[Chapter 1]{Riviere}. For the spatial convergence, we fix $\tau=10^{-3}\cdot 2^{-14}$ and consider varying spatial step sizes $h=0.1\cdot 2^{-j}$ with $j=\{0,\dots,5\}$. The reference solution is computed with $h_{\text{ref}} =  0.1\cdot 2^{-7}$ and $\tau_{\text{ref}}=10^{-3}\cdot 2^{-14}$ using the Douglas--Rachford splitting. The observed spatial convergence order is around one. It is expected to observe a lower spatial convergence order for the porous medium equation compared to the linear case. See for example \cite{Rulla1996}, where they prove a convergence result for fully-discretized porous medium equations in two dimensions. For the temporal convergence, we fix $h=1/200$ and consider varying temporal step sizes $\tau=10^{-3}\cdot 2^{-j}$ with $j=\{5,\dots,13\}$. The reference solution is computed with $h_{\text{ref}} =  1/200$ and $\tau_{\text{ref}}=10^{-3}\cdot 2^{-15}$ using the semi-implicit Euler method. Again, we compare the Douglas--Rachford splitting with the Lie splitting and the semi-implicit Euler method. The observed temporal convergence order is around $0.5$ as suggested for the less general case considered in our theory. The errors are estimated as described in \eqref{Num_Error_approx}. The temporal errors of the different schemes do not differ much. As for such a nonlinear problem, the parallelization of code is even more relevant than in the linear setting; this is a promising result.
\begin{figure}[ht]
	\includegraphics[width=0.45\textwidth]{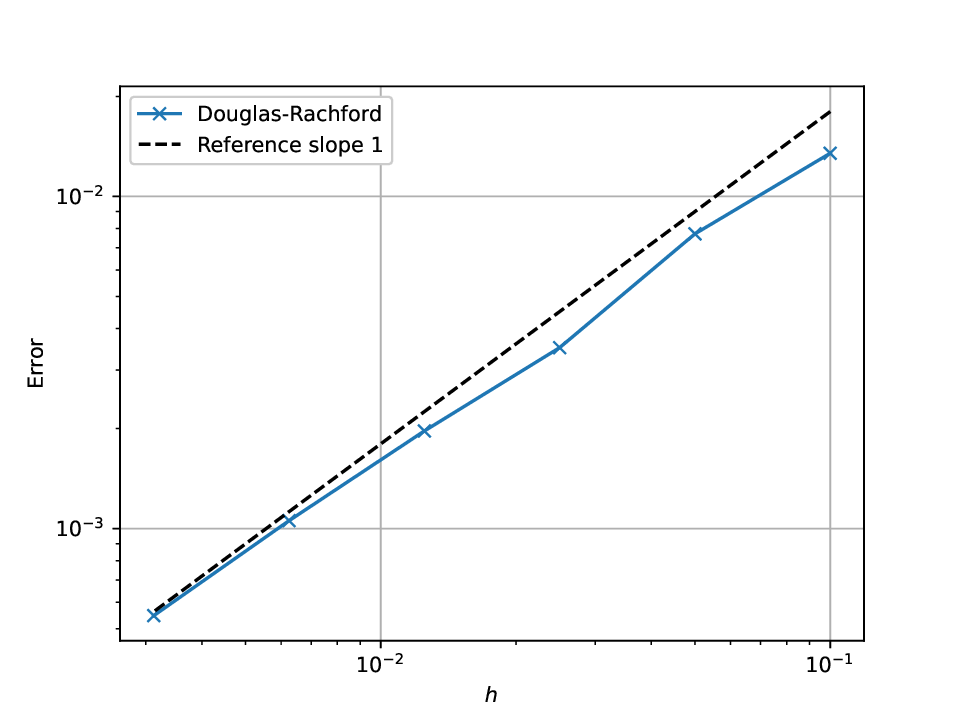}
	\includegraphics[width=0.45\textwidth]{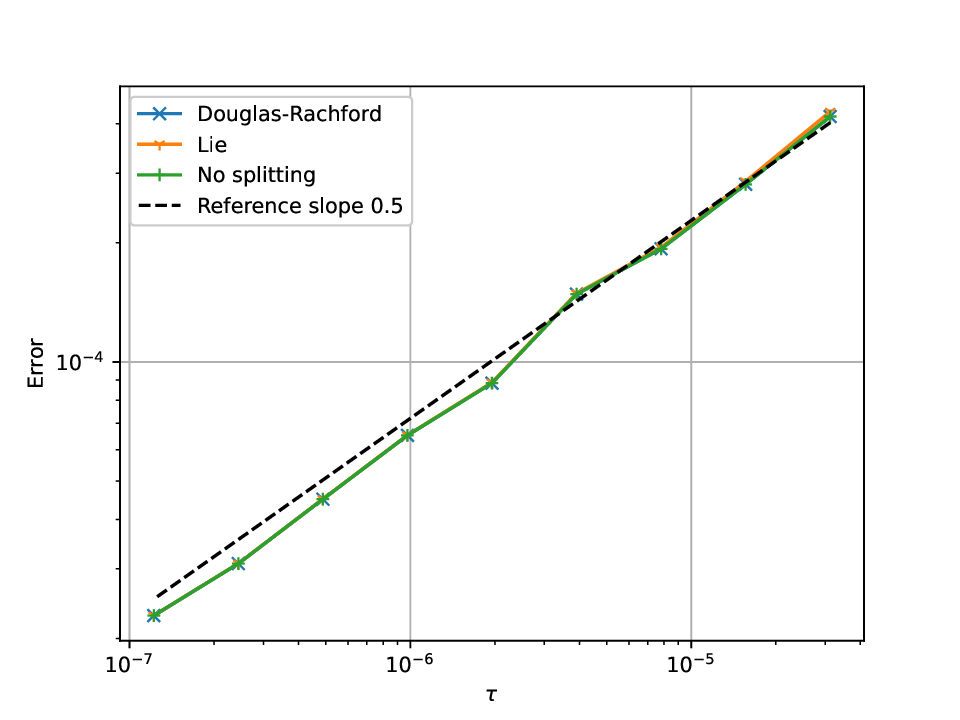}
	\caption{Left: Space convergence plot of the strong error at the final time for Experiment 2. Right: Time convergence plot of the strong error at the final time for Experiment 2.}
\end{figure}
\appendix
\section{Basic results} \label{appendix:basic_results}
In this first part of the appendix, we collect some basic inequalities that are of importance throughout the paper.

\begin{lemma}\label{lem:interA}
	For a real Hilbert space $H$, let $\mathcal{B} \in \L(H)$ be given and let $A$ fulfill Assumption~\ref{ass:A}. For every $\zeta \in [0,1]$, it follows that
	\begin{equation*}
		\| \mathcal{B}A^{-\zeta} \|_{\L(H)} 
		\leq C \| \mathcal{B}A^{-1} \|_{\L(H)}^{\zeta} \| \mathcal{B} \|_{\L(H)}^{1-\zeta}.
	\end{equation*}
\end{lemma}

\begin{proof}
	For the proof, we define the constant $C_0=\frac{\| \mathcal{B}A^{-1} \|_{\L(H)}}{\lno\mathcal{B}\Loptwon}$. Together with an application of the semigroup bounds \eqref{eq:ana_semi2} and \eqref{eq:ana_semi1}, we find 
	\begin{align*}
		\| \mathcal{B}A^{-\zeta} \|_{\L(H)}
		&\leq \big\| \mathcal{B}A^{-\zeta} (I- \e^{-C_0A}) \big\|_{\L(H)}
		+ \big\| \mathcal{B}A^{-\zeta} \e^{-C_0 A} \big\|_{\L(H)} \\
		&\leq \| \mathcal{B} \|_{\L(H)} 
		\big\| A^{-\zeta} (I-\e^{-C_0 A}) \big\|_{\L(H)} + \big\| \mathcal{B}A^{-1} \big\|_{\L(H)}
		\big\| A^{1-\zeta}\e^{-C_0 A} \big\|_{\L(H)}\\
		&\leq C \| \mathcal{B} \|_{\L(H)} C_0^{\zeta}
		+C \| \mathcal{B} A^{-1}\|_{\L(H)} C_0^{\zeta-1}
		\leq C \| \mathcal{B}A^{-1}\|_{\L(H)}^{\zeta} \| \mathcal{B} \|_{\L(H)}^{1-\zeta},
	\end{align*}
	which proves the claimed result.
\end{proof}

\begin{lemma} \label{lem:riemann}
	For $N \in \N$ and $t_f \in \R^+$, consider $\tau = \frac{t_f}{N}$ and $t_k = k \tau$ for $k \in \{1,\dots,N\}$. 
	For every $n \in \{1,\dots, N\}$ and $\zeta \in [0,1)$, it follows that	
	\begin{equation*}
		\tau \sum_{k=1}^{n} t_k^{-\zeta} \leq C
		\quad \text{and} \quad
		\tau \sum_{k=1}^{n} t_k^{-1} \leq 1 + \ln(n).
	\end{equation*}
\end{lemma}

\begin{proof}
	We use that for every $s \in (t_{k-1},t_k)$ it follows that $t_k^{-\zeta} \leq s^{-\zeta}$. Thus, we obtain
	\begin{equation*}
		\tau \sum_{k=1}^{n} t_k^{-\zeta}
		\leq \tau^{1-\zeta} + \sum_{k=2}^{n} \int_{t_{k-1}}^{t_k} s^{-\zeta} \diff{s}
		= \tau^{1-\zeta} + \frac{t_n^{1-\zeta} - \tau^{1-\zeta}}{1-\zeta} \leq C
	\end{equation*}
	and analogously for $\zeta = 1$, we find
	\begin{equation*}
		\tau \sum_{k=1}^{n} t_k^{-\zeta}
		\leq 1 + \sum_{k=2}^{n} \int_{t_{k-1}}^{t_k} s^{-1} \diff{s}
		= 1 + \ln(t_n) - \ln(\tau) = 1 + \ln(n).
	\end{equation*}	
\end{proof}

\begin{lemma}\label{gron}
	Let $a,b \in \R_0^+$ and $N \in \N$ be given. Further, for all $n\in\{0,\dots,N\}$, let $u_{n}\leq a+b\sum_{k=0}^{n-1} u_{k}$ be fulfilled. Then it follows that $u_{n}\leq a\e^{n	b}$.
\end{lemma}

For a proof, we refer to \cite{Clark.1987}.
\section{Auxiliary dG results} \label{appendix:dG}

The following lemmas are some basic bounds for the dG setting considered in Section~\ref{subsec:dG_Notation}. For the exact notation in this subsection, we refer the reader to Section~\ref{subsec:dG_Notation} for an explanation.

\begin{lemma}\label{lem:disctra}
	For every $p\in[1,\infty]$, there exists $\constChi(p,\kappa) = \constChi \in \R^+$ such that for all $T\in \mathcal{T}_h$, all $e \in F_h$ such that $e \subset \partial T$ and $v_h$, a $\kappa$-th degree polynomial restricted on $T$, it follows that
	\begin{equation*}
		\| v_h \|_{L^p\left(e\right)} \leq \constChi h_{T}^{-\frac{1}{p}} \| v_h \|_{L^p(T)}.
	\end{equation*}
\end{lemma}

For a proof, see \cite[Lemma~1.52]{DiPietro2012}. In this lemma, we explicitly state the constant $\constChi$, because the lower bound of $\sigma$ depends on $\constChi$. For more theoretical information on the magnitude of $\constChi$, we refer to \cite[Remark~1.51 and Remark~1.53]{DiPietro2012} and \cite[Section~12.2]{Ern2021}.

\begin{lemma}\label{lem:inv}
	For all $v_h\in V_h$ and $T\in \mathcal{T}_h$, it follows that
	\begin{equation*}
		\|\nabla v_h \|_{L^2(T)^d}\leq C h_T^{-1} \| v_h \|_{L^2(T)}.
	\end{equation*}
	The constant $C$ is dependent on the space $V_h$, but is independent of $h$ and $v_h$. 
\end{lemma}

For a general proof, we refer to \cite[Lemma~1.44]{DiPietro2012}.

\begin{lemma}\label{lem:tra}
	For all $T\in \mathcal{T}_h$, it follows that 
	\begin{equation*}
		\lno v\rno_{L^2\left(\partial T\right)}
		\leq C\lno v\rno_{L^2(T)}^{\frac{1}{2}}\left(h_{T}^{-1}\| v\|_{L^2(T)}+ \| \nabla v\|_{L^2(T)^d }\right)^{\frac{1}{2}}
	\end{equation*}
	for every $v\in H^1(T)$, where $C$ is independent of $h$ and $T$.
\end{lemma}

A proof can be found in \cite[Lemma~1.49]{DiPietro2012}.

\section{Abstract discretization} \label{appendix:spatialDisc}

The following lemmas show how to bound the error of the space-discretized semigroup in a dG setting. This kind of result for possibly non-selfadjoint operators $A$ can be found in \cite[Theorem~2 and 3]{MingyouThomee.1981} for a finite element setting. A suitable generalization can be found in lecture notes published by \cite{Crouzeix_lecture}. For the sake of completeness, we give the proofs.

\begin{lemma}\label{app1}
	Let Assumptions~\ref{ass:A} be fulfilled. Let $\varphi \in (0, \frac{\pi}{2})$ be given such that $S_{\varphi} = \{\lambda\in\mathbb{C} : \varphi < |\arg(\lambda)| \leq \pi \}$ lies in the resolvent set $\rho(A)$. For all $\lambda\in S_{\varphi}$, it follows that 
	\begin{equation*}
		\|A(\lambda I - A)^{-1}\|_{\mathcal{L}(H_{\mathbb{C}})} 
		\leq \frac{|\lambda|}{\mathrm{dist}(\lambda ,S_{\varphi}^c)},
	\end{equation*}
	where $S_{\varphi}^c = \mathbb{C} \setminus S_{\varphi}$ and ${\mathrm{dist}(\lambda ,S_{\varphi}^c)}:= \inf_{\tilde{\lambda}\in S_{\varphi}^c}|\lambda-\tilde\lambda|$.
\end{lemma}

\begin{proof}
	We consider $v\in \dom(A)$ and $w\in H$ such that $(\lambda I-A)v=w$ is fulfilled. Since $\inner[H]{A v}{v} \geq 0$ and $\R_0^+ \subset S_{\varphi}^c$, it follows that $\frac{|\lambda |^2}{\|Av\|_H^2} (Av,v)_H \in S_{\varphi}^c$. We can now use this fact to estimate the distance between $\lambda$ and $S_{\varphi}^c$ as follows
	\begin{align*}
		\mathrm{dist}(\lambda ,S_{\varphi}^c)\|Av\|_H^2
		&\leq \Big| \lambda - \frac{|\lambda |^2}{\|Av\|_H^2} (Av,v)_H \Big| \|Av\|_H^2
		= \big|\lambda \|Av\|_H^2 - \overline{\lambda} \lambda (Av,v)_H\big|\\
		&=|\lambda ||(Av,Av - \lambda v)_H|
		= |\lambda| |(Av,w)|
		\leq |\lambda |\|Av\|_H\|w\|_H.
	\end{align*}
	Dividing the inequality by $\|Av\|_H$ and inserting the definition for $w$ into the left-hand side of the previous inequality, we find $\|A(\lambda I-A)^{-1} w\|_{H_{\mathbb{C}}}\leq \frac{|\lambda |}{\mathrm{dist}(\lambda ,S_{\varphi}^c)}\|w\|_H$. This proves our result.
\end{proof}

\begin{lemma}\label{app2}
	Let Assumptions~\ref{ass:A}, \ref{ass:proj_err}, and \ref{ass:Ah} be fulfilled. For every $t\in(0,t_f]$, it follows that
	\begin{equation}\label{eq:proof_errS2A}
		\| \e^{-tA}-\e^{-tA_h}P_h \|_{\L(H)} \leq C\frac{h^{2}}{t}.
	\end{equation}
\end{lemma}

\begin{proof}
	In the following, let $\varphi \in (0,\frac{\pi}{2})$ be given such that $S_{\varphi} = \{\lambda\in\mathbb{C} : \varphi < |\arg(\lambda)| \leq \pi \}$ lies in both $\rho(A)$ and $\rho(A_h)$.
	Using the integral form of an analytical semigroup, see \cite[Chapter~1, Theorem~7.7]{Pazy1983}, and choosing a $\tilde{\varphi} \in(\varphi, \frac{\pi}{2})$, we find that
	\begin{equation} \label{eq:semigroup_integral_eq}
		\e^{-tA}-\e^{-tA_h}P_h
		= \frac{1}{2\pi i } \int_{\Gamma_{\tilde{\varphi}}} \e^{-t\lambda}\big( (\lambda I-A)^{-1} - (\lambda I-A_h)^{-1}P_h\big) \diff\lambda,
	\end{equation}
	where $\Gamma_{\tilde{\varphi}}:=\{\lambda\in\mathbb{C}: | \arg(\lambda) | =\tilde{\varphi}\}$. Note that this is well-defined since $\Gamma_{\tilde{\varphi}} \subset S_{\varphi}$. For the following analysis, we decompose the integrant into the two parts
	\begin{align*}
		&\e^{-t\lambda} \big((\lambda I-A)^{-1}-(\lambda I-A_h)^{-1}P_h\big)\\
		&= \e^{-t\lambda} \big((\lambda I-A)^{-1}-(\lambda I-A_h)^{-1}\big)P_h
		+ \e^{-t\lambda} \big((\lambda I-A)^{-1}-(\lambda I-A_h)^{-1}P_h\big) (I - P_h) 
		= \Gamma_{1} + \Gamma_{2}.
	\end{align*}	
	Some algebraic manipulations give us the following norm bound for $\Gamma_1$ 
	\begin{align*} 
		\| \Gamma_1\|_{\L(H_{\mathbb{C}})}
		&= \big\| \e^{-t\lambda} \big((\lambda I-A)^{-1} - (\lambda I-A_h)^{-1}  \big)P_h \big \|_{\L(H_{\mathbb{C}})}\\
		&= |\e^{-t\lambda}| \big\| A (\lambda I-A)^{-1} \big( A^{-1} (\lambda I-A_h) A_h^{-1} - (\lambda I-A)A^{-1} A_h^{-1} \big) A_h (\lambda I-A_h)^{-1} P_h \big\|_{\L(H_{\mathbb{C}})}\\
		&= |\e^{-t\lambda}| \big\| A (\lambda I-A)^{-1} \big( A_h^{-1} - A^{-1} \big) A_h (\lambda I-A_h)^{-1} P_h \big\|_{\L(H_{\mathbb{C}})}\\
		&\leq |\e^{-t\lambda}| \big\| A (\lambda I-A)^{-1}\big\|_{\L(H_{\mathbb{C}})}
		\big\| A_h^{-1} P_h - A^{-1} \big\|_{\L(H)}
		\big\| A_h (\lambda I-A_h)^{-1} P_h \big\|_{\L(H_{\mathbb{C}})}.
	\end{align*}	
	For the norm of $\Gamma_{2}$, we use that for a projection operator $P_h$ it holds that $P_h (I - P_h) = 0$. Thus, we obtain  $(\lambda I-A_h)^{-1} P_h (I - P_h) = A (\lambda I-A)^{-1} A_h^{-1} P_h (I - P_h) = 0$. This leads to
	\begin{align*} 
		\| \Gamma_2\|_{\L(H_{\mathbb{C}})}
		&= \big\| \e^{-t\lambda} \big((\lambda I-A)^{-1} - (\lambda I-A_h)^{-1} P_h \big) (I - P_h) \big\|_{\L(H_{\mathbb{C}})}\\
		&= |\e^{-t\lambda}| \big\| \big((\lambda I-A)^{-1} - A (\lambda I-A)^{-1} A_h^{-1} P_h \big) (I - P_h)\big\|_{\L(H_{\mathbb{C}})}\\
		&\leq |\e^{-t\lambda}| \big\| A (\lambda I-A)^{-1}\big\|_{\L(H_{\mathbb{C}})} 
		\big\|A^{-1} - A_h^{-1} P_h \big\|_{\L(H)}.
	\end{align*}	
	Since both $A$ and $A_h$ satisfy Assumption~\ref{ass:A_sect} (in the case of $A_h$ we change $H$ to $V_h$ equipped with the $\|\cdot\|_H$-norm), we can apply Lemma~\ref{app1} to obtain 
	\begin{equation*}
		\|A(\lambda I-A)^{-1} \|_{\mathcal{L}(H_{\mathbb{C}})}
		\leq \frac{|\lambda|}{\text{dist}(\lambda ,S_{\varphi}^c)}
		\quad \text{and} \quad 
		\| A_h (\lambda I - A_h )^{-1} P_h \|_{\mathcal{L}(H_{\mathbb{C}})}
		\leq \frac{|\lambda|}{\text{dist}(\lambda ,S_{\varphi}^c)}.
	\end{equation*}
	Combining the above bounds with Assumption~\ref{ass:Ah}~\ref{ass:dissa_exch}, we find that
	\begin{align*}
		&\big\|\e^{-t\lambda} (\lambda I-A)^{-1}-(\lambda I-A_h)^{-1}P_h \big\|_{\mathcal{L}(H_{\mathbb{C}})}\\
		&\leq \| \Gamma_{1}\|_{\mathcal{L}(H_{\mathbb{C}})}
		+ \| \Gamma_{2} \|_{\mathcal{L}(H_{\mathbb{C}})}\\
		&\leq |\e^{-t\lambda}| \big\| A (\lambda I-A)^{-1}\big\|_{\L(H_{\mathbb{C}})}
		\big\| A_h^{-1} P_h - A^{-1} \big\|_{\L(H)}
		\big\| A_h (\lambda I-A_h)^{-1} P_h \big\|_{\L(H_{\mathbb{C}})}\\
		&\quad +  |\e^{-t\lambda}| \big\| A (\lambda I-A)^{-1}\big\|_{\L(H_{\mathbb{C}})} 
		\big\|A^{-1} - A_h^{-1} P_h \big\|_{\L(H)}\\
		&\leq C  h^2 |\e^{-t\lambda}| \Big(\frac{|\lambda|^2}{(\text{dist}(\lambda ,S_{\varphi}^c))^2} + \frac{|\lambda| }{\text{dist}(\lambda ,S_{\varphi}^c)}\Big).
	\end{align*}
	It remains to insert this previous bound in \eqref{eq:semigroup_integral_eq}, which then shows that
	\begin{align*}
		\| \e^{-tA}-\e^{-tA_h}P_h \|_{\L(H)}
		&\leq C \int_{\Gamma_{\tilde{\varphi}}} \big\| \e^{-t\lambda}((\lambda I-A)^{-1} - (\lambda I-A_h)^{-1}P_h) \big\|_{\L(H_{\mathbb{C}} )} \diff\lambda\\
		&\leq C h^2 \int_{\Gamma_{\tilde{\varphi}}} | \e^{-t\lambda}| \Big(\frac{|\lambda|}{\text{dist}(\lambda ,S_{\varphi}^c)} + \frac{ |\lambda|^2}{(\text{dist}(\lambda ,S_{\varphi}^c))^2}\Big) \diff\lambda\\
		&= C h^2 \Big( \int_{0}^{\infty} \e^{-t r } \Big(\frac{ r}{\text{dist}(re^{- i \tilde{\varphi}} , S_{\varphi}^c)} + \frac{ r^2}{(\text{dist}(re^{- i \tilde{\varphi}} ,S_{\varphi}^c))^2}\Big) \diff{r}\\
		&\quad + \int_{0}^{\infty} \e^{-t r } \Big(\frac{ r}{\text{dist}(re^{i \tilde{\varphi}} ,S_{\varphi}^c)} + \frac{ r^2}{(\text{dist}(re^{i \tilde{\varphi}} ,S_{\varphi}^c))^2}\Big) \diff{r}\Big) \\
		&= C h^2 \Big( \int_{0}^{\infty} \e^{-t r } \Big( \frac{1}{\sin(\tilde{\varphi}-\varphi)} + \frac{ 1}{\sin^2(\tilde{\varphi}-\varphi)}\Big) \diff{r}
		= C \frac{h^2}{t},
	\end{align*}
	where we used that for $\tilde{\varphi}\in(\varphi,\frac{\pi}{2})$, it holds that $\text{dist}(re^{\pm i \tilde{\varphi}} ,S_{\varphi}^c) = r \sin(\tilde{\varphi}-\varphi)$.
\end{proof}

\begin{lemma}\label{app3}
	Let Assumptions~\ref{ass:A} and \ref{ass:Ah} hold, then it follows that
	\begin{equation*}
		\| ( \e^{-tA}-\e^{-tA_h}P_h )A^{-1} \|_{\L(H)} \leq Ch^{2}.
	\end{equation*}
\end{lemma}

\begin{proof}
	We begin to split the error into three parts
	\begin{align*}
		\| ( \e^{-tA}-\e^{-tA_h}P_h )A^{-1} \|_{\L(H)}
		&\leq \| (A^{-1}-A_h^{-1}P_h)\e^{-tA} \|_{\L(H)}+ \| A_h^{-1}(P_h\e^{-tA}- \e^{-tA_h}P_h) \|_{\L(H)}\\
		&\quad + \| \e^{-tA_h}P_h(A_h^{-1}P_h-A^{-1}) \|_{\L(H)} 
		:= \Gamma_1+\Gamma_2+\Gamma_3.
	\end{align*}
	Using Assumption~\ref{ass:Ah}~\ref{ass:dissa_exch}, the error terms $\Gamma_1$ and $\Gamma_3$ can be  bounded by $Ch^2$. It remains to bound $\Gamma_{2}$. To do this, we define 
	\begin{equation*}
		e_h(t):=A_h^{-1}(P_h\e^{-tA}-\e^{-tA_h}P_h).
	\end{equation*}
	A simple insertion verifies that this function is a solution to the initial value problem given by
	\begin{equation}\label{eq:err_A}
		\begin{cases}
			e_h'(t)+A_he_h(t)=(P_h-A_h^{-1}P_hA)\e^{-tA}, \quad t \in (0,t_f], \\
			e_h(t) = 0.
		\end{cases}
	\end{equation}
	Applying the structure of \eqref{eq:err_A} to $\Gamma_{2}$, we subdivide it into the two following terms
	\begin{align*}
		\Gamma_2
		&= \|e_h(t)\|_{\L(H)}
		\leq \Big\|\int_0^t \e^{-(t-s)A_h}(P_h - A_h^{-1}P_hA) \e^{-sA} \diff{s}\Big\|_{\mathcal{L}(H)}\\
		&\leq \Big\|\int_0^{\frac{t}{2}}  \e^{-(t-s)A_h}(P_h - A_h^{-1}P_hA) \e^{-sA}  \diff{s} \Big\|_{\mathcal{L}(H)}
		+ \int_{\frac{t}{2}}^t \big\| \e^{-(t-s)A_h}(P_h - A_h^{-1}P_hA) \e^{-sA} \big\|_{\mathcal{L}(H)}\diff{s}\\
		&=: \Gamma_{2,1} + \Gamma_{2,2}.
	\end{align*}
	To bound $\Gamma_{2,1}$, we integrate by parts, apply Assumption~\ref{ass:Ah}~\ref{ass:dissa_exch} and \eqref{eq:ana_semi1} to then find
	\begin{align*}
		\Gamma_{2,1}
		&= \Big\| \e^{-\frac{t}{2}A_h} (P_h A^{-1} - A_h^{-1}P_h) \e^{-\frac{t}{2}A} 
		- \e^{-tA_h} (P_h A^{-1} - A_h^{-1}P_h) \\
		&\qquad - \int_0^{\frac{t}{2}} \e^{-(t-s)A_h} (A_h P_h A^{-1} - P_h) \e^{-sA}  \diff{s}\Big\|_{\L(H)}\\
		&\leq \big\| \e^{-\frac{t}{2}A_h} P_h (A^{-1} - A_h^{-1}P_h) \e^{-\frac{t}{2}A} \big\|_{\L(H)}
		+ \big\|  \e^{-tA_h} P_h (A^{-1} - A_h^{-1}P_h)\big\|_{\L(H)} \\
		&\qquad + \Big\| \int_0^{\frac{t}{2}} A_h \e^{-(t-s)A_h} P_h (A^{-1} - A_h^{-1} P_h) \e^{-sA}  \diff{s} \Big\|_{\L(H)}\\
		&\leq C h^2 + C h^2 \int_0^{\frac{t}{2}} (t-s)^{-1} \diff{s}
		\leq C h^2.
	\end{align*}
	Next, we bound $\Gamma_{2,2}$ by using Assumption~\ref{ass:Ah}~\ref{ass:dissa_exch} and \eqref{eq:ana_semi1}. We then see that
	\begin{equation*}
		\Gamma_{2,2}
		\leq \int_{\frac{t}{2}}^t \big\| \e^{-(t-s)A_h}P_h (A^{-1} - A_h^{-1}P_h) A \e^{-sA} \big\|_{\mathcal{L}(H)}\diff{s}
		\leq Ch^2 \int_{\frac{t}{2}}^t s^{-1}\diff{s}
		\leq Ch^2.
	\end{equation*}
	Combining the bounds for $\Gamma_{1}$, $\Gamma_{2,1}$, $\Gamma_{2,2}$, and $\Gamma_3$, we obtain the claimed result.
\end{proof}

\section{Higher convergence} \label{appendix:higher_conv}
Under additional assumptions on the operator $A_h$ and its decomposition, we can improve the convergence result and thereby characterize better where the loss of convergence order compared to a non-split method comes from. This approach is inspired by \cite{Ichinose2001}.

\begin{lemma} \label{lem:high_conv}
	Let Assumptions~\ref{ass:A}, \ref{ass:proj_err}, \ref{ass:Ah} be fulfilled and assume additionally that $A_{h,1}$ and $A_{h,2}$ commute and are self-adjoint. For $S_{h,\tau}$ given in \eqref{eq:def_S}, it follows that
	\begin{equation}\label{eq:len_high_conv}
		\big\| \e^{-t_n  A_h} P_h - S_{h,\tau}^{n-1}(I+\tau A_{h,2})^{-1}(I+\tau A_{h,1})^{-1} P_h \big\|_{\L(H)} \leq C \frac{\tau }{t_n}
	\end{equation}
	for all $n \in \{1,\dots,N\}$.
\end{lemma}

\begin{proof}
	The proof follows similar arguments as in  \cite[Lemma~2.1]{Ichinose2001}. For a more compact notation, we will abbreviate
	\begin{equation*}
		\tilde{S}_{h,\tau} := (I+\tau A_{h,1})^{-1} (I+\tau^2A_{h,1}A_{h,2} ) (I+\tau A_{h,2})^{-1}
	\end{equation*}
	in the following proof. Inserting the definition of $S_{h,\tau}^{n-1}$, it follows that
	\begin{equation*}
		S_{h,\tau}^{n-1} (I+\tau A_{h,2})^{-1}(I+\tau A_{h,1})^{-1}
		= (I+\tau A_{h,2})^{-1} \tilde{S}_{h,\tau}^{n-1} (I+\tau A_{h,1})^{-1}.
	\end{equation*}
	With this in mind, we begin to decompose the right-hand side of \eqref{eq:len_high_conv} into three parts that we consider separately. Together with the abbreviation $\tilde{A}_h := (I+\tau A_{h,1})^{-1}A_h(I+\tau A_{h,2})^{-1}$, we find that 
	\begin{align*}
		&\big\| \e^{- t_n  A_h} P_h - S_{h,\tau}^{n-1} (I+\tau A_{h,2})^{-1}(I+\tau A_{h,1})^{-1} P_h \big\|_{\L(H)}\\
		&= \big\| \e^{- t_n  A_h} P_h - (I+\tau A_{h,2})^{-1} \tilde{S}_{h,\tau}^{n-1} (I+\tau A_{h,1})^{-1} P_h \big\|_{\L(H)}\\
		&\leq \big\| \e^{- t_n  A_h} P_h - (I+\tau A_{h,2})^{-1}\e^{- t_{n-1}  A_h}(I+\tau A_{h,1})^{-1} P_h \big\|_{\L(H)}\\
		&\quad+ \big\| (I+\tau A_{h,2})^{-1} \big( \e^{-t_{n-1}  A_h} - \e^{-t_{n-1}\tilde{A}_h} \big) (I+\tau A_{h,1})^{-1}P_h \big\|_{\L(H)}\\
		&\quad+\big\| (I+\tau A_{h,2})^{-1} \big( \e^{-t_{n-1} \tilde{A}_h} - \tilde{S}_{h,\tau}^{n-1} \big)(I+\tau A_{h,1})^{-1}P_h \big\|_{\L(H)}\\
		&=: \|\Gamma_1 P_h\|_{\L(H)} + \|\Gamma_2 P_h\|_{\L(H)} + \|\Gamma_3 P_h\|_{\L(H)}.
	\end{align*}
	First, we look at $\Gamma_{1}$ in more detail. In the following, we apply the equality $(I+\tau A_{h,\ell})^{-1}= I-\tau A_{h,\ell}(I+\tau A_{h,\ell})^{-1}$, $\ell \in \{1,2\}$, Lemma~\ref{lem:Ah}~\ref{lem:disc2cont}, the fact that $A_{h,\ell}A^{-1}_{h} = A^{-1}_{h} A_{h,\ell}$ due to the commutivity and the semigroup properties \eqref{eq:ana_semi_h2} and \eqref{eq:def_fh} of $A_h$, $(I+\tau A_{h,2})^{-1}\tau A_{h,2} = I - (I+\tau A_{h,2})^{-1}$ and find	
	\begin{align*}
		&\|\Gamma_1 P_h \|_{\L(H)}\\
		&= \big\| \e^{-t_n  A_h} P_h - \big(I-\tau A_{h,2}(I+\tau A_{h,2})^{-1}\big) \e^{-t_{n-1}  A_h} \big( I-\tau A_{h,1}(I+\tau A_{h,1})^{-1} \big) P_h \big\|_{\L(H)}\\
		&\leq \big\| (\e^{-\tau  A_h}-I) \e^{-t_{n-1}  A_h} P_h \big \|_{\L(H)} + \tau \big\| (I+\tau A_{h,2})^{-1} A_{h,2} \e^{-t_{n-1}  A_h} P_h \big\|_{\L(H)}\\
		&\quad+ \tau \big\| \e^{-t_{n-1}  A_h} A_{h,1}(I+\tau A_{h,1})^{-1}P_h \big\|_{\L(H)}\\
		&\quad + \tau^2 \big\| (I+\tau A_{h,2})^{-1} A_{h,2} \e^{-t_{n-1}  A_h} A_{h,1} (I+\tau A_{h,1})^{-1}P_h \big\|_{\L(H)}\\
		&\leq \big\| (\e^{-\tau  A_h}-I)A_h^{-1}A_h\e^{-t_{n-1}  A_h}P_h \big\|_{\L(H)}
		+ \tau \big\| (I+\tau A_{h,2})^{-1}A_{h,2}A_{h}^{-1}A_h\e^{-t_{n-1}  A_h}P_h \big\|_{\L(H)}\\
		&\quad+ \tau \big\| \e^{-t_{n-1}  A_h}A_hA_h^{-1}A_{h,1}(I+\tau A_{h,1})^{-1}P_h \big\|_{\L(H)}\\
		&\quad+ \tau \big\| (I+\tau A_{h,2})^{-1} \tau A_{h,2}\e^{-t_{n-1}  A_h}A_hA_h^{-1}A_{h,1}(I+\tau A_{h,1})^{-1}P_h \big\|_{\L(H)}
		\leq C \frac{\tau}{t_{n-1}}.
	\end{align*}
	Thus, we have a suitable bound for $\Gamma_1$. 
	We now turn to $\Gamma_2$ and look at this summand in more detail 
	\begin{align*}
		\Gamma_{2}
		&=(I+\tau A_{h,2})^{-1} \big( \e^{-t_{n-1}  A_h} - \e^{-t_{n-1}\tilde{A}_h} \big) (I+\tau A_{h,1})^{-1}\\
		&=(I+\tau A_{h,2})^{-1} A_h A_{h}^{-1} \big( \e^{-t_{n-1}  A_h} - \e^{-t_{n-1}\tilde{A}_h} \big) \tilde{A}_{h}^{-1} \tilde{A}_h (I+\tau A_{h,1})^{-1}\\
		&=(I+\tau A_{h,2})^{-1}A_h  \Big[-\e^{-(t_{n-1} -s)A_h}A_h^{-1} \tilde{A}_h^{-1}\e^{-s\tilde{A}_h}\Big]_{s=0}^{s=t_{n-1}} \tilde{A}_h (I+\tau A_{h,1})^{-1}\\
		&= (I+\tau A_{h,2})^{-1} A_h \int_0^{t_{n-1}} \e^{-(t_{n-1}-s)A_h} \big( A_h^{-1}-\tilde{A}_h^{-1} \big) \e^{-s\tilde{A}_h} \diff{s} \ \tilde{A}_h (I+\tau A_{h,1})^{-1} \\
		&= (I+\tau A_{h,2})^{-1} A_h \int_0^{\frac{t_{n-1}}{2}} \e^{-(t_{n-1}-s)A_h} \big( A_h^{-1}-\tilde{A}_h^{-1} \big) \e^{-s\tilde{A}_h} \diff{s} \ \tilde{A}_h (I+\tau A_{h,1})^{-1} \\
		&\quad + (I+\tau A_{h,2})^{-1} A_h \int_{\frac{t_{n-1}}{2}}^{t_{n-1}} \e^{-(t_{n-1}-s)A_h} \big( A_h^{-1}-\tilde{A}_h^{-1} \big) \e^{-s\tilde{A}_h} \diff{s} \ \tilde{A}_h (I+\tau A_{h,1})^{-1}\\
		&=: \Gamma_{2,1} + \Gamma_{2,2}.
	\end{align*}
	For $\Gamma_{2,1}$, we use integration by parts to find
	\begin{align}\label{eq:proof_higher_conv_int_parts}
		\begin{split}
			\Gamma_{2,1}
			&= (I+\tau A_{h,2})^{-1} \int_0^{\frac{t_{n-1}}{2}} A_h^2 \e^{-(t_{n-1}-s)A_h}\big(A_h^{-1}-\tilde{A}_h^{-1} \big)\e^{-s\tilde{A}_h} \diff{s} \ (I+\tau A_{h,1})^{-1} \\
			&\quad- (I+\tau A_{h,2})^{-1} A_h \Big[ \e^{-(t_{n-1}-s)A_h} \big(A_h^{-1}-\tilde{A}_h^{-1}\big) \e^{-s\tilde{A}_h} \Big]_{s=0}^{s=\frac{t_{n-1}}{2}} (I+\tau A_{h,1})^{-1} \\
			&=\int_0^{\frac{t_{n-1}}{2}} A_h^2 \e^{-(t_{n-1}-s)A_h} (I+\tau A_{h,2})^{-1} \big(A_h^{-1}-\tilde{A}_h^{-1} \big)(I+\tau A_{h,1})^{-1} \e^{-s\tilde{A}_h} \diff{s}\\
			&\quad- A_h \e^{-\frac{t_{n-1}}{2}A_h} (I+\tau A_{h,2})^{-1}\big(A_h^{-1}-\tilde{A}_h^{-1} \big)  (I+\tau A_{h,1})^{-1} \e^{-\frac{t_{n-1}}{2}\tilde{A}_h}\\
			&\quad + A_h \e^{-t_{n-1}A_h} (I+\tau A_{h,2})^{-1} \big(A_h^{-1}-\tilde{A}_h^{-1} \big) (I+\tau A_{h,1})^{-1}.
		\end{split}
	\end{align}
	where we inserted that the operators $(I + \tau A_{h,2})^{-1}$ and $A_h$ as well as $\tilde{A}_h$ and $(I + \tau A_{h,1})^{-1}$ commute because of the commutativity of $A_{h,1}$ and $A_{h,2}$.
	Next, we state a bound for $(I+\tau A_{h,2})^{-1} (A_h^{-1}-\tilde{A}_h^{-1}) (I + \tau A_{h,1})^{-1}$. First, we recall the definition $\tilde{A}_h = (I+\tau A_{h,1})^{-1}A_h(I+\tau A_{h,2})^{-1}$ and find
	\begin{align*}
		&\| (I + \tau A_{h,2})^{-1} (A_h^{-1} - \tilde{A}_h^{-1}) (I + \tau A_{h,1})^{-1} P_h \|_{\L(H)}\\
		&= \| (I + \tau A_{h,1})^{-1} A_h^{-1} (I + \tau A_{h,1})^{-1} P_h - A_h^{-1} P_h \|_{\L(H)}\\
		&\leq \big\| (I + \tau A_{h,1})^{-1} A_h^{-1} \big((I + \tau A_{h,1})^{-1} - I \big) P_h\big\|_{\L(H)}
		+ \big\| \big((I + \tau A_{h,1})^{-1} - I \big) A_h^{-1} P_h \big\|_{\L(H)}\\
		&= \tau \big\| (I + \tau A_{h,1})^{-1} A_h^{-1} A_{h,1} (I + \tau A_{h,1})^{-1} P_h\big\|_{\L(H)}
		+ \tau \big\| (I + \tau A_{h,1})^{-1} A_{h,1} A_h^{-1} P_h \big\|_{\L(H)}\\
		&\leq C \tau \big\| A_{h,1} A_h^{-1}  P_h\big\|_{\L(H)} \leq C \tau,
	\end{align*}
	where we inserted that the operators $(I + \tau A_{h,2})^{-1}$ and $A_h$ as well as $\tilde{A}_h$ and $(I + \tau A_{h,1})^{-1}$ commute because of the commutativity of $A_{h,1}$ and $A_{h,2}$.
	Next, we state a bound for $(I+\tau A_{h,2})^{-1} (A_h^{-1}-\tilde{A}_h^{-1}) (I + \tau A_{h,1})^{-1}$. First, we recall the definition $\tilde{A}_h = (I+\tau A_{h,1})^{-1}A_h(I+\tau A_{h,2})^{-1}$ and find
	\begin{align*}
		&\| (I + \tau A_{h,2})^{-1} (A_h^{-1} - \tilde{A}_h^{-1}) (I + \tau A_{h,1})^{-1} P_h \|_{\L(H)}\\
		&= \| (I + \tau A_{h,1})^{-1} A_h^{-1} (I + \tau A_{h,1})^{-1} P_h - A_h^{-1} P_h \|_{\L(H)}\\
		&\leq \big\| (I + \tau A_{h,1})^{-1} A_h^{-1} \big((I + \tau A_{h,1})^{-1} - I \big) P_h\big\|_{\L(H)}
		+ \big\| \big((I + \tau A_{h,1})^{-1} - I \big) A_h^{-1} P_h \big\|_{\L(H)}\\
		&= \tau \big\| (I + \tau A_{h,1})^{-1} A_h^{-1} A_{h,1} (I + \tau A_{h,1})^{-1} P_h\big\|_{\L(H)}
		+ \tau \big\| (I + \tau A_{h,1})^{-1} A_{h,1} A_h^{-1} P_h \big\|_{\L(H)}\\
		&\leq C \tau \big\| A_{h,1} A_h^{-1}  P_h\big\|_{\L(H)} \leq C \tau,
	\end{align*}
	where we inserted $(I + \tau A_{h,1})^{-1} - I = \tau A_{h,1} (I + \tau A_{h,1})^{-1}$, the commutativity of $A_{h,1}$ and $A_{h}$ as well as Lemma~\ref{lem:Ah}~\ref{lem:disc2cont}.
	Then we can apply the semigroup bound \eqref{eq:ana_semi_h1} to find
	\begin{align*}
		&\| \Gamma_{2,1} P_h \|_{\L(H)}\\
		&\leq \int_0^{\frac{t_{n-1}}{2}} \big\| A_h^2 \e^{-(t_{n-1}-s)A_h} P_h\big\|_{\L(H)} \big\| (I+\tau A_{h,2})^{-1} \big(A_h^{-1}-\tilde{A}_h^{-1} \big)(I+\tau A_{h,1})^{-1} P_h\big\|_{\L(H)} \\
		&\hspace{11cm} \times \big\| \e^{-s\tilde{A}_h} P_h\big\|_{\L(H)} \diff{s}\\
		&\quad+ \big\| A_h \e^{-\frac{t_{n-1}}{2}A_h} P_h\big\|_{\L(H)} 
		\big\|(I+\tau A_{h,2})^{-1}\big(A_h^{-1}-\tilde{A}_h^{-1} \big)  (I+\tau A_{h,1})^{-1} P_h\big\|_{\L(H)} 
		\big\| \e^{-\frac{t_{n-1}}{2}\tilde{A}_h} P_h\big\|_{\L(H)}\\
		&\quad + \big\| A_h \e^{-t_{n-1}A_h} P_h \big\|_{\L(H)} \big\| (I+\tau A_{h,2})^{-1} \big(A_h^{-1}-\tilde{A}_h^{-1} \big) (I+\tau A_{h,1})^{-1} P_h\big\|_{\L(H)} \\
		&\leq C \tau \int_0^{\frac{t_{n-1}}{2}} (t_{n-1}-s)^{-2} \diff{s} 
		+ C\frac{\tau }{t_{n-1}} + C \frac{\tau}{t_{n-1}} 
		\leq C\frac{\tau}{t_{n-1}}.
	\end{align*}
	We can use a similar argument to show that $\| \Gamma_{2,2} P_h \|_{\L(H)} \leq C\frac{\tau}{t_{n-1}}$. The difference is that in \eqref{eq:proof_higher_conv_int_parts} we use integration by parts but change the role of the two functions. Additionally, we need to apply \eqref{eq:ana_semi_h1} for $\tilde{A}_h$. This is possible since $A_{h,1}$ and $A_{h,2}$ are commutable and self-adjoint, which implies that $\tilde{A}_h$ is self-adjoint and therefore in particular sectorial. This shows that $\| \Gamma_{2} P_h \|_{\L(H)} \leq C\frac{\tau}{t_{n-1}}$.
	
	It remains to bound $\| \Gamma_{3}P_h\|_{\L(H)}$. We begin to rewrite $\tilde{S}_{h,\tau}$ as follows
	\begin{align*}
		\tilde{S}_{h,\tau}
		&= (I+\tau A_{h,1})^{-1} (I+\tau^2A_{h,1}A_{h,2} ) (I+\tau A_{h,2})^{-1}\\
		&= (I+\tau A_{h,1})^{-1} \big( (I + \tau A_{h,1})(I + \tau A_{h,2}) - \tau A_h \big) (I+\tau A_{h,2})^{-1}=I-\tau\tilde{A}_h,
	\end{align*}
	where we inserted the abbreviation $\tilde{A}_h = (I+\tau A_{h,1})^{-1}A_h(I+\tau A_{h,2})^{-1}$ in the last step.
	Then we can bound $\| \Gamma_{3}P_h\|_{\L(H)}$ as follows
	\begin{equation*}
		\| \Gamma_{3}P_h\|_{\L(H)}
		\leq \big\| \e^{-t_{n-1}\tilde{A}_h} P_h - \tilde{S}_{h,\tau}^{n-1}  P_h \big\|_{\L(H)}
		= \big\| \e^{-(n-1)(I - (I-\tau \tilde{A}_h ) ) }P_h - \big(I-\tau\tilde{A}_h \big)^{n-1} P_h \big\|_{\L(H)}.
	\end{equation*}
	In the case that $A_{h,1}$ and $A_{h,2}$ are commutable and self-adjoint, we know that $\tilde{S}_{h,\tau} = I-\tau\tilde{A}_h$ is self-adjoint. Additionally, the eigenvalues of the operator are between $0$ and $1$, which can be deduced from the proof of Lemma~\ref{lem:Sne}. Thus, applying the functional calculus theorem \cite[Theorem~2.7.11(b)]{Conway2007}, it will be sufficient in the following to find a bound for  $|\e^{-(n-1)(1-\lambda)}-\lambda^{n-1} |$ for $\lambda \in [0,1]$. To this end, we denote $g(\lambda):= \e^{-(n-1)(1-\lambda)}-\lambda^{n-1}$ and will find $g$'s extrema for $\lambda \in [0,1]$. We note that $g(0)=e^{-(n-1)}$ and $g(1)=0$ are finite and will look for further local extrema of $g$ in the open interval $(0,1)$. The derivative of $g$ is given by $g'(\lambda)= (n-1) \e^{-(n-1)(1-\lambda)} - (n-1)\lambda^{n-2}$. Setting this to zero implies that every candidate $\lambda_{*}$ for a local extremum fulfills $\e^{-(n-1)(1-\lambda_{*})} = \lambda_{*}^{n-2}$. First, we assume that such a value $\lambda_{*}$ is a maximum. Then we find
	\newpage
	\begin{align*}
		g(\lambda_{*}) 
		&= \e^{-(n-1)(1-\lambda_{*})}-\lambda_{*}^{n-1}
		= \e^{-(n-1)(1-\lambda_{*})} - \lambda_{*} \e^{-(n-1)(1-\lambda_{*})}\\
		&\leq \max_{\lambda \in [0,1]} \e^{-(n-1)(1-\lambda)} - \lambda \e^{-(n-1)(1-\lambda)}
		= \max_{\lambda \in [0,1]} \frac{1}{n-1} (n-1) (1-\lambda) \e^{-(n-1)(1-\lambda)}
		\leq \frac{\tau}{t_{n-1}}
	\end{align*}
	as $x \e^{-x} \leq 1$ for all $x \in \R_0^+$.
	In case $\lambda_{*}$ is a minimum, we argue in a similar way using that $\min_{\lambda \in [0,1]} \frac{1}{n-1} (n-1) (1-\lambda) \e^{-(n-1)(1-\lambda)} \geq 0$. This shows that $|\e^{-(n-1)(1-\lambda)}-\lambda^{n-1} | \leq \frac{1}{n-1}$ is fulfilled for all $\lambda \in [0,1]$. Altogether, we have shown that $\Gamma_{3} \leq \frac{\tau}{t_{n-1}}$.
	\end{proof}

\section*{Acknowledgments}
The computations were enabled by resources provided by LUNARC, The Centre for Scientific and Technical Computing at Lund University. 
The authors would like to thank Robert Kl\"ofkorn for valuable help and advice when producing our code.
Additionally, the authors would also like to thank the referees for constructive and insightful feedback.

\section*{Funding}
The first and the third author were supported in part by the Swedish Research Council under the grant 2023-03930, eSSENCE: The e-Science Collaboration, and the Crafoord foundation and the second author by the Swedish Research Council under the grant 2023–04862.

\bibliographystyle{amsplain}
\bibliography{lit.bib}

\end{document}